\documentclass[12pt,a4paper]{article}

\usepackage{fancyhdr}
\usepackage{titlesec}
\usepackage{titletoc}
\usepackage{graphicx}
\usepackage{stmaryrd}
\usepackage{amsmath}

\usepackage{varioref} 
\usepackage[unicode,bookmarks,colorlinks=true,citecolor=black,linkcolor=black
]{hyperref}
\usepackage[all]{hypcap}

\usepackage{amstext}
\usepackage{amsbsy}
\usepackage{amssymb}
\usepackage{theorem}
\usepackage{dsfont}
\usepackage{times}
\usepackage[T1]{fontenc}
\usepackage{eucal}
\usepackage{mathrsfs}
\usepackage{url}
\usepackage[usenames]{xcolor}
\newcommand{\ds}{\displaystyle}

\renewcommand{\mathbb}{\mathds}

\renewcommand{\C}[1]{\mathcal{#1}}

\newcommand{\ov}[1]{\overline{#1}}
\newcommand{\wt}[1]{\widetilde{#1}}
\newcommand{\wh}[1]{\widehat{#1}}
\newcommand{\B}[1]{\mathds{#1}}
\newcommand{\oB}[1]{\ov{\mathds{#1}}}
\newcommand{\mb}[1]{\mathbf{#1}}

\newcommand{\ud}{\mathrm{d}}

\DeclareMathOperator{\Tr}{{\mathbf Tr}}
\DeclareMathOperator{\Gr}{{\mathbf Gr}}
\DeclareMathOperator{\Id}{\mathbf{I}}
\DeclareMathOperator{\diag}{\mathbf{diag}}

\DeclareMathOperator{\Var}{\mathbf{Var}}

\DeclareMathOperator{\Span}{\mathbf{span}}
\DeclareMathOperator{\IM}{\mathbf{Im}}
\DeclareMathOperator{\Ker}{\mathbf{Ker}}

\DeclareMathOperator{\rank}{\mathbf{rank}}
\DeclareMathOperator{\bO}{\mathbf{O}}
\newtheorem{thm}{Theorem}[section]
\newtheorem{prop}[thm]{Proposition}

\newtheorem{lemma}[thm]{Lemma}
\newtheorem{cor}[thm]{Corollary}

{\theorembodyfont{\rm} 

}

\theoremstyle{remark}

\theoremheaderfont{\sc}
\newenvironment{proof}{{\sc Proof.}}{\ $\square$}
\titleformat{\section}[block]{\sc\center}{\thetitle.}{5pt}{}[]
\titlespacing{\section}{0pt}{*4.5}{*3}
\titleformat{\subsection}[runin]{\sc}{\thetitle.}{5pt}{}[.]
\titlespacing{\subsection}{0pt}{*3}{*2}
\titleformat{\subsubsection}[runin]{\it}{\thetitle.}{5pt}{}[.]
\titlespacing{\subsubsection}{0pt}{*2}{*2}
\contentsmargin[30pt]{30pt}
\titlecontents{section}[15pt]{\vspace{6pt}\sc}
{\contentslabel[\thecontentslabel.]{15pt}}{\hspace*{-15pt}}
{\titlerule*[1pc]{.}\contentspage[\rm\thecontentspage]}
\titlecontents{subsection}[37pt]{\vspace{3pt}\small\sc}
{\contentslabel[\thecontentslabel.]{22pt}}{\hspace*{-22pt}}
{\titlerule*[1pc]{.}\contentspage[\normalsize\rm\thecontentspage]}
\titlecontents{subsubsection}[69pt]{\it}
{\contentslabel[\thecontentslabel.]{32pt}}{\hspace*{-32pt}}
{~\titlerule*[1pc]{.}\contentspage[\rm\thecontentspage]}

\newcommand{\thmref}[1]{\ref{#1} on page \pageref{#1}}
\newcommand{\myeq}[1]{{\rm (\ref{#1})} on page \pageref{#1}}
\allowdisplaybreaks

\usepackage{Sweave}
\begin{document}
\renewcommand{\sectionmark}[1]{\markboth{\thesection\ #1}{}}
\renewcommand{\subsectionmark}[1]{\markright{\thesubsection\ #1}}
\fancyhf{}
\fancyhead[RE]{\small\sc\nouppercase{\leftmark}}
\fancyhead[LO]{\small\sc\nouppercase{\rightmark}}
\fancyhead[LE,RO]{\thepage}
\fancyfoot[RO,LE]{\small\sc Olivier Catoni}
\fancyfoot[LO,RE]{\small\sc\today}
\renewcommand{\footruleskip}{1pt}
\renewcommand{\footrulewidth}{0.4pt}
\newcommand{\mypoint}{\makebox[1ex][r]{.\:\hspace*{1ex}}}
\addtolength{\footskip}{11pt}
\pagestyle{plain}
\renewcommand{\thefootnote}{\fnsymbol{footnote}}
\begin{center}
{\large \sc PAC-Bayesian bounds for the Gram matrix 
and least squares regression with a random design. 
}\\[12pt]
{\sc Olivier Catoni
\footnotetext{CNRS --  CREST, UMR 9194,
3 avenue Pierre Larousse, 92240 Malakoff, France. 
e-mail: olivier.catoni@ensae.fr}\\[12pt]
{\small \it \today }}\\[12pt]
\begin{minipage}{0.8\textwidth}
{\small {\sc Abstract: } The topics dicussed in this paper 
take their origin in the estimation of the Gram matrix 
$\B{E} \bigl( X X^{\top} \bigr)$ of a random vector $X 
\in \B{R}^d$ from a sample made of $n$ independent 
copies $X_1, \dots, X_n$ of $X$. They comprise the estimation 
of the covariance matrix 
and 
the study of least squares regression with a random 
design. We propose four types of results, based on 
non-asymptotic PAC-Bayesian 
generalization bounds: a new robust estimator 
of the Gram matrix and of the covariance matrix, 
new results on the empirical Gram matrix 
$\frac{1}{n} \sum_{i=1}^n X_i X_i^{\top}$, 
new robust least squares estimators and new 
results on the ordinary least squares estimator, 
including its exact rate of convergence under 
polynomial moment assumptions.\\[1ex]
}
\end{minipage} 
\end{center}

\tableofcontents

\pagestyle{fancy}

\section*{Introduction}

Let us consider $n$ independent copies $(X_1, \dots, X_n)$ 
of a random vector $X$ whose probability distribution 
$\B{P}$ belongs 
to $\C{M}_+^1 \bigl( 
\B{R}^d \bigr)$, the set of probability measures on $\B{R}^d$ (equiped 
with the Borel $\sigma$-algebra).

The topics discussed in this paper take their origin 
in the estimation of  
the Gram matrix $G = \B{E} \bigl( X X^{\top} \bigr)$ and 
comprise the estimation of the covariance matrix 
$G - \B{E}(X) \B{E}(X)^{\top}$ and 
least squares regression with a random design.  

We propose four things. A new robust estimator of $G$, 
new results on the empirical estimator 
\[
\ov{G} = \frac{1}{n} \sum_{i=1}^n X_i X_i^{\top}, 
\]
new robust least squares estimators and new 
results on the ordinary least squares estimator
\[ 
\wh{\theta} \in \arg \min_{\theta \in \B{R}^d} 
\sum_{i=1}^n \bigl( Y_i - \langle \theta, X_i \rangle \bigr)^2, 
\] 
where $(X_1, Y_1), \dots, (X_n, Y_n)$ are $n$ independent 
copies of a couple $(X, Y) \in \B{R}^d \times \B{R}$ 
of random variables. In particular we give the exact 
rate of convergence of $R(\wh{\theta}) - \inf_{\theta \in \B{R}^d} R(\theta)$, 
where $R(\theta) = \B{E} \bigl[ \bigl( Y - \langle \theta, X \rangle
\bigr)^2 \bigr]$. 

Technically, our approach is based on the estimation of the   
quadratic form 
$$
N(\theta) \overset{\rm def}{=} \B{E} \bigl( \langle \theta, X \rangle^2 
\bigr) = \theta^{\top} G \theta, 
$$
that computes the energy in direction $\theta$. 
It can also be seen as the square of the norm defined by the Gram 
matrix $G$. Recovering $G$ from $N$ can be done through the
polarization identity 
\begin{equation}
\label{eq:1}
\xi^{\top} G \theta = \frac{1}{4} \bigl[ N(\xi + \theta) - 
N (\xi - \theta) \bigr] = \frac{1}{2} \bigl[ 
N(\xi + \theta) - N(\xi) - N(\theta) \bigr], 
\end{equation} 
that gives as a special case 
\[
G_{i,j} = e_i^{\top} G e_j = \frac{1}{4} \bigl[ N(e_i + e_j) - N(e_i - e_j) 
\bigr] = \frac{1}{2} \bigl[ N(e_i + e_j) - N(e_i) - N(e_j) \bigr], 
\]
where $e_i$ are the vectors of the canonical basis of $\B{R}^d$. 

Our purpose is to define and study robust estimators, 
whose estimation error 
$\lvert N(\theta) - \wh{N}(\theta) \rvert$ 
can be bounded with a probability close to one jointly 
for all values of $\theta \in \B{R}^d$ 
under weak polynomial moment assumptions. 
More specifically, we make a $\theta$-dependent assumption on 
the variance of $\langle \theta, X \rangle^2$, 
that takes the form 
\[ 
\sup \Bigl\{ \B{E} \bigl( \langle \theta, X \rangle^4 \bigr), 
\theta \in \B{R}^d, 
\B{E} \bigl( \langle \theta, X \rangle^2 \bigr) \leq 1 \Bigr\} 
 \leq \kappa, 
\] 
implying that 
\[ 
\Var \bigl( \langle \theta, X \rangle^2 \bigr) \leq \bigl( \kappa - 1 \bigr) 
 \B{E} \bigl( \langle \theta, X \rangle^2 \bigr)^2. 
\] 
This kurtosis coefficient measures the heaviness of the 
tail of the distribution of $\langle \theta, X \rangle^2$. 
To give a point of comparison, when $X$ is a Gaussian random 
vector, whatever its Gram matrix maybe, the above assumption 
is satisfied for $\kappa = 3$. 

Based on this assumption, we define an estimator $\wh{N}(\theta)$
and prove a $\theta$-dependent uniform bound 
on the estimation error $\lvert N(\theta) - \wh{N}(\theta) \rvert$.  
More precisely, instead of bounding merely 
\[ 
\sup \Bigl\{ \bigl\lvert N(\theta) - \wh{N}(\theta) \bigr\rvert \, 
: \,  \theta \in \B{R}^d, 
\lVert \theta \rVert \leq 1 \Bigr\}, 
\] 
we bound with a probability close to one 
\[ 
\sup_{\theta \in \B{R}^d} \; \biggl\lvert \frac{N(\theta)}{\wh{N}(\theta)} - 1 
\biggr\rvert,
\] 
with the convention that $ 0/0 = 1$ and $z / 0 = + \infty$ 
when $z > 0$. Remark that this type of bound implies that it is possible 
to estimate exactly the null space $\Ker(G)$ with a probability close to one.   

This new estimator, built on the same principles 
as the robust mean estimator of \cite{Cat10}, is interesting in at least 
two ways. 
First, it can be actually used to estimate
$G$, with increased performances in
some heavy tail situations and with 
mathematical guaranties taking the 
form of non-asymptotic convergence 
bounds under weak hypotheses. 

Second, it can be compared with the empirical 
estimate $\ov{G}$ of $G$ and used as a mathematical tool 
to prove new generalization bounds for $\ov{G}$, 
under various hypotheses, including 
polynomial moment assumptions.  

The estimation of the Gram matrix has 
many interesting possible applications. The most obvious 
one is to derive robust alternatives to the 
classical principal component analysis based on $\ov{G}$. 
In this paper however, we will rather focus on 
least squares regression with a random design. 
We propose and study new stable 
least squares regression estimators, and also 
provide new bounds for the 
ordinaray least squares estimator. 
We already studied robust least squares 
regression in \cite{AuCat10a}, but this 
time, as the reader will see, we come 
up with simplified estimators and tighter 
results. We also come up with interesting 
new results about the ordinary least squares
estimator. In particular we give its exact
rate of convergence under polynomial moment assumptions, 
including the case when the noise, defined 
as $Y - \langle \theta_*, X \rangle$, 
where $\theta_* \in \arg \min_{\theta \in \B{R}^d} 
\B{E} \bigl[ \bigl( Y - \langle \theta, X \rangle \bigr)^2 \bigr]$, 
is not independent from $X$. When independence
is not assumed, this exact rate is quite 
interestingly not equal to 
$\B{E} \bigl[ \bigl( Y - \langle \theta_*, X \rangle
\bigr)^2 \bigr] d / n$ and can depart from it by an 
arbitrarily large or small factor, as we show on examples.   

Let us close this introduction with some precisions on our 
use of the big $\bO$ notation. In this paper we will always 
prove precise and fully explicit non-asymptotic bounds. 
Nevertheless, as these bounds are sometimes difficult to read, 
we will use the big $\bO$ notation to give 
a representation of their order of magnitude. 
When we write $A = \bO(B)$, where
$A$ and $B$ are two expressions depending on parameters 
of the problem, we mean that there is a numerical constant 
$c$ such that $A \leq c B$. 
When we write 
$A = \bO_{n \rightarrow \infty} (B)$, 
where $n$ is the sample size,
we mean 
that there is a numerical constant $c$ such that 
$\lim \sup_{n \rightarrow \infty} A/B \leq c$.  
Remark that $A = \bO \bigl( B \bigr)$ implies 
that $A = \bO_{n \rightarrow \infty} \bigl( B \bigr)$,  
but that the reverse implication is false in general. 
The notation $A = \bO_{n \rightarrow \infty} \bigl( B \bigr)$ 
means that $B$ bounds the order of magnitude of the first 
order term of $A$ seen as a function of the sample size, 
whereas $A = \bO \bigl( B \bigr)$ means that $B$ bounds 
the order of magnitude of $A$ in all circumstances. 

\section{A robust Gram matrix estimate}

\subsection{Definition of a new estimator} 

Following the same route as in \cite{Cat10} in a more elaborate setting, 
we define some $M$-estimator of $N(\theta)$, the energy 
in direction $\theta$, and derive for it non-asymptotic
deviation bounds that are uniform with respect to $\theta$.  
To do this, we need to introduce the influence function 
$\psi : \B{R} \rightarrow \B{R}$ defined as\\[4ex]
\mbox{} \hfill \framebox{
\begin{minipage}{0.9\textwidth}
\noindent \mbox{} \hfill 
\raisebox{-4ex}[45ex][0ex]{
\includegraphics[width=0.8\textwidth]{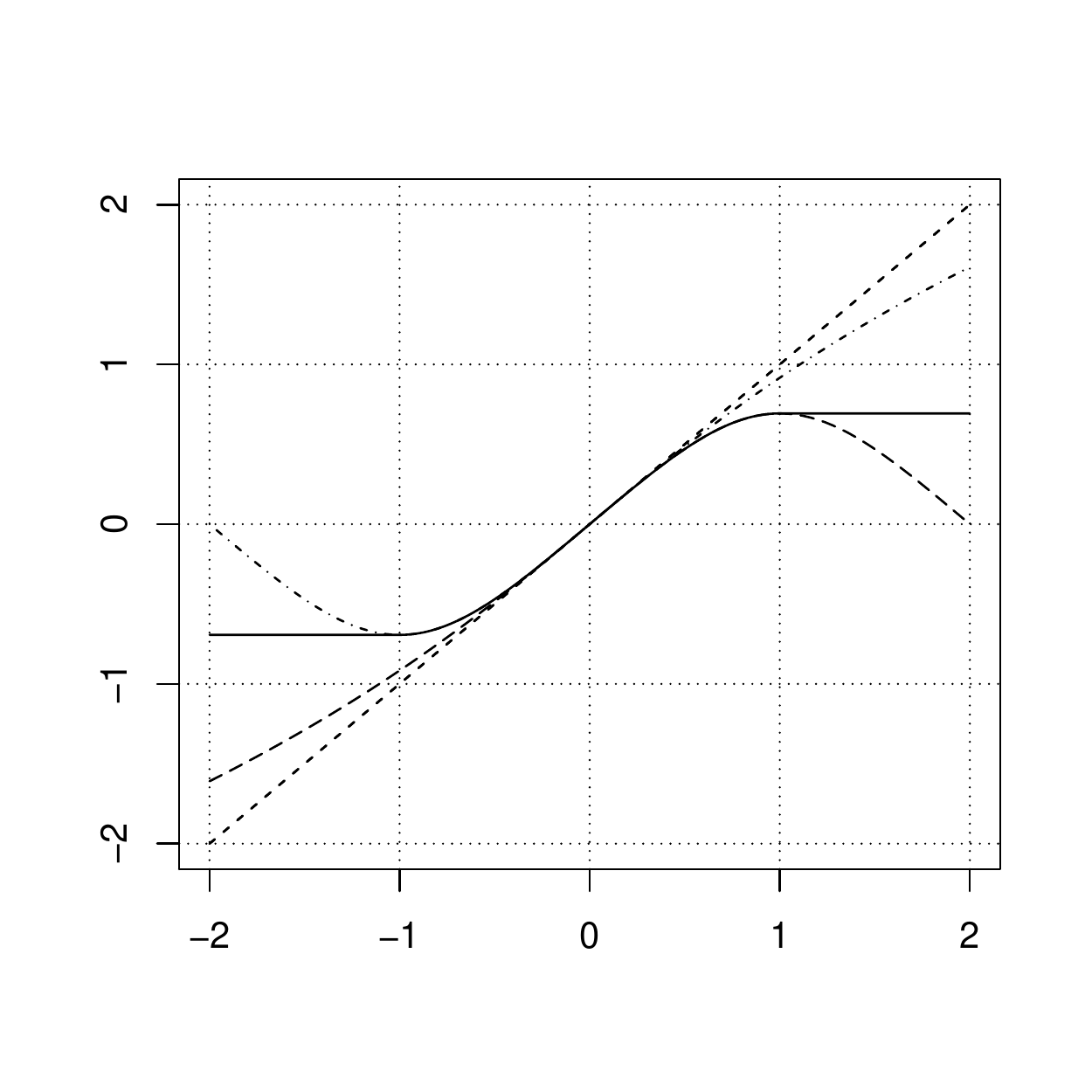}
}
\hfill \hfill \hfill \mbox{}\\[-3ex]
\begin{small}
\mbox{} \hfill $x \mapsto \psi(x)$, compared with $x \mapsto x$ \hfill \mbox{}\\
\mbox{} \hfill $x \mapsto \log \bigl(1 + x + x^2/2 \bigr)$, and $x \mapsto 
- \log \bigl( 1 - x + x^2/2 \bigr) $ \hfill \mbox{}\\
\end{small}
\end{minipage}} \hfill \mbox{}\\[1ex]

$$
\psi(x) = \begin{cases}
\log(2), & x \geq 1,\\
- \log \bigl( 1 - x + x^2 / 2 \bigr), & 0 \leq x \leq 1,\\
- \psi(-x), & x \leq 0.
\end{cases}
$$
This is a symmetric, non-decreasing, and bounded function, 
satisfying
$$
- \log \bigl(1 - x + x^2/2 \bigr) \leq \psi(x) \leq 
\log \bigl( 1 + x + x^2 / 2 \bigr), \qquad x \in \B{R},
$$
as can be seen from the identity 
$$
\bigl( 1 - x + x^2/2 \bigr)^{-1} = \frac{ 1 + x + x^2/2 }{1 + x^4/4} 
\leq 1 + x + x^2/2, \qquad x \in \B{R}.
$$

Let us consider for any positive parameter $\lambda \in \B{R}_+$ the empirical sample distribution 
$$
\oB{P} = \frac{1}{n} \sum_{i=1}^n \delta_{X_i}
$$ 
and the empirical criterion  
$$
r_{\lambda} (\theta) = \lambda^{-1} \int \psi \Bigl\{ \lambda \bigl[ 
\langle \theta, x \rangle^2 - 1 \bigr] \Bigr\} \, \ud \oB{P}(x). 
$$
This empirical criterion depends on a positive scale 
parameter $\lambda \in ]0, \infty[$, whose value will 
be set later on. 

The centering of the influence function 
is done by adjusting the norm $\lVert \theta \rVert$, 
as we are going to explain now. 
To achieve centering, we introduce the multiplicative factor
$\wh{\alpha}(\theta)$ defined as
$$
\wh{\alpha}(\theta) = \sup \Bigl\{ \alpha \in \B{R}_+ \, : \, r_{\lambda}(\alpha \theta) 
\leq 0 \Bigr\}. 
$$ 
As $r_{\lambda}(0) = - \lambda^{-1} \psi(\lambda) < 0$, we see that 
$\wh{\alpha}(\theta) \in \B{R}_+ \cup \{ + \infty \}$, for any  
$\theta \in \B{R}^d$. Moreover, since $\alpha \mapsto 
r_{\lambda}(\alpha \, \theta)$ 
is continuous, $r_{\lambda} \bigl( \wh{\alpha} \, \theta \bigr) = 0$ 
as soon as $\wh{\alpha}(\theta) < + \infty$. 
Considering that $\psi$ is close to the identity on a neighbourhood 
of $0$ and that the empirical measure $\oB{P}$ is typically close 
to $\B{P}$, we may hope that, for suitable values of $\lambda$, 
 $r_{\lambda}(\theta) \simeq N(\theta) - 1$ with large probability. 
If this is the case, and if moreover, $\wh{\alpha}(\theta) < \infty$,
then 
$$
0 = r_{\lambda} \bigl[ 
\wh{\alpha}(\theta) \theta \bigr] \simeq N \bigl[ 
\wh{\alpha}(\theta) \theta \bigr] - 1 = 
\wh{\alpha}(\theta)^2 N(\theta) - 1. 
$$
This is an incitation to define a new estimator 
of $N(\theta)$ as 
\begin{equation}
\label{eq1.2}
\wh{N}_{\lambda}(\theta) = \wh{\alpha}(\theta)^{-2}.
\end{equation}
To make things easier to understand, we can also 
define this estimator without introducing intermediate 
steps as 
\[ 
\wh{N}_{\lambda}(\theta) = \inf \biggl\{ \rho \in \B{R}_+^* \, : \, 
\sum_{i=1}^n \psi \Bigl[ \lambda \Bigl( 
\rho^{-1} \langle \theta, X_i \rangle^2 - 1 \bigr) \Bigr]
\leq 0 \biggl\}. 
\] 

\subsection{Generalization bounds}
We prove in Appendix \ref{appA} a detailed proposition, 
Proposition \vref{prop1.3}, whose main conclusions 
can be summarized as follows. 
\begin{prop}
\label{prop1.2.3} 
Let us assume that for some known constant $\kappa$ 
$$
\sup 
\Bigl\{ \B{E} \bigl( \langle \theta, X \rangle^4 \bigr) 
\, : \,  
\theta \in \B{R}^d, \B{E} \bigl( \langle \theta, X \rangle^2 \bigr) \leq 1 
\Bigr\}  
\leq \kappa < \infty.
$$
\begin{align*}
\text{Choose }
\lambda & = \sqrt{\frac{2}{(\kappa -1) n} \bigl[ \log(\epsilon^{-1}) 
+ 0.73 \, d \bigr]}, \\  
\text{and consider } 
\mu & = \sqrt{\frac{2 (\kappa -1)}{n} \bigl[ \log(\epsilon^{-1} ) 
+ 0.73 \, d \bigr]} + 6.81 \sqrt{ \frac{2 \kappa\, d}{n}}. 
\end{align*}
For any confidence parameter $\epsilon > 0$, and any sample size $n$ such that 
\begin{multline}
\label{eq2.2}
n > \Biggl[ 20 \sqrt{\kappa d} + \biggl( \frac{5}{2} + \frac{1}{2 (\kappa-1)} 
\biggr) \sqrt{2 (\kappa -1) \bigl[ 
\log(\epsilon^{-1}) + 0.73 d \bigr] } \Biggr]^2 \\ = 
\bO \Bigl( \kappa \big[ d + \log(\epsilon^{-1}) \bigr] \Bigr),
\end{multline}
with probability at least $1 - 2 \epsilon$, for any $\theta \in \B{R}^d$, 
\begin{equation}
\label{eq:3}
\left\lvert \frac{N(\theta)}{\wh{N}_\lambda(\theta)} - 1 \right\rvert \leq 
\frac{\mu}{1 - 2 \mu} = \bO \Biggl( 
\sqrt{\frac{\kappa \bigl[ d + \log(\epsilon^{-1})\bigr]}{n}} \; \Biggr), 
\end{equation}
where it should be understood that $z/0 = 1$ if $z = 0$ and 
$+\infty$ if $z>0$, and where inequality \eqref{eq:3}  
can be decomposed into two inequalities (by removing the absolute value) 
that hold each with probability at least $1 - \epsilon$. 
\end{prop}
The conclusion of this proposition is that the accuracy of the 
estimation of $N(\theta)$ by $\wh{N}_{\lambda}(\theta)$ 
 is of order $N(\theta) \sqrt{\kappa d / n}$ with a subgaussian 
tail up to very high (exponential with $n$) confidence levels. 
Indeed, equation \eqref{eq:3} can also be written as  
\[ 
\Bigl\lvert N(\theta) - \wh{N}_{\lambda}(\theta) \Bigr\rvert \leq 
\frac{\mu \wh{N}_{\lambda}(\theta)}{1 - 2 \mu} \leq \frac{\mu N(\theta)}{
1 - 3 \mu}.
\]
This fairly strong result is made possible by the assumption that 
$\kappa$ is bounded. 

As already mentioned in the introduction, when $X$ is a multidimensional 
Gaussian random variable, 
we can take $\kappa = 3$. Assuming that $\kappa < \infty$
requires that 
the behaviour of the fourth moment of the distribution of 
$\langle \theta, X \rangle$ is not too far from the Gaussian case. 
This remains nonetheless 
a much weaker assumption than the existence of exponential moments. 

Let us mention that instead of getting a generalization bound 
or order 
\[ 
\bO \Biggl( \sqrt{\frac{ \kappa \bigl[ d + \log(\epsilon^{-1})
\bigr]}{n}} \; \Biggr),
\] 
depending on the dimension $d$ of the 
ambient space (or more accurately on the rank of the Gram matrix $G$), it is also 
possible to obtain dimension free bounds where the dimension 
$d$ is replaced with the $\theta$ dependent term 
$\Tr(G) \lVert \theta \rVert^2 / N(\theta)$ 
(to see that it is indeed some substitute for the dimension, 
we can remark that this $\theta$ dependent factor 
is uniformly equal to the dimension $d$ in the case 
when $G = \mb{I}$, because then $\Tr(G) = d$ and $N(\theta) = 
\lVert \theta \rVert^2$). For such results, 
we refer to the works of our student Ilaria Giulini 
\cite{Giulini1,Giulini2}. 

The estimator $\wh{N}_{\lambda}(\theta)$ is not a quadratic form 
in $\theta$, and therefore does not define an estimate of the 
Gram matrix $G$ in an obvious way. Nevertheless, we show in appendix 
that it is possible to deduce from $\wh{N}_{\lambda}$ a robust estimate 
$\wh{G}$ of $G$.  

\begin{prop}
\label{prop:1.2}
There exists an estimator $\wh{G}$ of the Gram matrix $G$, 
deduced from $\wh{N}_{\lambda}$ 
as explained in Appendix \ref{app:B}, such that under the same hypotheses
as in Proposition \ref{prop1.2.3}, for any 
confidence parameter $\epsilon > 0$ and any sample size $n$ 
satisfying equation \eqref{eq2.2}, with probability at least $1 - 2 \epsilon$, 
for any $\theta \in \B{R}^d$, 
\begin{equation}
\label{eq:05}
\biggl\lvert \frac{ \theta^{\top} \wh{G} \theta}{\theta^{\top} G \theta} - 1 
\biggr\rvert \leq \frac{2 \mu}{1 - 4 \mu} = \bO \Biggl( 
\sqrt{\frac{ \kappa \bigl[ d + \log(\epsilon^{-1}) \bigr]}{n}} \; \Biggr), 
\end{equation}
where $\mu$ is defined as in Proposition \ref{prop1.2.3}.
\end{prop}
\begin{proof}
This proposition is a simplified formulation of the end of Corollary 
\vref{cor:B.2}, where the construction is applied to $\wh{N}_{\lambda}$ 
that satisfies equation \myeq{eq:B.13} with $\delta = 
\mu/(1 - 2 \mu)$, so that 
\[ 
\frac{2 \delta}{1 - 2 \delta} = \frac{2 \mu}{1 - 4 \mu}.
\]  
\end{proof}

\subsection{Estimation of the eigenvalues}

Let us mention that the result stated in Proposition \ref{prop:1.2} 
induces an estimation of the eigenvalues of $G$. Indeed, 
if $\lambda_1 \geq \lambda_2 \geq \cdots \geq \lambda_d$ are 
the eigenvalues of $G$ (counted with their multiplicities), 
and $\wh{\lambda}_1 \geq \wh{\lambda}_2 \geq \cdots \geq \wh{\lambda}_d$ 
are the eigenvalues of $\wh{G}$, then when equation \eqref{eq:05} 
holds,
\[ 
\sup_{i \in \{1, \dots, d \}} \biggl\lvert \frac{\wh{\lambda}_i}{\lambda_i} - 1 \biggr\rvert \leq 
\frac{2 \mu}{1 - 4 \mu}. 
\] 
\begin{proof}
Let $\Gr \bigl( \B{R}^d, i \bigr)$ be the set of linear subspaces of 
$\B{R}^d$ of dimension $i$. 
The above inequality is 
a direct consequence
of the fact that 
\begin{align*}
\wh{\lambda}_i & = \sup \, \biggl\{  \inf \Bigl\{ 
\theta^{\top} \wh{G} \theta \, : \, \theta \in V \cap \B{S}_d \Bigr\} \, : \,  V 
\in \Gr \bigl( \B{R}^d, i \bigr)  
\biggr\}, \\ 
\text{whereas } 
\lambda_i & = \sup \, \biggl\{  \inf \Bigl\{ 
\theta^{\top} G \theta \, : \, \theta \in V \cap \B{S}_d  \Bigr\} \, : \,  V 
\in \Gr \bigl( \B{R}^d, i \bigr)  
\biggr\}. 
\end{align*}
These two identities themselves can be established 
from the remark that for any $V \in \Gr \bigl( \B{R}^d, i \bigr)$, 
and any orthonormal basis $(e_1, \dots, e_d)$, 
\[ 
\dim \Bigl( V \cap \Span \bigl\{ e_{i}, e_{i+1}, \dots, e_d \bigr\} \Bigr) 
\geq 1, 
\] 
considering the case when $(e_1, \dots, e_d)$ is a basis  
of eigenvectors of $\wh{G}$ or of $G$, 
corresponding to the eigenvalues $(\wh{\lambda}_1, \dots, 
\wh{\lambda}_d )$ or $(\lambda_1, \dots, \lambda_d )$ 
respectively.  
\end{proof}

\section{The empirical Gram matrix estimate}

In this section, we study the empirical Gram matrix estimate 
\[
\ov{G} = \frac{1}{n} \sum_{i=1}^n X_i X_i^{\top}, 
\]
and the corresponding quadratic form $\ov{N}(\theta) = \theta^{\top} 
\ov{G} \theta$. 
We use the previous robust estimate 
$\wh{N}_{\lambda}(\theta)$ of $\theta^{\top} 
G \theta$ as a tool. As we will always use the value of $\lambda$
defined in Proposition \vref{prop1.2.3}, 
we will write in this section for short $\wh{N}$ 
instead of $\wh{N}_{\lambda}$.  Our approach is  
to analyze the difference $\ov{N}(\theta) - \wh{N}(\theta)$, 
showing that it is small under suitable assumptions. 

First of all, we deduce from the definitions of 
$\wh{N}$ and $\ov{N}$ that $\wh{N}(\theta) = \ov{N}(\theta) = 0$ 
almost surely for any $\theta \in \Ker G$ (that is any 
$\theta$ such that $N(\theta) = 0$).

On the other hand, under the hypotheses of Proposition 
\vref{prop1.2.3}, with probability at least $1 - \epsilon$, 
$N(\theta)/ \wh{N}(\theta) \leq 1 + \wh{\delta}$, 
so that in the case when $N(\theta) > 0$, $\wh{N}(\theta) > 0$
also, so that $r_{\lambda} \bigl( \wh{N}(\theta)^{-1/2} \theta \bigr) 
= 0$.  
As a consequence, 
\begin{align*}
\frac{\ov{N}(\theta)}{\wh{N}(\theta)} - 1 & = \lambda^{-1} 
\int \lambda \bigl[ \langle \theta, x \rangle^2 \wh{N}(\theta)^{-1} 
- 1 \bigr] \, \ud \oB{P}(x) \\ & = 
\lambda^{-1} \int g \bigl[ \lambda \bigl( \langle \theta, x \rangle^2 
\wh{N}(\theta)^{-1} - 1 \bigr) \bigr] \, \ud \oB{P}(x),
\end{align*}
where $g(z) = z - \psi(z)$. It is easy to compute 
\[
g'(z) = 
\begin{cases} 
\ds 1, & z \geq 1 \\ 
\ds \frac{z^2}{1 + (z-1)^2}, & 0 \leq z \leq 1, \\ 
\ds g'(-z), & z \leq 0,
\end{cases}  
\] 
showing that $g'(z) \leq z^2$, and therefore that $ g(z) \leq \max \{ z, 0 \}^3/3$, 
for any $z \in \B{R}$. 
We see also that for any $p \in [0,2]$, and any $z \in \B{R}_+$, 
$g'(z) \leq z^p$, so that more generally $g(z) \leq \max \{ z, 0 \}^{p+1}/(p+1)$
for any $z \in \B{R}$ and any $p \in [0,2]$.  
As a consequence 
\begin{prop}
\label{prop2.1.2}
Let us make the same assumptions as in Proposition \vref{prop1.2.3}.
On an event of probability at least $1 - \epsilon$ that includes 
the event of probability at least $1 - 2 \epsilon$ 
described in Proposition \vref{prop1.2.3}, 
\begin{align*}
\frac{\ov{N}(\theta)}{ \wh{N}(\theta) } - 1 
& \leq \inf_{p \in [0,2]} \frac{\lambda^p}{p+1} \int \Bigl( \langle \theta, x \rangle^2 
\wh{N}(\theta)^{-1} - 1 \Bigr)_+^{p+1} \, \ud \oB{P}(x) \\ 
& \leq \frac{\lambda^2}{3} \int \biggl( \langle \theta, x \rangle^2 
\wh{N}(\theta)^{-1} - 1 \biggr)_+^3 \, \ud \oB{P}(x),  
\end{align*}
and
\[
1 - \frac{\ov{N}(\theta)}{\wh{N}(\theta)} \leq \frac{\lambda^2}{3} 
\int \Bigl( 1 - \langle \theta, x \rangle^2 \wh{N}(\theta)^{-1} 
\Bigr)_+^3 \, \ud \oB{P}(x) \leq \frac{\lambda^2}{3},
\]
where $(z)_+ = \max \{ z, 0 \}$ and $\lambda$ is defined as in Proposition \vref{prop1.2.3}, 
so that 
\[
\lambda^2 = \frac{2 \bigl[ \log(\epsilon^{-1}) + 0.73 \, d \bigr]}{ 
(\kappa-1) n}.
\] 
\end{prop}
This proposition uses random upper bounds. 
Nevertheless, it gives an indication that in good cases, 
when the fluctuations 
of these random upper bounds are not too wild, the difference 
between $\ov{N}$ and $\wh{N}$, measured by $\ds \biggl\lvert 
\frac{\ov{N}(\theta)}{\wh{N}(\theta)} - 1 \biggr\rvert $, 
should be of order $\lambda^2$, that is of order $n^{-1}$, whereas, as we have already 
seen, $\ds \biggl\lvert \frac{N(\theta)}{\wh{N}(\theta)} 
-  1 \biggr\rvert$ is of order $n^{-1/2}$. 
More precisely, the second inequality proves that $\ov{N}(\theta)$ 
cannot be significantly {\it smaller} than $\wh{N}(\theta)$, 
while the first inequality shows that it can be significantly 
larger, but only in the case when the fluctuations of the 
random quantity 
\[ 
\int \langle \theta, x \rangle^6 N(\theta)^{-3} \ud \ov{\B{P}}(x)
\] 
are not bounded with $n$, since with probability $1 - 2 \epsilon$ 
\begin{multline*}
\int \Bigl( \langle \theta, x \rangle^2 \wh{N}(\theta)^{-1} - 1 \Bigr)^3_+ 
\, \ud \ov{\B{P}}(x) \leq \int \biggl( \langle \theta, x \rangle^2 
N(\theta)^{-1} \biggl( \frac{1 - \mu}{1 - 2 \mu} 
\biggr) - 1 \biggr)^3_+ \, \ud \oB{P} (x) \\ 
\leq \biggl( \frac{1 - \mu}{1 - 2 \mu} \biggr)^{3} 
\int \langle \theta, x \rangle^6 N(\theta)^{-3} \, \ud \oB{P}(x),
\end{multline*}
where $\mu$ is defined as in Proposition \vref{prop1.2.3}. 

We will now replace the bounds in the previous proposition 
by more explicit ones.\\
Write the Gram matrix $G$ in diagonal form as 
\[
G = U \diag\bigl( \lambda_1, \dots, \lambda_d \bigr) \, U^{\top},
\] 
where $U U^{\top} = \mb{I}$ and $\lambda_1 \geq \cdots \geq \lambda_d$, 
and define 
\[
G^{-1/2} = U \diag \Bigl[ \B{1}\bigl( \lambda_i > 0 \bigr) 
\lambda_i^{-1/2}, i=1, \dots, d \Bigr] U^{\top}. 
\]
As almost surely $X_i \in \IM(G)$, almost surely
$ G^{1/2} \, G^{-1/2} X_i = X_i$ and therefore
\begin{align*}
\bigl\lVert G^{-1/2} X_i \bigr\rVert & = \sup \, \Bigl\{ 
\bigl\langle G^{-1/2} X_i, 
\theta \bigr\rangle \, : \, \theta \in \IM(G), \lVert \theta \rVert \leq 1 \Bigr\} \\
& = \sup \, \Bigl\{ \bigl\langle G^{-1/2} X_i , G^{1/2} \theta \bigr\rangle 
\, : \, \theta \in \B{R}^d, \bigl\lVert G^{1/2} \theta \bigr\rVert \leq 1 
\Bigr\} 
\\ & = 
\sup \, \Bigl\{ \langle X_i  , \theta \rangle \, : \, \theta \in \B{R}^d, 
\B{E} \bigl( \langle \theta, X_i \rangle^2 \bigr) \leq 1 \Bigr\}. 
\end{align*}
Consider 
\begin{equation}
\label{eq:06}
R = \max_{i=1, \dots, n} \lVert G^{-1/2} X_i \rVert 
= \max_{i=1, \dots, n} \sup \Bigl\{ \langle \theta, X_i \rangle
\; : \; 
\theta \in \B{R}^d, \B{E} \bigl( \langle \theta, X \rangle^2 \bigr) 
\leq 1 \Bigr\}. 
\end{equation}
Using these remarks and this definition, 
we can state the following consequence of Proposition \ref{prop2.1.2}:
\begin{prop} 
\label{prop3.1} 
Define the quantities
\begin{align*}
\wh{\delta} & = \frac{\mu}{1 - 2 \mu}, 
\quad \text{ where $\mu$ is as in Proposition \thmref{prop1.2.3}}, 
\\ 
\gamma_+ & = \frac{2 \bigl[ \log(\epsilon^{-1}) + 0.73 \, d \bigr] 
R^4 \, ( 1 + \wh{\delta})^2}{ 
3 (\kappa - 1) n}, \\ 
& \qquad \text{ where $R$ is defined in equation \eqref{eq:06}}, 
\\
\gamma_- & = \frac{2 \bigl[ \log(\epsilon^{-1}) + 0.73 \, 
d\bigr]}{3(\kappa-1) n}, 
\\ 
\ov{\delta}_+(\theta)  & =  
\frac{2 \bigl[ \log(\epsilon^{-1}) + 0.73 \, d \bigr]}{3(\kappa-1) n}
\int \Bigl( \langle \theta, x \rangle^2 \wh{N}(\theta)^{-1} - 1 \Bigr)_+^3 \, \ud \oB{P}(x) \\
& \leq \gamma_+ \frac{\ov{N}(\theta)}{\wh{N}(\theta)} \\   
\ov{\delta}_- (\theta) & =  
\frac{2 \bigl[ \log(\epsilon^{-1}) + 0.73 \, d \bigr]}{3(\kappa-1) n}
\int \Bigl( 1 - \langle \theta, x \rangle^2 \wh{N}(\theta)^{-1} 
\Bigr)_+^3 \, \ud \oB{P}(x) \\
& \leq \gamma_-. 
\end{align*}
Under the hypotheses of Proposition \vref{prop1.2.3}, using 
the above notation and definitions, with probability 
at least $1 - 2 \epsilon$, for any $\theta \in \B{R}^d$
\begin{multline*}
- \underbrace{\frac{\wh{\delta} + \gamma_-}{1 + \wh{\delta}}}_{
= \bO ( \wh{\delta} )
}\leq - \frac{\wh{\delta} + \ov{\delta}_-(\theta)}{1 
+ \wh{\delta}} \leq \frac{\ov{N}(\theta)}{N(\theta)} - 1 
\leq \frac{\wh{\delta} + \ov{\delta}_+(\theta)}{1 - \wh{\delta}} 
\\ \leq \frac{1}{(1 - \wh{\delta})(1 - \gamma_+)} - 1 
\leq \frac{\wh{\delta} + \gamma_+}{(1 - \wh{\delta})(1 - \gamma_+)}, 
\end{multline*}
where it is useful to remember that $\ds \wh{\delta}
= \bO \Biggl( \sqrt{\frac{ \kappa [ 
d + \log(\epsilon^{-1}) ]}{n}} \; \Biggr)$. 
\end{prop}
The proof of this proposition is given in appendix.
Let us remark that our lower bound for $\ov{N}(\theta) / N(\theta)$, 
that is always $\bO \bigl( \wh{\delta} \bigr)$ and holds under the 
hypotheses stated in Proposition \vref{prop1.2.3}, 
can be compared with the lower bound on the smallest 
singular value of a random matrix with i.i.d. isotropic 
columns given in \cite[Theorem 1.3]{Kol2015}. The first point 
in this theorem of V. Koltchinskii and S. Mendelson gives
a slightly worse lower bound with less explicit constants
(in particular, the dependence in $\kappa$ is not explicit)
under a slightly stronger condition, whereas points 2. and 3. 
of their theorem prove slower rates than $1/\sqrt{n}$ under 
weaker assumptions than ours.    

Let us now upper-bound the random quantity $R$ defined by equation 
\eqref{eq:06} and consequently $\gamma_+$ 
under suitable assumptions. 
A simple choice is to assume an exponential moment
of the type
\begin{equation}
\label{eq4.2}
\B{E} \biggl\{ \exp \biggl[ \, \frac{\alpha}{2} \Bigl( \lVert G^{- 1/2} X \rVert^2 - 
d - \eta \Bigr) \biggr] \biggr\} \leq 1, 
\end{equation}
where $\alpha$ and $\eta$ are two positive real constants. 

Under this assumption, with probability at least $1 - \epsilon$, 
\[ 
R^2 \leq d + \eta +  \frac{2}{\alpha} \log \bigl( n / \epsilon
\bigr).
\] 
Remark that $\B{E} \bigl( \lVert G^{-1/2} X \rVert^2 \bigr) = d$, 
so that assumption \eqref{eq4.2} can also be written as 
\[
\B{E} \biggl\{ \exp \biggl[ \, \frac{\alpha}{2} 
\Bigl( \lVert G^{- 1/2} X \rVert^2 - \B{E} \bigl( \lVert G^{-1/2} X \rVert^2 \bigr) - \eta \Bigr) \biggr] \biggr\} \leq 1. 
\]
Let us also remark that in the case when $X \in \B{R}^d$ is a 
centered Gaussian vector,
\[ 
\B{E} \biggl\{ \exp \biggl[ \frac{\alpha}{2} \Bigl( \lVert G^{-1/2} 
X \rVert^2 + \frac{d}{\alpha} \log (1 - \alpha) \Bigr) \biggr] \biggr\} = 1, 
\qquad 0 < \alpha < 1,
\] 
so that we can take in this case
$\ds \eta =  - \frac{d}{\alpha}  \bigl[ \log(1 - \alpha) 
+ \alpha \bigr]$. 
In the case when $X$ is a non centered Gaussian vector, 
we can also check that 
\[ 
\B{E} \biggl\{ \exp \biggl[ \, \frac{\alpha}{2} \Bigl( 
\lVert G^{-1/2} X \rVert^2 + \frac{d}{\alpha}
\log( 1 - \alpha) \Bigr) \biggr] \biggr\} \leq 1, \qquad 
0 < \alpha < 1,
\] 
so that the same choice of $\eta$ is still valid. 

Hypothesis \eqref{eq4.2} is quite strong, and can be replaced by 
the more general assumption that  
\begin{equation} 
\label{eq4.3}
\B{E} \biggl\{ \exp \biggl[ \frac{\alpha}{2} \Bigl( \lVert G^{-1/2} 
X \rVert^{2p} - \B{E} \bigl( \lVert G^{-1/2} X \rVert^{2p} \bigr)  - \eta \Bigr) \biggr] \biggr\} \leq 1,  
\end{equation}
for some exponent $p \in ]0, 1]$ and positive constants $\alpha$ 
and $\eta$. Under this new assumption, with probability at least 
$1 - \epsilon$, 
\begin{align*}
R^2 & \leq 
\biggl( \B{E} \bigl( \lVert G^{-1/2} X \rVert^{2p} \bigr) + \frac{2}{\alpha} \log \bigl( n / \epsilon \bigr) + \eta \biggr)^{1/p} \\ 
& \leq \biggl( d^p + \frac{2}{\alpha} \log 
\bigl( n / \epsilon \bigr) + \eta \biggr)^{1/p}, 
\end{align*}
since
\[ 
\B{E} \bigl( \lVert G^{-1/2} X \rVert^{2p} \bigr) \leq \B{E} 
\bigl( \lVert G^{-1/2} X \rVert^2 \bigr)^p = d^p.
\] 
\begin{prop}
\label{prop:2.3}
Let us assume that condition \eqref{eq4.3} as well as the  
hypotheses of Proposition \vref{prop1.2.3} are satisfied, 
and introduce the constant
\begin{align*}
\wh{\gamma}_+ & = \frac{2 \bigl[ \log(\epsilon^{-1}) + 
0.73 \, d \bigr] 
\bigl[ \B{E} \bigl( \lVert G^{-1/2} X \rVert^{2p} \bigr) 
+ 2 \alpha^{-1} \log(n / \epsilon) + \eta  
\bigr]^{2/p}(1 + \wh{\delta} )^2 }{ 
3 (\kappa - 1) n} \\ & \leq 
\frac{2 \bigl[ \log(\epsilon^{-1}) + 
0.73 \, d \bigr] 
\bigl[ d^p 
+ 2 \alpha^{-1} \log(n / \epsilon) + \eta  
\bigr]^{2/p}(1 + \wh{\delta} )^2 }{ 
3 (\kappa - 1) n}.
\end{align*}
With probability at least $1 - \epsilon$, 
\[
\gamma_+ \leq \wh{\gamma}_+, 
\]
so that with probability at least $1 - 3 \epsilon$, 
for any $\theta \in \B{R}^d$, 
\[ 
- \underbrace{\frac{\wh{\delta} + \gamma_-}{1 + \wh{\delta}}}_{= 
\bO (\wh{\delta})} 
\leq \frac{\ov{N}(\theta)}{N(\theta)} 
- 1 \leq \underbrace{\frac{\wh{\delta} + \wh{\gamma}_+}{
(1 - \wh{\delta})(1 - \wh{\gamma}_+)_+}}_{= \, \bO_{n \rightarrow \infty} 
( \wh{\delta})},
\] 
where the notations are the same as in Proposition \vref{prop3.1} and 
where it is useful to remember that 
\[
\wh{\delta} = \bO \Biggl( \sqrt{\frac{
\kappa \bigl[ d + \log(\epsilon^{-1}) \bigr]}{n}} \; \Biggr).
\]
\end{prop}
We can also replace hypothesis \eqref{eq4.3} by a polynomial moment assumption.
Remembering Proposition \vref{prop2.1.2}, remark that on the set 
$\Omega$ of probability at least $1 - 2 \epsilon$ appearing in 
Proposition \vref{prop1.2.3}, for any $\theta \in \B{R}^d$,  
\begin{multline}
\label{eq:7}
\frac{\ov{N}(\theta)}{\wh{N}(\theta)} - 1 \leq 
\frac{\lambda^p}{p+1} \int \Bigl( \langle \theta, x \rangle^2 \wh{N}(\theta)^{-1}  - 1 
\Bigr)_+^{p+1} \, \ud \oB{P}(x) \\ \leq  \frac{\lambda^p}{p+1} \int 
\Bigl( N(\theta) \lVert G^{-1/2} x \rVert^2 \wh{N}(\theta)^{-1} -  1 
\Bigr)_+^{p+1} \, \ud \oB{P}(x).
\end{multline}
Since 
$N(\theta) \wh{N}(\theta)^{-1} \leq 1 + \wh{\delta}$
on the event $\Omega$, 
\[ 
\int \Bigl( \langle \theta, x \rangle^2 \wh{N}(\theta)^{-1} 
- 1 \Bigr)_+^{p+1} \, \ud \oB{P}(x) \leq \int 
f(x) \oB{P}(x), 
\] 
where 
\[
f(x) = \Bigl( (1 + \wh{\delta}) \lVert G^{-1/2} x \rVert^2 - 1 \Bigr)_+^{p+1}.
\]
From Bienaym\'e Chebyshev's inequality,
with probability at least $1 - \epsilon$, 
\[
\int f(x) \, \ud \oB{P}(x) \leq \B{E} \bigl[ f(X) \bigr] + \biggr( \frac{
\Var \bigl[ f(X) \bigr]}{ n \epsilon} \biggr)^{1/2} 
\\ \leq \B{E} \bigl[ f(X) \bigr] + \biggl( 
\frac{\B{E} \bigl[ f(X)^2 \bigr]}{n\epsilon} \biggr)^{1/2}.
\]
This leads to 
\begin{prop}
\label{prop:2.4} 
Consider some exponent $p \in ]1,2]$ and introduce the bound
\begin{multline*}
\wt{\gamma}_+ =  
\frac{1}{p+1} \biggl( \frac{2 \bigl[ \log(\epsilon^{-1}) + 0.73 \, d \bigr]}{(\kappa - 1) n} \biggr)^{p/2} 
\Biggl[ \B{E} \Bigl[ \Bigl( 
( 1 + \wh{\delta}) \lVert G^{-1/2} X \rVert^2 - 1 \Bigr)_+^{p+1} \Bigr] 
\\ \shoveright{+ \Biggl( \frac{\ds \B{E} \bigl[ \bigl( (1 + \wh{\delta}) 
\lVert G^{-1/2} X \rVert^2 - 1 \bigr)_+^{2p + 2} \bigr]}{n \epsilon} 
\Biggr)^{1/2} \Biggr]} \\
\leq 
\frac{1}{p+1} \biggl( \frac{2 \bigl[ \log(\epsilon^{-1}) + 0.73 \, d \bigr]}{  
(\kappa - 1) n} \biggr)^{p/2} 
(1 + \wh{\delta})^{p+1} \Biggl[ \B{E} \bigl( \lVert 
G^{-1/2} X \rVert^{2p + 2} \bigr) \\ + \Biggl( \frac{ \B{E} \bigl( 
\lVert G^{-1/2} X \rVert^{4 p + 4} \bigr)}{n\epsilon} \Biggr)^{1/2} \Biggr].
\end{multline*}
Under the hypotheses of Proposition \vref{prop1.2.3}, with probability at least 
$1 - 3 \epsilon$, for any $\theta \in \B{R}^d$, 
\[ 
- \underbrace{\frac{\wh{\delta} + \gamma_-}{1 + \wh{\delta}}}_{=
\, \bO (\wh{\delta})} 
\leq \frac{\ov{N}(\theta)}{N(\theta)} - 1 \leq 
\underbrace{\frac{\wh{\delta} + \wt{\gamma}_+}{(1 - \wh{\delta})(1 - 
\wt{\gamma}_+)_+}}_{ 
= \, \bO_{n \rightarrow \infty} (\wh{\delta})}, 
\]
where the notations are otherwise the same as in 
Proposition \vref{prop3.1} and whese we recall that 
\[ 
\wh{\delta} = \bO \Biggl( \sqrt{\frac{
\kappa \bigl[ d + \log(\epsilon^{-1}) \bigr]}{n}} \; \Biggr). 
\]
\end{prop}
To give an idea of the best possible order of the bound with respect 
to the dimension, let us notice that, from Jensen's inequality,
\[
\wt{\gamma}_+ \geq  
\frac{1}{p+1} \biggl( \frac{2 \bigl[ \log(\epsilon^{-1}) + 0.73 \, 
d \bigr]}{ 3 (\kappa - 1) n} \biggr)^{p/2} 
\bigl[(1 + \wh{\delta}) d - 1 \bigr]^{p+1}
\bigl[ 1 + (n \epsilon)^{-1/2} \bigr].
\]
Therefore, if we take for example $\epsilon = 1/n$, the best we 
can hope for is to get a $\wt{\gamma}_+$ of order $d^{3p/2+1} / n^{p/2} 
= d^4/ n$ when $p = 2$.
The power of $d$ in this second order term is clearly not 
optimal, due to the rather crude inequality \eqref{eq:7} used 
to get this proposition. 
It would have been more satisfactory to get an upper bound $\wt{\gamma}_+$ 
of the same order as $\gamma_-$, that is of order $d/n$, 
at least when $p = 2$. 
Anyhow, the idea here is that, since we are dealing with 
a second order term in $1/n$, we can privilege the simplicity 
of the proof over the sharpness of the result. 

Also, one may think that imposing that $\B{E} \bigl( 
\lVert G^{-1/2} X \rVert^{4p+4} \bigr) < \infty$ for some 
$p > 1$ that is necessary to get $\wt{\gamma}_+  < \infty$
is asking for a pretty high 
moment condition. One can get a rate depending on a lower moment 
assumption using the following variant of Bienaym\'e Chebishev's 
inequality. 
\begin{lemma}
\label{lem:2.5}
Let $q \in [1,2]$ be some exponent.
Consider the constant 
\[ 
C_q = 
\frac{q^{q-1} }{2 (q-1)^{q-1} (1 - q/2)^{(2-q)/q}} \leq 1.4, 
\]  
where by convention $C_1 = \lim_{q \rightarrow 1+} C_q = 1$ 
and $C_2 = \lim_{q \rightarrow 2-} C_q = 1$. 
Let $W_1, \dots, W_n$ be $n$ independent copies of a non-negative 
real valued random variable $W$. With probability at least $1 - 2 \epsilon$, 
\[ 
\frac{1}{n} \sum_{i=1}^n W_i \leq 
\B{E}(W) + \frac{C_q \B{E}(W^q)^{1/q}}{ 
\epsilon^{1/q} n^{1 - 1/q}}.
\] 
\end{lemma} 
\mbox{} \hfill \framebox{
\begin{minipage}{0.9\textwidth}
\noindent \mbox{} \hfill 
\raisebox{-4ex}[45ex][0ex]{
\includegraphics[width=0.8\textwidth]{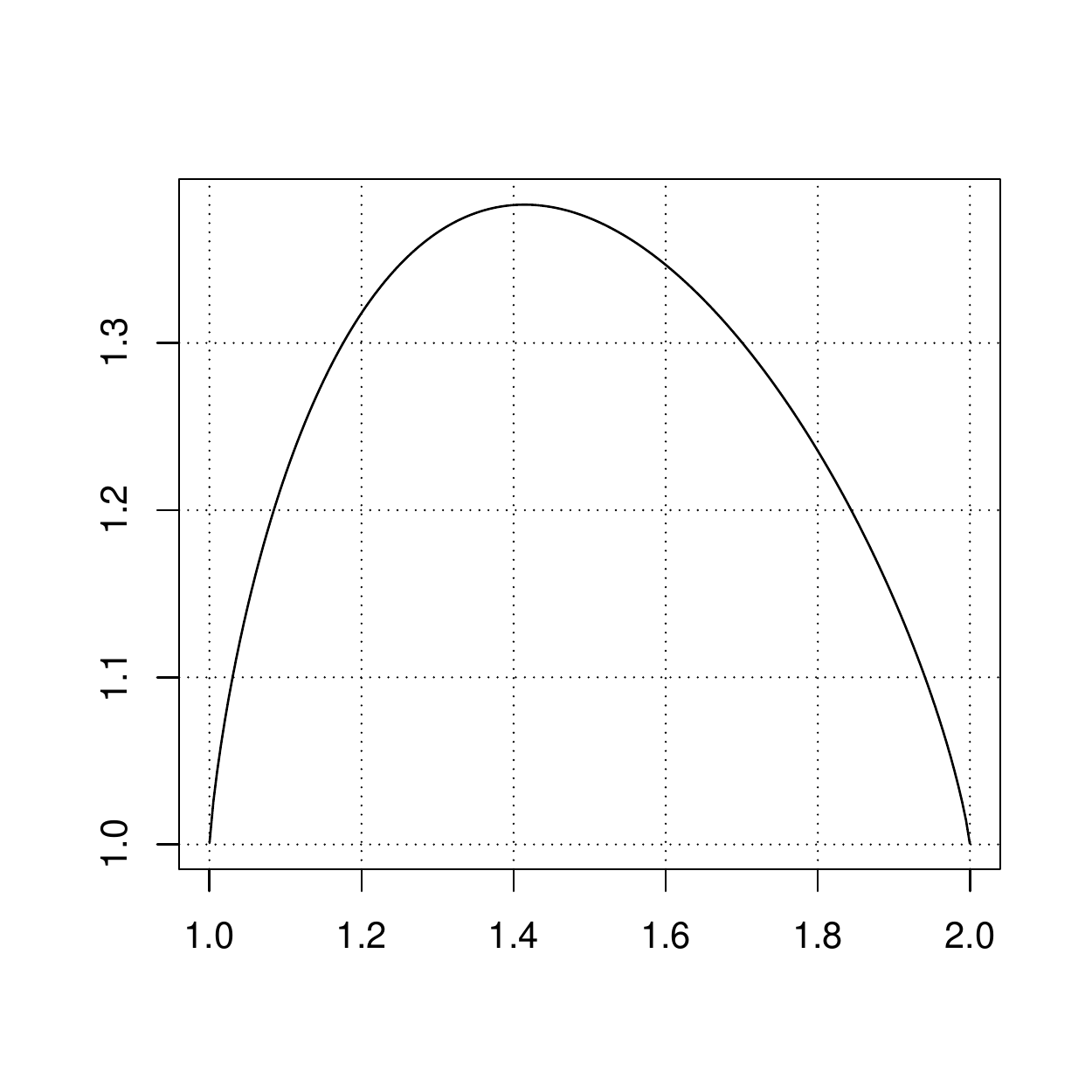}
}
\hfill \hfill \hfill \mbox{}\\[-2ex]
\mbox{} \hfill $q \mapsto C_q$ \hfill \mbox{}\\
\end{minipage}} \hfill \mbox{}\\[2ex]

Using this lemma, 
we obtain the following variant of Proposition \ref{prop:2.4}. 
\begin{prop}
\label{prop:2.6}
Consider some exponents $p \in [1, 2]$ and $q \in [1,2]$ 
and the bound 
\begin{multline*}
\wt{\gamma}_+ = \frac{1}{p+1} \biggl( \frac{2 \bigl[ \log(\epsilon^{-1} 
+ 0.73 d \bigr]}{(\kappa - 1) n} \biggr)^{p/2} 
\bigl( 1 + \wh{\delta})^{p+1} \\ \times \Biggl[ 
\B{E}\bigl( \lVert G^{-1/2} X \rVert^{2(p+1)} + \frac{ 
C_q \B{E} \Bigl( \lVert G^{-1/2} X \rVert^{2 q (p+1)} \Bigr)^{1/q}}{
\epsilon^{1/q} n^{1 - 1/q}} \; \Biggr].
\end{multline*}
Under the hypotheses of Proposition \vref{prop1.2.3}, 
with probability at least $1 - 4 \epsilon$, 
for any $\theta \in \B{R}^d$, 
\[ 
- \underbrace{\frac{\wh{\delta} + \gamma_-}{1 + \wh{\delta}}}
_{ = \, \bO(\wh{\delta})}  \leq 
\frac{\ov{N}(\theta)}{N(\theta)} - 1 \leq 
\frac{1}{(1 - \wh{\delta})(1 - \wt{\gamma}_+)_+} - 1 \leq \underbrace{\frac{\wh{\delta} + \wt{\gamma}_+}{
(1 - \wh{\delta})(1 - \wt{\gamma}_+)_+}}_{\substack { 
= \, \bO_{n \rightarrow \infty} (\wh{\delta}) \\ 
\text{when } p > 1 }},  
\] 
where it is useful to remember that 
\[ 
\wh{\delta} = \bO \Biggl( \sqrt{\frac{\kappa \bigl[ d + 
\log(\epsilon^{-1}) \bigr]}{n}} \; \Biggr).
\] 
\end{prop}  

We have done two things in this section about the empirical Gram 
matrix estimate. The first was to 
make a direct comparison between $\ov{N}$ and $N$. The 
second was to introduce non-random bounds in terms of 
the quantities $\gamma_-$, $\wh{\gamma}_+$ and $\wt{\gamma}_+$. 

Remark that $\gamma_-$, controlling the accuracy of $\ov{N}(\theta)$ 
as an upper bound for $N(\theta)$, 
is always of order $1/n$, meaning that the empirical 
quadratic form $\ov{N}(\theta)$ always provides an upper bound of the 
same quality as the robust estimate $\wh{N}(\theta)$. 
This comes from the fact that $\langle \theta, X \rangle^2$ 
being a non-negative random variable cannot have a long tail 
on the left-hand side, but only on the right-hand side, 
so that its empirical mean can be too large, but not too 
small.  

On the other hand, $\delta_+(\theta)$, 
controlling the accuracy
of $\ov{N}(\theta)$ as a lower bound for $N(\theta)$, 
is not necessarily a second order term in all circumstances.
When this is not case, 
the fact that the empirical Gram estimate may become 
less accurate than the robust estimate $\wh{N}(\theta)$
is not ruled out, 
and we will see on simulations that this does happen  
in practice.  

\section{Estimation of the covariance matrix} 

To estimate the covariance matrix 
\[
\Sigma = \B{E} \bigl[ \bigl( X - \B{E}(X) \bigr) \bigl( X - \B{E}(X) 
\bigr)^{\top} \bigr], 
\]
one can remark that 
\[ 
\theta^{\top} \Sigma \theta = \inf_{\gamma \in \B{R}} 
\B{E} \bigl[ \bigl( \langle \theta, X \rangle - \gamma \bigr)^2 \bigr]
\] 
and consider the quadratic form 
\[
N(\theta, \gamma) = \B{E} \bigl[ \bigl[ \langle 
\theta, X \rangle - \gamma \bigr)^2 \bigr] 
\]
that corresponds to the Gram matrix of the extended variable $(X, -1) \in 
\B{R}^{d+1}$. 
\begin{prop}
\label{prop3.1.2}
Consider an estimator $\wh{N}(\theta, \gamma)$ of $N(\theta, \gamma)$.
Let $\delta$ and $\epsilon$ be two positive constants. Assume that  
with probability 
at least $1 - \epsilon$, for any $(\theta, \gamma) \in \B{R}^{d+1}$
\[ 
\Biggl\lvert \frac{N(\theta, \gamma)}{\wh{N}(\theta, \gamma)} - 1 
\Biggr\rvert \leq \delta.
\]   
Define
\[
\wt{N}(\theta) = \inf_{\gamma \in \B{R}} 
\wh{N}(\theta, \gamma)
\]
With probability at least $1 - \epsilon$, for any $\theta \in \B{R}^d$,  
\[ 
\biggl\lvert \frac{ \theta^{\top} \Sigma \theta}{\wt{N}(\theta)} 
- 1 \biggr\rvert \leq \delta.
\] 
\end{prop}
\begin{proof}
With probability at least $1 - \epsilon$, 
\begin{multline*}
(1 - \delta) \wt{N}(\theta) = 
\inf_{\gamma \in \B{R}} (1 - \delta) \wh{N}(\theta, \gamma) \leq  
\inf_{\gamma \in \B{R}} 
N(\theta, \gamma) \\ = \theta^{\top} \Sigma \theta \leq  
\inf_{\gamma} (1 + \delta) 
\wh{N}(\theta, \gamma) = (1 + \delta) \wt{N}(\theta).
\end{multline*}
\end{proof}

We can then use the previous sections to describe more precisely  
estimators $\wh{N}(\theta, \gamma)$ that can be used 
under suitable conditions. To establish those conditions, 
the following lemma will be helpful.
\begin{lemma}
\label{lemma:3.2}
Let us put
\begin{align*}
\kappa & = \sup \, \Bigl\{ 
\B{E} \bigl( \langle \theta, X - \B{E}(X) \rangle^4 \bigr)
\, : \, \theta \in \B{R}^d, \; \B{E} \bigl( \langle \theta, X - \B{E}(X) 
\rangle^2 \bigr) \leq 1 \Bigr\} \\ 
\text{and } \quad  \kappa' & = \sup \, \Bigl\{ 
\B{E} \bigl[ \bigl( \langle \theta, X \rangle - \xi \bigr)^4 \bigr] 
\, : \, \theta \in \B{R}^d, \xi \in \B{R}, \; 
\B{E} \bigl[ \bigl( \langle \theta, X \rangle - \xi \bigr)^2 \bigr] \leq 1 
\Bigr\}. 
\end{align*}
These two kurtosis coefficients are related by the inequality
\[ 
\kappa' \leq ( \sqrt{\kappa} + 1 )^2.
\] 
\end{lemma} 
\begin{proof}
Using successively the triangular inequality in $\B{L}^4(\B{P})$, 
the definition of $\kappa$ and the Cauchy-Schwarz inequality in $\B{R}^2$,   
\begin{multline*}
\B{E} \bigl[ \bigl( \langle \theta, X \rangle - \xi \bigr)^4 \bigr]^{1/2}
\leq 
\Bigl( 
\B{E}\bigl( \langle \theta, X - \B{E}(X) \rangle^4 \bigr)^{1/4} + 
\lvert \langle \theta, \B{E}(X) \rangle - \xi \rvert \Bigr)^2 
\\ \leq \Bigl( \kappa^{1/4} 
\B{E} \bigl( \langle \theta, X - \B{E}(X) \rangle^2 \bigr)^{1/2} + 
\lvert \langle \theta, \B{E}(X) \rangle - \xi \rvert \Bigr)^2 
\\ \leq ( \kappa^{1/2} + 1 ) \Bigl[ \B{E} \bigl( \langle \theta, X 
- \B{E}(X) \rangle^2 \bigr) + \bigl( \langle \theta, \B{E}(X) \rangle 
- \xi \bigr)^2 \Bigr] \\  
= (\kappa^{1/2} + 1) \B{E} \bigl[ \bigl( \langle \theta, X \rangle 
- \xi \bigr)^2 \bigr]. 
\end{multline*}
\end{proof}

\begin{prop}
\label{prop:3.3} 
Let $(X_1, \dots, X_n)$ be $n$ independent copies of a vector 
valued random variable $X \in \B{R}^d$. 
Assume that for any $\theta \in \B{R}^d$, 
\[ 
\B{E} \bigl( \langle \theta, X - \B{E}(X) \rangle^4 \bigr) 
\leq \kappa \; \B{E} \bigl( \langle \theta, X - \B{E}(X) \rangle^2
\bigr)^2,
\] 
for a known constant $\kappa \in \B{R}_+$. 
Let 
\[
\wh{N}(\theta) = \inf_{\xi \in \B{R}} \inf \Biggl\{ 
\rho \in \B{R}_+^* \, , \, \sum_{i=1}^n 
\psi \biggl\{ \lambda \biggl[ \frac{ \bigl( 
\langle \theta, X_i \rangle - \xi \bigr)^2}{\rho}
- 1 \biggr] \biggr\} \leq 0 \Biggr\}, 
\] 
where 
\[ 
\lambda = \sqrt{\frac{2}{(\kappa + 2 \sqrt{\kappa})n} \bigl[ \log(\epsilon^{-1}) 
+ 0.73 (d+1) \bigr]}.
\] 
For any confidence parameter $\epsilon > 0$, and any sample size $n$ 
such that 
\begin{multline}
\label{eq6.2}
n > \Biggl[ 20 \, (\kappa^{1/2} + 1) (d+1)^{1/2} \\ + \biggl( \frac{5}{2} 
+ \frac{1}{
2 ( \kappa + 2 \kappa^{1/2} )} \biggr) 
\sqrt{2\,(\kappa + 2 \kappa^{1/2}) \bigl[ \log(\epsilon^{-1}) + 0.73 (d+1) 
\bigr]} \Biggr]^2, 
\end{multline}
with probability at least $1 - 2 \epsilon$, for any $\theta \in \B{R}^d$, 
\[
\biggl\lvert \frac{ \B{E} \bigl( \langle \theta, X - \B{E}(X) \rangle^2 \bigr)}{
\wh{N}(\theta)} - 1 \biggr\rvert \leq \frac{\mu}{1 - 2 \mu} 
= \bO \Biggl( \sqrt{\frac{\kappa \bigl[ d 
+ \log(\epsilon^{-1}) \bigr]}{n}} \; \Biggr), 
\]
where 
\[
\mu = \sqrt{\frac{2 (\kappa + 2 \kappa^{1/2})}{n} \bigl[ \log(\epsilon^{-1}) 
+ 0.73 (d+1) \bigr]} + 6.81 \, (\kappa^{1/2} + 1) \sqrt{ \frac{ 2(d+1)}{n}}, 
\]
assuming by convention that $0/0 = 1$ and $z/0 = + \infty$ for any 
$z > 0$. 
\end{prop}
The proof of this proposition is given in the appendix.\\[1ex]
Consider now the empirical covariance estimate
\[ 
\ov{\Sigma} = \frac{1}{2n^2} \sum_{i=1}^n \sum_{j=1}^n 
(X_i - X_j)(X_i - X_j)^{\top}. 
\] 
(We choose a biased normalization by $n^2$ instead of $n(n-1)$, 
because, as we will see, we can prove 
a simpler non-asymptotic result for it.)

Remark that 
\[
\theta^{\top} \ov{\Sigma} \, \theta = \inf_{\xi \in \B{R}} 
\frac{1}{n} \sum_{i=1}^n \bigl( \langle \theta, X_i \rangle - \xi \bigr)^2 
= \frac{1}{n} \sum_{i=1}^n  \Bigl\langle \theta, X_i - 
\frac{1}{n} \sum_{j=1}^n X_j \Bigr\rangle^2.
\] 
In order to use Proposition \vref{prop3.1}, we need the following lemma. 
\begin{lemma}
\label{lemma:3.4}
Almost surely, for any $(\theta, \xi) \in \B{R}^{d+1}$ such that 
\[
\B{E} \bigl[ \bigl( \langle \theta, X \rangle - \xi \bigr)^2 \bigr] 
\leq 1, 
\] 
\[
\langle \theta, X_i \rangle - \xi
\leq \bigl( \lVert \Sigma^{-1/2} \bigl( X_i - \B{E}(X) \bigr) \rVert^2 + 1  \bigr)^{1/2}.
\] 
\end{lemma} 
This lemma is proved in appendix. 

We are now ready to apply Proposition \vref{prop3.1}.
Define 
\begin{align*}
R & = \max_{i=1, \dots, n} \bigl( \bigl\lVert 
\Sigma^{-1/2} \bigl(X_i - \B{E}(X) \bigr) \bigr\rVert^2 + 1 \bigr)^{1/2}, \\  
\mu & = \sqrt{ \frac{2 (\kappa + 2 \kappa^{1/2})}{n} \bigl[ 
\log(\epsilon^{-1}) + 0.73 \, (d+1) \bigr]} + 6.81 \, (\kappa^{1/2} + 1) 
\sqrt{\frac{2(d+1)}{n}}, \\ 
\wh{\delta} & = \frac{\mu}{1 - 2 \mu}, \\ 
\gamma_- & = \frac{2 \bigl[ \log(\epsilon^{-1}) + 0.73 \, (d+1) 
\bigr]}{3 (\kappa + 2 \kappa^{1/2}) \, n},  \\ 
\gamma_+ & = \gamma_- \, R^4 \bigl( 1 + \wh{\delta} \bigr)^2. 
\end{align*}

\begin{prop}
Assume that for some known constant $\kappa \in \B{R}_+$, 
\[ 
\sup \Bigl\{ \B{E} \bigl( \langle \theta, X - \B{E}(X) \rangle^4 \bigr) \, 
: \, \theta \in \B{R}^d, \B{E} \bigl( \langle \theta, X - \B{E}(X) \rangle^2 
\bigr) \leq 1 \Bigr\} \leq \kappa < \infty.
\]
Consider any confidence parameter $\epsilon > 0$ and any 
sample size $n$ satisfying 
\begin{multline*}
n > \Biggl[ 20 (\kappa^{1/2} + 1) \sqrt{d+1} \\ + \biggl( \frac{5}{2} + \frac{1}{2 (\kappa + 2 \kappa^{1/2})} \biggr) \sqrt{2 ( \kappa + 2 \kappa^{1/2}) 
\bigl[ \log(\epsilon^{-1}) + 0.73 \, (d+1) \bigr] } \biggr]^2.
\end{multline*}
With probability at least $1 - 2 \epsilon$, for any $\theta \in \B{R}^d$, 
\[
- \underbrace{\frac{\wh{\delta} + \gamma_-}{1 + \wh{\delta}}}_{ = \, \bO 
( \wh{\delta})} \leq \frac{\theta^{\top} \ov{\Sigma} \theta}{
 \theta^{\top} \Sigma \theta} - 1 \leq 
\frac{\wh{\delta} + \gamma_+}{(1 - \wh{\delta}) (1 - \gamma_+)_+},
\]
where $\ds \wh{\delta} = \bO \Biggl( \sqrt{\frac{\kappa \bigl[ 
d + \log (\epsilon^{-1}) \bigr]}{n}} \; \Biggr)$. 
\end{prop}
\begin{proof}
Let us introduce 
\[ 
\ov{N}(\theta, \xi) = \frac{1}{n} \sum_{i=1}^n \bigr( \langle \theta, X_i \rangle
- \xi \bigr)^2.
\] 
From proposition \vref{prop3.1}, we see that on some event $\Omega'$ 
of probability 
at least $1 - 2 \epsilon$, for any $(\theta, \xi) \in \B{R}^{d+1}$, 
\[ 
B_- \overset{\mathrm{def}}{=} - \frac{\wh{\delta} + \gamma_-}{ 1 + \wh{\delta}} \leq 
\frac{\ov{N}(\theta, \xi)}{N(\theta, \xi)} - 1 \leq B_+ \overset{\mathrm{def}}{=} \frac{\wh{\delta} 
+ \gamma_+}{(1 - \wh{\delta})(1 - \gamma_+)_+}.
\] 
Remark that on $\Omega'$, 
\begin{align*}
\theta^{\top} \ov{\Sigma} \theta & = \inf_{\xi \in \B{R}} \ov{N}(\theta, \xi) 
\leq (1 + B_+) \inf_{\xi \in \B{R}} N(\theta, \xi)  = 
(1 + B_+) \, \theta^{\top} \Sigma \theta, \\ 
\theta^{\top} \ov{\Sigma} \theta & \geq (1 - B_-) \inf_{\xi 
\in \B{R}} N(\theta, 
\xi) = (1 - B_-) \, \theta^{\top} \Sigma \theta.
\end{align*}
\end{proof}

It is easy from there to give non-asymptotic bounds for the 
empirical term $R$. The discussion is similar to the one 
following equation \myeq{eq4.2}, with $G$ replaced by $\Sigma$ 
and $X$ replaced by $X - \B{E}(X)$, so that we will 
not repeat it. One can also obtain results similar to 
Propositions \vref{prop:2.4} and Proposition \vref{prop:2.6}.

\section{Least squares regression}

In this section, we use our results on 
the estimation of the Gram matrix 
in two ways. First, 
we derive new robust estimators for 
least squares regression with a random design.  
Second, we obtain new results for the ordinary 
least squares estimator, including its 
exact asymptotic rate of convergence
in expectation under quite weak assumptions. 
In particular, it turns out that this rate 
is $C/n$, with a constant $C$ different 
from $d \sigma^2$, where $\sigma^2$ 
is the variance of the noise, in the case 
when the noise is correlated to the design. 

\subsection{From Gram matrix to least squares estimates}
Let us first make a connection between Gram matrix estimates and 
least squares estimates, that will serve to study a new robust 
least squares estimate as well as the empirical risk minimizer. 
Consider $n$ independent 
copies $(X_i, Y_i)$, $i = 1, \dots, n$, of the couple of random variables 
$(X, Y) \in \B{R}^d \times \B{R}$ 
and the question of minimizing the quadratic risk
$$
R(\theta) = \B{E} \bigl[ \bigl( Y - \langle \theta, X \rangle \bigr)^2 \bigr]
$$
in $\theta \in \B{R}^d$. 
Introduce the homogeneous quadratic form
$$
N(\theta, \xi) = \B{E} \bigl[ \bigl( \xi Y - \langle \theta, X \rangle 
\bigr)^2 \bigr],  \quad \theta \in \B{R}^d, 
\xi \in \B{R},
$$
and assume that $\breve{N}(\theta, \xi )$ is a quadratic estimator 
of $N(\theta, \xi )$, based on the sample $(X_i, Y_i)$, $i=1, \dots, n$,  
such that, for some positive constant $\delta$,
with probability $1 - \epsilon$,  
for any $(\theta, \xi ) \in \B{R}^{d+1}$,  
\begin{equation}
\label{eq6}
\biggl\lvert \frac{\breve{N}(\theta, \xi )}{N(\theta, \xi )} - 1 
\biggr\rvert \leq \delta. 
\end{equation}
Hypotheses under which such an estimator exists 
are stated in Proposition \vref{prop:1.2}
for a robust estimator  
and in Proposition \vref{prop:2.3} and \vref{prop:2.4} 
for the empirical Gram matrix 
estimator. We will come back to this in more detail 
afterwards.  
\begin{prop}
\label{prop:4.1}
Under the above hypotheses, any estimator 
\[
\wh{\theta} \in \arg \min_{\theta \in \B{R}^d} 
\breve{N}(\theta, 1)
\]
 is such that 
with probability at least $1 - \epsilon$,   
$$
R(\wh{\theta}) - \inf_{\theta \in \B{R}^d} R(\theta) \leq 
\frac{\delta^2}{(1 - \delta)(1 - \delta^2)} \breve{N}(\wh{\theta}, 1). 
$$
\end{prop}
\begin{cor} 
With probability at least $1 - \epsilon$, for any $\theta_* 
\in \arg \min_{\theta \in \B{R}^d} R(\theta)$, 
\begin{multline*}
\frac{\breve{N}(\wh{\theta} - \theta_*, 0)}{1 + \delta} \leq 
N(\wh{\theta} - \theta_*, 0) = R(\wh{\theta}) - R(\theta_*) 
\leq \frac{\delta^2}{(1 - \delta)(1 - \delta^2)} \breve{N}(\wh{\theta}, 1) 
\\ \leq \frac{\delta^2}{(1 - \delta)(1 - \delta^2)} \breve{N}(\theta_*, 1) 
\leq \frac{\delta^2}{(1 - \delta)^2} R(\theta_*) \leq 
\frac{\delta^2}{(1 - \delta)^2} R(0) \\ = \frac{\delta^2}{(1 - \delta)^2} 
\B{E} \bigl( Y^2 \bigr).
\end{multline*}
This gives an observable non-asymptotic confidence region for $\theta_*$, 
defined by the equation 
$$
\breve{N}(\wh{\theta} - \theta_*, 0) \leq \frac{\delta^2}{(1 - \delta)^2} 
\breve{N}(\wh{\theta}, 1). 
$$
This also provides a simple non-asymptotic bound for the excess risk 
$$
R(\wh{\theta}) - R(\theta_*) \leq \frac{\delta^2}{(1 - \delta)^2} 
R(\theta_*) \leq \frac{\delta^2}{(1 - \delta)^2} \B{E} 
\bigl( Y^2 \bigr).
$$
As a consequence, in the case when condition \myeq{eq6} is satisfied 
and $Y = f(X) + W$ where $W$ is independent 
of $X$, centered and $\B{E}(W^2) \leq \sigma^2$, 
with probability at least $1 - \epsilon$, 
$$
R(\wh{\theta}) - R(\theta_*) \leq \frac{\delta^2}{(1 - \delta)^2} 
\Bigl\{ \sigma^2 + 
\B{E} \bigl[ f(X)^2 \bigr] \Bigr\}.
$$
\end{cor}
\begin{proof}
We take advantage of the fact that 
not only the derivatives of $\theta \mapsto R(\theta)$ and $ \theta   
\mapsto \breve{N}(\theta, 1)$ vanish at $\theta_*$ and $\wh{\theta}$,
respectively, but also their symmetric finite differences 
$R(\theta + \theta') - R(\theta - \theta')$ and $\breve{N} \bigl( 
\theta + \theta', 1 \bigr) - \breve{N} \bigl( \theta - \theta', 1 
\bigr)$.  

More precisely, using the Cauchy-Schwarz inequality (as well 
as its case of equality), the fact that $\B{E} \bigl[ 
\bigl( \langle \theta_*, X \rangle - Y  \bigr) \langle \theta', X \rangle 
\bigr] = 0$ for any $\theta' \in \B{R}^d$, and equation 
\eqref{eq6}, we obtain the 
following chain of inequalities, that holds with probability 
at least $1 - \epsilon$, for any (possibly random) positive number $a$: 
\begin{multline*}
R(\wh{\theta}) - R(\theta_*) = \B{E} \bigl( \langle \wh{\theta}
- \theta_*, X \rangle^2 \bigr) \\ 
= \sup \Bigl\{ a^{-1} \B{E} \bigl( \langle \wh{\theta} - \theta_*, X \rangle 
\langle \theta', X \rangle \bigr) \, : \, \theta' \in \B{R}^d, 
\B{E} \bigl( \langle \theta', X \rangle^2 \bigr) \leq a^2 \Bigr\}^2 \\
= \sup \Bigl\{ a^{-1} \B{E} \bigl[ \bigl( \langle \wh{\theta}, X \rangle - Y
\bigr) \langle \theta', X \rangle \bigr] \, | \, \theta' \in \B{R}^d, 
\B{E} \bigl( \langle \theta', X \rangle^2 \bigr) \leq a^2 \Bigr\}^2 \\
= \sup \Bigl\{ \frac{1}{4 a} \Bigl[ 
\B{E} \Bigl( \bigl( \langle \wh{\theta} + \theta', X \rangle 
- Y \bigr)^2 \Bigr) 
- \B{E} \Bigl( \bigl( \langle \wh{\theta} - \theta', X \rangle   
- Y \bigr)^2 \Bigr) \Bigr] \, \\ \shoveright{: \, \theta' \in \B{R}^d, 
\B{E} \bigl( \langle \theta', X \rangle^2 \bigr) \leq a^2 \Bigr\}^2} \\ 
\leq \sup \biggl\{ \frac{1}{4a} \biggl[ \frac{\breve{N} \bigl( \wh{\theta} 
+ \theta', 1 \bigr)}{1 - \delta}  - 
\frac{\breve{N} \bigl( \wh{\theta} - \theta', 1 \bigr)}{ 
1 + \delta}  \biggr] \\ \shoveright{\, : \, 
\theta' \in \B{R}^d, \breve{N}(\theta', 0) \leq a^2 (1 + \delta) 
\biggr\}^{2}} \\ 
\shoveright{= \sup \biggl\{ \frac{\delta}{2a(1 - \delta^2)} \Bigl[ \breve{N} \bigl( \wh{\theta}, 1 
\bigr) + \breve{N} \bigl( \theta', 0 \bigr) \Bigr] \, : \, \theta' \in \B{R}^d, 
\breve{N} \bigl( \theta', 0 \bigr) \leq a^2 (1 + \delta) \biggr\}^2} \\ 
= \frac{\delta^2}{4 ( 1 - \delta^2)^2} \biggl[  
\frac{\breve{N} \bigl( \wh{\theta}, 1 \bigr)}{a}  + a ( 1 + \delta) \biggr]^2 
\end{multline*}
Taking the optimal value $\ds a = \sqrt{\frac{\breve{N} \bigl( 
\wh{\theta}, 1 \bigr)}{1 + \delta}}$, we obtain as desired that with 
probability at least $1 - \epsilon$, 
\[ 
R(\wh{\theta}) - R(\theta_*) \leq \frac{\delta^2}{(1 - \delta)(1 - \delta^2)} \breve{N} 
\bigl( \wh{\theta}, 1 \bigr).  
\]  
The statements made in the corollary are obvious consequences of this 
inequality and the definitions. Note that this line of proof would also 
work (with the necessary modifications) in the case when $(X_i,Y_i)$ 
are independent couples of random variables that are not necessarily 
identically distributed. 
\end{proof}

\subsection{A robust least squares estimator} 

To draw the consequences of Proposition \ref{prop:4.1}, 
we have to explain when the required hypothesis, 
expressed by equation \myeq{eq6}, is satisfied. 

For this, we will apply Proposition \vref{prop:1.2}, 
assuming that 
\begin{multline}
\label{eq7}
\sup \, \Bigl\{  
\B{E} \Bigl[ \bigl( \xi Y - \langle \theta, X \rangle \bigr)^4 \Bigr]
\, : \,  \theta \in \B{R}^d, \xi \in \B{R}, \\  
\B{E} \Bigl[ \bigl( \xi Y - \langle \theta, X \rangle \bigr)^2 \Bigr] \leq 1 
\Bigr\} \leq \kappa < \infty.
\end{multline}
When this is satisfied, there is a robust estimator satisfying condition 
\eqref{eq6} with $\delta = 2 \mu / (1 - 4 \mu)$, 
where $\mu$ is as in Proposition 
\vref{prop1.2.3} with $d$ replaced by $d+1$, and is therefore of order
$\bO \Bigl( \sqrt{\kappa \bigl[ d + \log(\epsilon^{-1}) \bigr]/n} \,
\Bigr)$. 

Condition \eqref{eq7} is not very explicit, 
since it bears on the joint distribution of $X$ and $Y$. 
It may be more instructive to replace it by separate kurtosis assumptions 
bearing on $Y - \langle \theta_*, X \rangle$ and on $\langle \theta, X \rangle$, 
where $ \theta_* \in \arg \min_{\theta \in \B{R}^d} 
\B{E} \bigl[ \bigl( Y - \langle \theta, X \rangle \bigr)^2 \bigr]$.
This is what the following lemma does. 
\begin{lemma}
\label{lemma:4.3}
Let us define 
\begin{align*}
\kappa & = \sup \, \Bigl\{  
\B{E} \bigl[ \bigl( \xi Y - \langle \theta, X \rangle \bigr)^4 \bigr] \, : \,
\theta \in \B{R}^d, \xi \in \B{R}, \;     
\B{E} \bigl[ \bigl( \xi Y - \langle \theta, X \rangle \bigr)^2 \bigr] \leq 1 
\Bigr\}, 
\\  
\kappa_1 & = \sup \, \Bigl\{ \B{E} \bigl( \langle \theta, X \rangle^4
\bigr) \, : \, \theta \in \B{R}^d, \;  
\B{E} \bigl( \langle \theta, X \rangle^2 \bigr) \leq 1 \Bigr\}, 
\\ 
\kappa_2 & = 
\begin{cases} 
\ds \frac{ \B{E} \bigl[ \bigl( Y - \langle \theta_*, X \rangle 
\bigr)^4 \bigr]}{ \B{E} \bigl[ \bigl( Y - \langle \theta_*, X \rangle 
\bigr)^2 \bigr]^2}, &  
\B{E} \bigl[ \bigl( Y - \langle \theta_*, X \rangle \bigr)^2 \bigr] > 0, 
\\ 0, & \text{otherwise}.
\end{cases}
\end{align*}
Those three kurtosis coefficients are linked together by the relation 
\[ 
\sqrt{\kappa} \leq \sqrt{\kappa_1} + \sqrt{\kappa_2}. 
\] 
\end{lemma}
\begin{proof}
Using the triangular inequality in $\B{L}^4$, followed by the definitions 
of $\kappa_1$ and $\kappa_2$ and the Cauchy-Schwarz inequality
in $\B{R}^2$, we see that
\begin{multline*}
\B{E} \bigl[ \bigl( \xi Y - \langle \theta, X \rangle \bigr)^4 
\bigr]^{1/2} = \B{E} \Bigl\{ \bigl[ \xi \bigl( Y - \langle \theta_*, 
X \rangle \bigr) + \langle \xi \theta_* - \theta, 
X \rangle \bigr]^4 \Bigr\}^{1/2} \\ 
\shoveleft{\qquad \leq \Bigl\{ \lvert \xi \rvert \B{E} \bigl[ \bigl( 
Y - \langle \theta_*, X \rangle \bigr)^4 \bigr]^{1/4} + 
\B{E} \bigl( \langle \xi \theta_* - \theta, X \rangle^4 
\bigr)^{1/4} \Bigr\}^2 }\\ 
\shoveleft{\qquad \leq \Bigl\{ \lvert \xi \rvert \kappa_2^{1/4} \B{E} \bigl[ 
\bigl( Y - \langle \theta_*, X \rangle \bigr)^2 \bigr]^{1/2} 
+ \kappa_1^{1/4} \B{E} \bigl( \langle \xi \theta_* - \theta, X 
\rangle^2 \bigr)^{1/2} \Bigr\}^2} \\ 
\shoveleft{\qquad \leq ( \kappa_1^{1/2} + \kappa_2^{1/2} ) \Bigl\{ 
\xi ^2 \B{E} \bigl[ \bigl( Y - \langle \theta_*, X \rangle
\bigr)^2 \bigr] + \B{E} \bigl( \langle \xi \theta_* - \theta, 
X \rangle^2 \bigr) \Bigr\}} \\ 
= ( \kappa_1^{1/2} + \kappa_2^{1/2} ) \B{E} \bigl[ \bigl( \xi Y 
- \langle \theta, X \rangle \bigr)^2 \bigr].
\end{multline*}
\end{proof} \\ 
As a consequence of this lemma and of Proposition \vref{prop:1.2} 
\begin{prop}
Consider the bound 
\begin{multline*}
\mu = \sqrt{\frac{2 \bigl[ \bigl( \sqrt{\kappa_1} + \sqrt{\kappa_2} 
\bigr)^2 - 1 \bigr] \bigl[ \log(\epsilon^{-1}) + 0.73 \, (d+1) \bigr]}{n}
} \\ + \, 6.81 \, \bigl( \sqrt{\kappa_1} + \sqrt{\kappa_2} \bigr) 
\sqrt{\frac{2 (d+1)}{n}}. 
\end{multline*}
There exists a robust quadratic estimator $\breve{N}(\theta, \xi)$   
such that for any estimator 
\[ 
\wh{\theta} \in \arg \min_{\theta \in \B{R}^d} \breve{N}(\theta, 1), 
\] 
for any $\epsilon > 0$ and $n \in \B{N}$ 
satisfying
\begin{multline*}
n > \Biggl[ 20 \bigl( \sqrt{\kappa_1} + \sqrt{\kappa_2} \bigr) \sqrt{d+1}
+  \biggl( \frac{5}{2} + \frac{1}{2 \bigl[ \bigl( \sqrt{\kappa_1} 
+ \sqrt{\kappa_2} \bigr)^2 - 1 \bigr]} \biggr) 
\\ \times \biggl( 2 \bigl[ \bigl( \sqrt{\kappa_1} + \sqrt{\kappa_2} 
\bigr)^2 - 1 
\bigr] \bigl[ \log(\epsilon^{-1}) + 0.73 \, (d+1) \bigr] \biggr)^{1/2} 
\, \Biggr]^2 \\ = \bO \Bigl( 
(\kappa_1 + \kappa_2) \bigl[ 
d + \log(\epsilon^{-1}) \bigr] \Bigr),  
\end{multline*}
with probability at least $1 - 2 \epsilon$, 
\[ 
R(\wh{\theta}) - R(\theta_*) \leq \frac{\delta^2}{(1 - \delta)^2} R(\theta_*), 
\] 
where 
\[ 
\delta^2 = \Biggl( \frac{2 \mu}{1 - 4 \mu} \Biggr)^{2} = \bO \biggl( \frac{(\kappa_1 + \kappa_2) 
\bigl[ d + \log(\epsilon^{-1}) \bigr]   }{n} 
\biggr). 
\] 
\end{prop}

\subsection{Generalization bounds for the empirical risk minimizer} 

Let us now examine the conditions under which 
equation \myeq{eq6} is satisfied by the empirical Gram matrix estimator 
\[ 
\breve{N}(\theta, \xi) = \ov{N}(\theta, \xi) 
= \frac{1}{n} \sum_{i=1}^n \bigl( \xi Y_i - \langle \theta, X_i 
\rangle \bigr)^2. 
\] 
To apply Proposition \vref{prop3.1}, we have to bound 
\begin{equation}
\label{eq:11.2}
R = \max_{i=1, \dots, n} \sup \Bigl\{ \xi Y_i - \langle \theta, X_i \rangle 
\, : \, (\theta, \xi) \in \B{R}^{d+1}, \B{E} 
\bigl[ \bigl( \xi Y - \langle \theta, X \rangle \bigr)^2 \bigr] \leq 1 
\Bigr\}. 
\end{equation}
\begin{lemma}
The above defined quantity satisfies almost surely
\begin{align*}
R^2 & \leq 
\max_{i=1, \dots, n} \biggl\{ \frac{\bigl( Y_i - \langle \theta_*, 
X_i \rangle \bigr)^2}{ \B{E} \bigl[ \bigl( Y - \langle \theta_*, X \rangle
\bigr)^2 \bigr]} + \bigl\lVert G^{-1/2} X_i \rVert^2 \biggr\}, 
\\ & \leq \max_{i=1, \dots, n} \frac{\bigl( Y_i - \langle \theta_*, 
X_i \rangle \bigr)^2}{ \B{E} \bigl[ \bigl( Y - \langle \theta_*, X \rangle
\bigr)^2 \bigr]} + \max_{i = 1, \dots, n} \bigl\lVert G^{-1/2} X_i \rVert^2, 
\end{align*}
with the convention that $0/0 = 0$. 
\end{lemma} 
\begin{proof}
Remark that for any positive constants $a$ and $b$,  
\begin{multline*}
\bigl( \xi Y_i - \langle \theta, X_i \rangle \bigr)^2 
= \bigl[  \xi \bigl( Y_i - \langle \theta_*, X_i \rangle \bigr) a^{-1} a + 
\langle \xi \theta_* - \theta, X_i \rangle b^{-1} b \bigr]^2 
\\ \leq \bigl[ \xi^2 \bigl(Y_i - \langle \theta_*, X_i \rangle \bigr)^2 a^{-2}  
+ \langle \xi \theta_* - \theta, X_i \rangle^2 b^{-2} \bigr] \bigl( a^2 + 
b^2 \bigr).
\end{multline*}
Now take $a^2 = \xi^2 \B{E} \bigl[ \bigl(Y - \langle \theta_*, X \rangle
\bigr)^2 \bigr]$ and $b^2 = \B{E} \bigl( \langle \xi \theta_*
 - \theta, X \rangle^2 \bigr)$, and remark that 
\[ 
a^2 + b^2 = \B{E} \bigl[ \bigl( \xi Y - \langle \theta, X \rangle \bigr)^2 
\bigr]. 
\] 
Notice also that in the case when $b^2 > 0$,  
\[ 
\langle \xi \theta_* - \theta, X_i \rangle^2 b^{-2} 
\leq \lVert G^{1/2} (\xi \theta_* - \theta) \rVert^2 \lVert G^{-1/2} X_i 
\rVert^2 b^{-2} = \lVert G^{-1/2} X_i \rVert^2. 
\] 
Consequently, when $a^2 > 0$ and $b^2 > 0$, 
\begin{multline}
\label{eq:11}
\bigl( \xi Y_i - \langle \theta, X_i \rangle \bigr)^2 \leq 
\biggl( \frac{ \bigl( Y_i - \langle \theta_*, X_i \rangle \bigr)^2}{ 
\B{E} \bigl[ \bigl( Y - \langle \theta_*, X \rangle \bigr)^2 \bigr]} 
\\ + \lVert G^{-1/2} X_i \rVert^2 \biggr) \B{E} \bigl[ 
\bigl( \xi Y - \langle \theta, X \rangle \bigr)^2 \bigr]. 
\end{multline}
In the case when $a^2 = 0$, $Y_i - \langle \theta_*, X_i \rangle = 0$ 
almost surely, and in the case when $b = 0$, $
\langle \xi \theta_* - \theta, X_i \rangle = 0$ almost surely, 
so that in those cases, equation \eqref{eq:11} 
is still satisfied, and the desired result is an easy consequence of 
this inequality.  
\end{proof}

In view of the above lemma, suitable hypotheses 
to obtain a bound for $R$ are that $\B{E} \bigl[ (Y - \langle \theta_*, 
X \rangle \bigr)^2 \bigr] > 0$ and for positive constants 
$p, q \in ]0,1]$, $\alpha_i$ 
and $\eta_i$, $i = 1, 2$,  
\begin{align}
\label{eq:13}
\B{E} \Bigl\{ \exp \Bigl[ 
\frac{\alpha_1}{2} \Bigl( \lVert G^{-1/2} X \rVert^{2p} 
- d^{p} - \eta_1 \Bigr) \Bigr] \Bigr\} & \leq 1 \\ 
\label{eq:14} 
\text{and } 
\B{E} \biggl\{ \exp \biggl[ \frac{\alpha_2}{2} \biggl( 
\frac{ (Y - \langle \theta_* , X \rangle )^{2q}}{
\B{E} \bigl[ (Y - \langle \theta_*, X \rangle)^2 
\bigr]^q} - 1 - \eta_2 \biggr) \biggr] \biggr\} & \leq 1. 
\end{align}
Under those hypotheses, with probability at least $1 - 2 \epsilon$, 
\[ 
R^2 \leq \biggl( d^p + \eta_1 + \frac{2}{\alpha_1} \log \bigl( n / \epsilon \bigr) \biggr)^{1/p} + \biggl( 1 + \eta_2 + \frac{2}{\alpha_2} \log \bigl( n / 
\epsilon \bigr) \biggr)^{1/q}.
\]  
In the case when $\B{E} \bigl[ \bigl(Y - \langle \theta_*, X \rangle \bigr)^2 
\bigr] = 0$, it is easy to see that when \eqref{eq:13} is satisfied, 
with probability at least $1 - \epsilon$ 
\[ 
R^2 \leq \biggl( d^p + \eta_1 + \frac{2}{\alpha_1} \log(n /\epsilon) 
\biggr)^{1/p}.
\]   
Applying Proposition \vref{prop3.1} and Proposition \vref{prop:4.1} 
and its corollary, we obtain
\begin{prop}
\label{prop:4.5}
Assume that the hypotheses expressed by equations \eqref{eq:13} and 
\eqref{eq:14} are 
satisfied. 
Consider
\begin{align*}
\mu & = \biggl( \frac{2 \bigl[ \bigl( \sqrt{\kappa_1} + \sqrt{\kappa_2} 
\bigr)^2 - 1 \bigr]}{
n} \bigl[ \log(\epsilon^{-1}) + 0.73 \, (d + 1) \bigr] \biggr)^{1/2} \\ 
& \qquad + 6.81 \,  
\bigl( \sqrt{\kappa_1} + \sqrt{\kappa_2} \bigr) \sqrt{\frac{2 (d+1)}{n}}, \\
\wh{\delta} & = \frac{\mu}{1 - 2 \mu}, \\ 
\gamma_- & = 
\frac{2 \bigl[ \log(\epsilon^{-1}) + 0.73 \, (d + 1)
\bigr] }{3 \bigl[ (\sqrt{\kappa_1} + \sqrt{\kappa_2})^2  - 1 \bigr] \, n}, \\ 
\gamma_+ & = \gamma_1 R^4 \bigl( 1 + \wh{\delta} \bigr)^2, \\ 
\wh{\gamma}_+ & = 
\frac{2 \bigl[ \log(\epsilon^{-1}) + 0.73 \, (d+1)
\bigr](1 + \wh{\delta})^2 }{3 \bigl[ (\sqrt{\kappa_1} + \sqrt{\kappa_2})^2  - 1 \bigr] \, n} 
\\ & \qquad \times  
\biggl[ \Bigl( d^p + \eta_1 + 2 \alpha_1^{-1} \log(n / \epsilon) 
\Bigr)^{2/p} + \Bigl( 1 + \eta_2 + 2 \alpha_2^{-1} 
\log(n / \epsilon) \Bigr)^{2/q} \biggr]^2, 
\end{align*}
where $\kappa_1$ and $\kappa_2$ are defined in Lemma \vref{lemma:4.3}
and $R$ is defined in \myeq{eq:11.2}.  
Consider any confidence parameter $\epsilon > 0$ and any sample size 
$n$ such that 
\begin{multline*}
n > \Biggl[ 20\, \bigl( \sqrt{\kappa_1} + \sqrt{\kappa_2} \bigr) \sqrt{d+1} 
+ \biggl( \frac{5}{2} + \frac{1}{2 \bigl[ \bigl( \sqrt{\kappa_1} + \sqrt{
\kappa_2} \bigr)^2 - 1 \bigr]} \biggr) \\ 
\times \biggl( 2 \bigl[ \bigl( \sqrt{\kappa_1} 
+ \sqrt{\kappa_2} \bigr)^2 - 1 \bigr] \bigl[ \log(\epsilon^{-1}) 
+ 0.73 \, (d+1) \bigr] \biggr)^{1/2}  \Biggr]^2 
\\ = \bO \Bigl( (\kappa_1 + \kappa_2) \bigl[ 
d + \log(\epsilon^{-1}) \bigr] \Bigr). 
\end{multline*}
With probability at least $1 - 4 \epsilon$, for any $\theta \in \B{R}^d$, 
any $\xi \in \B{R}$, 
\[ 
- \frac{\wh{\delta} + \gamma_-}{1 + \wh{\delta}}  \leq \frac{\ov{N}(\theta, \xi )}{N(\theta, \xi )} 
- 1 \leq \frac{\wh{\delta} + \gamma_+}{(1 - \wh{\delta})_+(1 - \gamma_+)_+} 
\leq 
\frac{\wh{\delta} + \wh{\gamma}_+}{(1 - \wh{\delta})_+(1 - \wh{\gamma}_+)_+} 
\] 
so that in particular 
\[
\biggl\lvert \frac{\ov{N}(\theta, \xi )}{N(\theta, \xi )} - 1 
\biggr\rvert \leq \frac{\wh{\delta} + 
\wh{\gamma}_+}{(1 - \wh{\delta})_+(1 - \wh{\gamma}_+)_+} 
= \bO_{n \rightarrow \infty} \Biggl( \sqrt{\frac{ \bigl( \kappa_1 
+ \kappa2 \bigr) \bigl[ d + \log(\epsilon^{-1}) \bigr]}{n}} \; \Biggr).
\]
As a consequence, the empirical risk minimizer 
\[ 
\wh{\theta} \in \arg \min \sum_{i=1}^n \bigl( Y_i - \langle \theta, X_i \rangle 
\bigr)^2 
\] 
is such that with probability at least $1 - 4 \epsilon$, 
\begin{align*}
\B{E} \bigl[ \bigl( Y - \langle \wh{\theta}, X \rangle \bigr)^2 \bigr] 
- \B{E} \bigl[ \bigl( Y - \langle \theta_* , X \rangle \bigr)^2 \bigr]
\leq  
\frac{\delta^2}{(1 - \delta) (1 - \delta^2)n}  \sum_{i=1}^n 
\bigl( Y_i - \langle \wh{\theta}, X_i \rangle \bigr)^2 \\ 
\leq \frac{\delta^2}{(1 - \delta)^2} \B{E} \bigl[ \bigl( Y - \langle 
\theta_*, X \rangle \bigr)^2 \bigr],  
\end{align*}
where 
\[ 
\delta^2 = \Biggl( 
\frac{\wh{\delta} + \wh{\gamma}_+}{(1 - \wh{\delta})(1 - \wh{\gamma}_+ 
)} \Biggr)^2 = \bO_{n \rightarrow \infty} \biggl( \frac{
(\kappa_1 + \kappa_2) \bigl[ d + \log(\epsilon^{-1}) \bigr]}{n} \biggr) . 
\] 
\end{prop}
We can also apply Proposition \vref{prop:2.4} and work under 
polynomial moment assumptions. Consider the 
Gram matrix 
\[ 
\wt{G} = \B{E} \biggl[ \begin{pmatrix} X \\  - Y \end{pmatrix} 
\bigl( X^{\top}, - Y \bigr) \biggr].  
\]  
Remark that 
\begin{multline*}
\biggl\lVert \wt{G}^{-1/2} \begin{pmatrix} X \\ - Y \end{pmatrix}  
\biggr\rVert^2 = \sup \biggl\{ \bigl( \xi Y - \langle \theta, 
X \rangle \bigr)^2 \\ \, : \, 
(\theta, \xi) \in \B{R}^{d+1},  \B{E} \bigl[ \bigl( \xi Y - \langle \theta, X \rangle
\bigr)^2 \bigr]   \leq 1 \biggr\}  
\\ 
\leq \frac{(Y - \langle \theta_*, X \rangle )^2}{ \B{E} \bigl[ 
\bigl( Y - \langle \theta_*, X \rangle \bigr)^2 \bigr]} 
+ \lVert G^{-1/2} X \rVert^2,  
\end{multline*}
according to equation \myeq{eq:11}.
Therefore, according to the Minkowski inequality in $\B{L}^{p+1}$, 
\begin{multline*}
\B{E} \Biggl( \biggl\lVert \wt{G}^{- 1/2} \begin{pmatrix}X \\ -Y \end{pmatrix} 
\biggr\rVert^{2p+2} \Biggr) \leq 
\Biggl( \frac{ 
\B{E} \bigl[ \bigl( Y - \langle \theta_*, X \rangle \bigr)^{2p+2} \, \bigr]^{1/(p+1)}
}{
\B{E} \bigl[ \bigl( Y - \langle \theta_*, X \rangle \bigr)^2 \bigr]}
\\ + \B{E} \bigl( \lVert G^{-1/2} X \rVert^{2p+2} \, \bigr)^{1/(p+1)} \Biggr)^{p+1}.
\end{multline*}
In the same way, with a change of notation,
\begin{multline*}
\B{E} \Biggl( \biggl\lVert \wt{G}^{- 1/2} \begin{pmatrix}X \\ -Y \end{pmatrix} 
\biggr\rVert^{4p + 4} \Biggr) \leq 
\Biggl( \frac{ 
\B{E} \bigl[ \bigl( Y - \langle \theta_*, X \rangle \bigr)^{4p+4} 
\, \bigr]^{1/(2p+2)}
}{
\B{E} \bigl[ \bigl( Y - \langle \theta_*, X \rangle \bigr)^2 \, \bigr]}
\\ + \B{E} \bigl( \lVert G^{-1/2} X \rVert^{4p+4} \, \bigr)^{1/(2p+2)} \Biggr)^{2p+2}.
\end{multline*}
\begin{prop}
\label{prop:4.7}
Consider some exponent $p \in ]1,2]$. Consider the same hypotheses as in Proposition \vref{prop:4.5}, except that 
instead of conditions \eqref{eq:13} and \myeq{eq:14} we assume now that 
\begin{align*}
\B{E} \bigl[ \bigl( Y - \langle \theta_*, X \rangle \bigr)^{4p+4} \bigr] 
& < \infty, \\ 
\text{and } \B{E} \bigl( \lVert G^{-1/2} X \rVert^{4p+4} \bigr)  
& < \infty.  
\end{align*}
Define 
\begin{multline*}
\wt{\gamma}_+ = \frac{1}{p+1} \biggl( \frac{2 \bigl[ \log(\epsilon^{-1}) + 0.73 \, (d+1) \bigr]}{
3 \bigl[ (\sqrt{\kappa_1} + \sqrt{\kappa_2})^2 - 1 \bigr) n} \biggr)^{p/2} 
(1 + \wh{\delta})^{p+1} \\ 
\times \Biggl[ 
\Biggl( \frac{ 
\B{E} \bigl[ \bigl( Y - \langle \theta_*, X \rangle \bigr)^{2p+2} \, \bigr]^{1/(p+1)}
}{
\B{E} \bigl[ \bigl( Y - \langle \theta_*, X \rangle \bigr)^2 \bigr]}
+ \B{E} \bigl( \lVert G^{-1/2} X \rVert^{2p+2} \, \bigr)^{1/(p+1)} \Biggr)^{p+1}
\\ + 
\frac{1}{\sqrt{n \epsilon}} \Biggl( \frac{ 
\B{E} \bigl[ \bigl( Y - \langle \theta_*, X \rangle \bigr)^{4p+4} \, \bigr]^{1/
(2p+2)}
}{
\B{E} \bigl[ \bigl( Y - \langle \theta_*, X \rangle \bigr)^2 \, \bigr]}
+ \B{E} \bigl( \lVert G^{-1/2} X \rVert^{4p+4} \, \bigr)^{1/(2p+2)} \Biggr)^{p+1} 
\, \Biggr],
\end{multline*}
where $\kappa_1$ and $\kappa_2$ are as in Lemma \vref{lem:2.5}. 
Under the same condition on $n$ and $\epsilon$ as in Proposition 
\ref{prop:4.5}, with probability at least $1 
- 3 \epsilon$, for any $( \theta, \xi ) \in \B{R}^{d+1}$, 
\[ 
\Biggl\lvert \frac{\ov{N}(\theta, \xi)}{N(\theta, \xi)} - 1 
\Biggr\rvert \leq \frac{\wh{\delta} + \wt{\gamma}_+}{(1 - \wh{\delta})_+
(1 - \wt{\gamma}_+)_+}, 
\] 
so that in particular the empirical risk minimizer is such that 
with probability at least $1 - 3 \epsilon$, 
\begin{align*}
\B{E} \bigl[ \bigl( Y - \langle \wh{\theta}, X \rangle \bigr)^2 \bigr] 
- \B{E} \bigl[ \bigl( Y - \langle \theta_* , X \rangle \bigr)^2 \bigr]
\leq  
\frac{\delta^2}{(1 - \delta) (1 - \delta^2)n}  \sum_{i=1}^n 
\bigl( Y_i - \langle \wh{\theta}, X_i \rangle \bigr)^2 \\ 
\leq \frac{\delta^2}{(1 - \delta)^2} \B{E} \bigl[ \bigl( Y - \langle 
\theta_*, X \rangle \bigr)^2 \bigr],  
\end{align*}
where
\begin{multline*}
\delta^2 = \biggl( \frac{\wh{\delta} + \wt{\gamma}_+}{(1 - \wh{\delta})(1 
- \wt{\gamma}_+)} \biggr)^2 = \bO_{n \rightarrow \infty} 
\bigl( \, \wh{\delta}^{2} \, \bigr) \\ = \bO_{n 
\rightarrow \infty} \biggl( \frac{(\kappa_1 + \kappa_2) 
\bigl[ d + \log(\epsilon^{-1}) \bigr] }{n} \biggr). 
\end{multline*}
\end{prop}

We can also apply Proposition \vref{prop:2.6} to get weaker 
moment assumptions at the price of a worse non-asymptotic bound, 
but still with the same leading term when the 
sample size $n$ goes to infinity. 

\begin{prop}
\label{prop:4.8}
Consider some exponents $p \in ]1,2]$ and $q \in ]1,2[$. Make 
the same hypotheses as in Proposition \vref{prop:4.5}, 
except that instead of conditions \eqref{eq:13} and \myeq{eq:14}, 
we assume that 
\[ 
\B{E} \bigl[ ( Y - \langle \theta_*, X \rangle \bigr)^{2q(p+1)} \bigr] 
< \infty \text{ and } \B{E} \bigl( \lVert G^{-1/2} X \rVert^{2q(p+1)} \bigr) 
< \infty. 
\] 
Define 
\begin{multline*}
\wt{\gamma}_+ = \frac{1}{p+1} \biggl( \frac{ 2 \bigl[ \log(\epsilon^{-1} ) 
+ 0.73 (d + 1) \bigr]}{ \bigl[ (\sqrt{\kappa_1} + \sqrt{\kappa_2})^2 
- 1) n} \biggr)^{p/2} \bigl( 1 + \wh{\delta} \, \bigr)^{p+1} \\ 
\times \Biggl[ \Biggl( \frac{\B{E} \bigl[ \bigl( Y - \langle \theta_*, 
X \rangle \bigr)^{2p + 2} \bigr]^{1/(p+1)}}{\B{E} \bigl[ \bigl( Y 
- \langle \theta_*, X \rangle \bigr)^2 \bigr]}  + 
\B{E} \bigl( \lVert G^{-1/2} X \rVert^{2p + 2} \bigr)^{1/(p+1)} 
\Biggr)^{p+1} \\ 
+ \frac{C_q}{\epsilon^{1/q} n^{1 - 1/q}} \Biggl( \frac{\B{E} \bigl[ 
\bigl( Y - \langle \theta_*, X \rangle \bigr)^{2q(p+1)} \bigr]^{q^{-1}(p+1)^{-1}}}{\B{E} \bigl[ \bigl( Y - \langle \theta_*, X \rangle \bigr)^2 \bigr]} 
\\ 
+ \B{E} \Bigl( \lVert G^{-1/2} X \rVert^{2q(p+1)} \Bigr)^{q^{-1}(p+1)^{-1}} 
\Biggr)^{p+1} \Biggr],
\end{multline*}
where the constant $C_q$ is defined in Lemma \vref{lem:2.5}
and $\kappa_1$ and $\kappa_2$ are defined in Lemma \vref{lemma:4.3}. 
Under the same conditions on $n$ and $\epsilon$ as in 
Proposition \vref{prop:4.5}, with probability at least 
$1 - 4 \epsilon$, for any $(\theta, \xi) \in \B{R}^{d+1}$, 
\[ 
\Biggl\lvert \frac{\ov{N}(\theta, \xi)}{N(\theta, \xi)} - 1 \Biggr\rvert 
\leq \frac{\wh{\delta} + \wt{\gamma}_+}{(1 - \wh{\delta})_+ ( 1 - 
\wt{\gamma}_+)_+}, 
\] 
so that the empirical risk minimizer is such that with probability 
at least $1 - 4 \epsilon$, 
\begin{multline*}
\B{E} \bigl[ \bigl( Y - \langle \wh{\theta}, X \rangle
\bigr)^2 \bigr] - \B{E} \bigl[ \bigl( Y - \langle \theta_*, 
X \rangle \bigr)^2 \bigr] \\ \leq 
\frac{\delta^2}{(1 - \delta)_+(1 - \delta^2)_+ n} 
\sum_{i=1}^n \bigl( Y_i - \langle \wh{\theta}, 
X_i \rangle \bigr)^2 \\ 
\leq \frac{\delta^2}{(1 - \delta)^2} \B{E} \bigl[ 
\bigl( Y - \langle \theta_*, X \rangle \bigr)^2 \bigr], 
\end{multline*}
where 
\begin{multline*}
\delta^2 = \Biggl( \frac{\wh{\delta} + \wt{\gamma}_+}{(1 - \wh{\delta})_+ 
(1 - \wt{\gamma}_+)_+} \Biggr)^2 = \bO_{n \rightarrow \infty} 
\bigl( \, \wh{\delta}^2 \, \bigr) 
\\ = \bO_{n \rightarrow \infty} \Biggl( 
\frac{(\kappa_1 + \kappa_2) \bigl[ d + \log(\epsilon^{-1}) \bigr]}{n} 
\Biggr). 
\end{multline*}
\end{prop}

Remark that in these last two propositions, the hypotheses on the design $X$ 
and the noise $Y - \langle \theta_*, X \rangle$  
are weaker than in Proposition \ref{prop:4.5}, since they 
involve only polynomial moment assumptions. As a counterpart,  
the dependence in the confidence parameter $\epsilon$  
is worse, since we have a factor $1 / \sqrt{n \epsilon}$
in Proposition \ref{prop:4.7} and  
$\epsilon^{-1/q} n^{-(1-1/q)}$ in Proposition \ref{prop:4.8} 
that precludes the use of a confidence level $1 - 3 \epsilon$  
(resp. $1 - 4 \epsilon$)
much higher than $1 - 1/n$ (resp. $1 - n^{-(q-1)}$). 
Nevertheless, this factor $1/\sqrt{n \epsilon}$ (resp. 
$\epsilon^{-1/q} n^{-(1 - 1/q)}$) 
appears only in second order terms, so that we still get an asymptotic 
upper bound (expressed here in big $\bO$ notation) where the dependence
on the confidence parameter is proportional to $\log(\epsilon^{-1})$.  

The results in this section are stronger than those obtained in 
\cite{AuCat10a}. In particular, for the robust estimator, we are 
not limited to considering the minimization in $\theta$ in 
a bounded subset $\Theta$ of $\B{R}^{d}$, as in Theorem 3.1 of 
\cite{AuCat10a}, where some of the constants involved in the bound 
depend on the fact that $\Theta$ is bounded. Also, we provide 
non-asymptotic bounds for the empirical risk minimizer, 
whereas in \cite{AuCat10a} the corresponding result, Theorem 2.1, 
is asymptotic, since it is satisfied only for large enough 
sample sizes $n$, without an explicit condition on $n$.    
Moreover the constant $B$ appearing in Theorem 2.2 
of \cite{AuCat10a}  
is in fact necessarily dependent on the dimension 
and more precisely not smaller than $d$, as mentioned to 
us by Guillaume L\'ecu\'e, whom we are grateful for 
pointing out this mistake (the comments we made 
after Theorem 2.2 of \cite{AuCat10a} about the size 
of the constant $B$ are in fact 
false, we apologize for this error).

It is also interesting to compare our results with those of \cite{Lecue}. 
In particular, it is relevant to compare Proposition \ref{prop:4.8} above 
with Theorem 1.3 of \cite{Lecue}. Our hypotheses are only 
slightly stronger than those of \cite{Lecue}, since we require 
that $\lVert G^{-1/2} X \rVert$ belongs to a little more than $\B{L}^4$, 
namely to $\B{L}^{2q(p+1)}$, where $q > 1$ and $p > 1$ can be taken 
arbitrarily close to $1$, so that $2q(p+1)$ can also be made as close 
to $4$ as desired, if one is willing to accept a larger second order term 
$\wt{\gamma}_+$. Indeed, the coefficient $\theta_0$ appearing in 
\cite{Lecue} is related to the kurtosis coefficient $\kappa_1$ 
by the relation $\theta_0 = \kappa_1^{1/4}$. Moreover,  
as proved in Lemma \vref{lem1.7} 
\[ 
\B{E} \bigl( \lVert G^{-1/2} X \rVert^4 \bigr) 
\leq d \kappa_1,  
\]  
so that when $\kappa_1$ (or $\theta_0$) is finite, 
$\lVert G^{-1/2} X \rVert$ belongs to $\B{L}^4$. 
So we ask for a little higher moment, and we get as a reward a 
better bound, since our bound still has a subexponential first 
order term (meaning that the dependence 
in the confidence parameter $\epsilon$ 
is in $\log(\epsilon^{-1})$), whereas the bound obtained in \cite{Lecue} 
is 
\[ 
\bO_{n \rightarrow \infty} \Biggl[ \frac{d \kappa_1^{3} \kappa_2^{1/2} }{n \epsilon} \B{E} \bigl[ 
\bigl( Y - \langle \theta_*, X \rangle \bigr)^2 \bigr] \Biggr], 
\] 
when cast into our notation. For instance, if we choose
$p = q = 3/2$ in our bound, we have to assume that $\B{E} 
\bigl( \lVert G^{-1/2} X \rVert^{15/2} \bigr) < \infty$, 
but we can take $\epsilon = n^{-1/2}$ and still 
have a second order term in $n^{-3/2}$, whereas the first order term 
is in $n^{-1}$. As a comparison, if we choose $\epsilon = n^{-1/2}$, 
the bound from $\cite{Lecue}$ is no more in $n^{-1}$, but in $n^{-1/2}$. 

We would like also to remark that in some situations, $\kappa_1$, 
that is equal to $\theta_0^4$ in \cite{Lecue}, 
may in fact depend on $d$. This is for instance the case in 
uniform histogram regression, where 
\begin{equation}
\label{eq:16}
X = \bigl( X_k, k=1, \dots d \bigr) = \bigl[ g_k(U), k=1, \dots, d \bigr]
\end{equation}
the random variable $U$ being distributed according to the uniform 
distribution in the unit interval $[0,1]$ 
and the functions $g_k$ being defined as
\begin{equation}
\label{eq:17}
g_k(u) = \B{1} \bigl[ \, (k-1)/d \leq u < k/d \, \bigr]. 
\end{equation}
In this case $\B{E} \bigl[ g_k (U)^p \bigr] = 1/d$ for all 
values of the exponent $p$, so that necessarily $\kappa_1 
\geq \B{E} \bigl[ g_k(U)^4 \bigr] / \B{E} \bigl[ g_k(U)^2 \bigr]^2 
= d$ (in fact, it is easy to see that $\kappa_1 = d$ in this case). The same unfavourable scaling appears in all 
local bases of the wavelet type. Indeed, in the context
of functional regression of $Y$ by $\sum_{k=1}^d \theta_k g_k(U)$, 
where $U$ is uniform in the unit interval, rescaling a regression 
function $g$ by a scale factor $\lambda$ will impact 
its $\B{L}^4 / \B{L}^2$ ratio according to the formula 
\[ 
\frac{ \B{E} \bigl[ g(\lambda U)^4 \bigr] }{
\B{E} \bigr[ g( \lambda U)^2 \bigr]^2} = \lambda 
\frac{ \B{E} \bigl[ g(U)^4 \bigr] }{
\B{E} \bigr[ g(U)^2 \bigr]^2}. 
\] 

\subsection{Some lower bound for the empirical risk minimizer}

We will show in this section that the order of magnitude of the 
previous upper bound cannot be improved in the worst case, 
except for the values of the numerical constants. 
\begin{prop}
\label{prop:4.9}
Consider a sample $(\wt{X}_1, \wt{Y}_1), \dots, (\wt{X}_n, \wt{Y}_n)$ 
made of $n$ 
independent copies of the couple of random variables $(\wt{X},\wt{Y}) \in \B{R}^d \times 
\B{R}$. 
Assume that for some $\delta \in ]0,1[$, some $\epsilon \in ]0,1/2[$ and some $n_{\epsilon} 
\in \B{N}$, for any $n \geq n_{\epsilon}$, with probability at least 
$1 - 2 \epsilon$, for any $\theta \in \B{R}^d$, 
\begin{equation}
\label{eq:18}
\biggl\lvert \frac{\frac{1}{n} 
\sum_{i=1}^n \langle \theta, \wt{X}_i \rangle^2}{ \B{E} 
\bigl( \langle \theta, \wt{X} \rangle^2 \bigr)} - 1 \biggr\rvert  
\leq \delta, 
\end{equation}
where by convention $0/0 = 1$ and $z / 0 = \infty$ for $z > 0$. 
Assume that 
\[ 
\wt{Y} = \langle \theta_*, \wt{X} \rangle + \eta,
\]
where 
$\B{P}_{\eta 
| \wt{X}} = \C{N}(0, \sigma^2)$, with $\sigma > 0$, 
 so that the noise $\eta$ is a Gaussian noise 
independent from $\wt{X}$. Assume also that 
\[ 
G = \B{E} \bigl( \wt{X} \wt{X}^{\top} \bigr)
\] 
is of full rank $d$. 
Consider any empirical risk minimizer 
\[ 
\wt{\theta} \in \arg \min_{\theta \in \B{R}^d} \sum_{i=1}^n 
\bigl( \wt{Y}_i - \langle \theta, \wt{X}_i \rangle \bigr)^2. 
\] 
(There may be more than one when $\frac{1}{n} \sum_{i=1}^n \wt{X}_i 
\wt{X}_i^{\top}$ is not of full rank.)

For any $n \geq n_{\epsilon}$, with probability at least $\ds 1 - 3 \epsilon$,
\[ 
\B{E} \bigl[ \bigl( \wt{Y} - \langle \wt{\theta}, \wt{X} 
\rangle \bigr)^2 \bigr] 
- \B{E} \bigl[ \bigl( \wt{Y} - \langle \theta_*, \wt{X} \rangle \bigr)^2 \bigr] 
\geq \frac{\bigl[ d \log(2) - 2 \log(\epsilon^{-1}) \bigr]
\sigma^2 }{(1 + \delta) \, n}.  
\] 
Consider now a Bernoulli random variable $\xi$ of parameter $p 
\in ]0, 1]$, 
independent from $(\wt{X}, \wt{Y})$, so that 
\[ 
\B{P} \bigl( \xi = 1 \, | \, (\wt{X}, \wt{Y}) \bigr)   
= 1 - \B{P} \bigl( \xi = 0 \, | \, (\wt{X}, \wt{Y}) \bigr) = p. 
\] 
Define the censored couple of random variables 
\[ 
\bigl( X, Y \bigr) = \bigl( \xi \wt{X}, \xi \wt{Y} \bigr). 
\] 
Consider a sample $(X_1, Y_1), \dots, (X_n, Y_n)$
made of $n$ independent copies of $(X, Y)$. 
Consider any empirical risk minimizer of the censored data 
\[ 
\wh{\theta} \in \arg \min_{\theta \in \B{R}^d} 
\sum_{i=1}^n \bigl( Y_i - \langle \theta, X_i \rangle \bigr)^2.
\] 
Define $\kappa_1$ and $\kappa_2$ as in Lemma \vref{lemma:4.3}.  
Define in the same way
\begin{align*}
\wt{\kappa}_1 & = \sup \, \Bigl\{  
\B{E} \bigl( \langle \theta, \wt{X} \rangle^4 \bigr) \, : \,
\theta \in \B{R}^d, \; 
\B{E} \bigl( \langle \theta, \wt{X} \rangle^2 \bigr) \leq 1 \Bigr\}, \\ 
\wt{\kappa}_2 & = 
\frac{ \B{E} \bigl[ \bigl( \wt{Y} - 
\langle \theta_*, \wt{X} \rangle \bigr)^4 \bigr]}{ 
\B{E} \bigl[ \bigl( \wt{Y} - 
\langle \theta_*, \wt{X} \rangle \bigr)^2 \bigr]^2} = 3 \text{ (since } \eta 
\text{ is Gaussian)}.  
\end{align*}
Remark that $\kappa_1 = \wt{\kappa}_1 / p$ and $\kappa_2 = \wt{\kappa}_2 / p$, 
so that in particular $p^{-1} = (\kappa_1 + \kappa_2)/(\wt{\kappa}_1 + \wt{\kappa}_2)$. 
For any 
\[
n \geq \frac{4 (\kappa_1 + \kappa_2) \max \bigl\{ n_\epsilon, 2 \log(
\epsilon^{-1}) \bigr\}}{\wt{\kappa}_1 + \wt{\kappa}_2}
\]
with probability at least $1 - 5 \epsilon$, 
\begin{multline}
\label{eq:19.2} 
\B{E} \bigl[ \bigl( Y - \langle \wh{\theta}, X \rangle 
\bigr)^2 \bigr] 
- \B{E} \bigl[ \bigl( Y - \langle \theta_*, X \rangle 
\bigr)^2 \bigr] \\ \geq 
\frac{ 4 \bigl[ 
d \log(2) - 2 \log(\epsilon^{-1}) \bigr] (\kappa_1 + \kappa_2)}{7 
(1 + \delta) (\wt{\kappa}_1 + \wt{\kappa}_2) \, n} 
\B{E} \bigl[ \bigl( Y - \langle \theta_*, 
X \rangle \bigr)^2 \bigr]. 
\end{multline}
Therefore, the first order term dependence in $\kappa_1 + \kappa_2$ 
in Propositions \vref{prop:4.7} and \ref{prop:4.8} cannot be 
improved without further assumptions, since by varying the value 
of $p^{-1}$, we can make $\kappa_1 + \kappa_2$ arbitrarily large 
in this lower bound.  

In the case when $\wt{X}$ is a Gaussian vector, 
$\wt{\kappa}_1 + \wt{\kappa}_2 = 6$ and we can take 
for example   
\[ 
n_{\epsilon} = 8.4 \bigl[0.73 \, d + \log(\epsilon^{-1}) \bigr] 
\bigl[ 1.4 \, d + 4 \log(\epsilon^{-1}) + 26.3 \bigr]^2. 
\] 
When $n \geq n_{\epsilon}$, we can deduce from Proposition 
\vref{prop:2.3} that equation \myeq{eq:18} holds with 
$\delta = 8 / 9$. 

So in the case of a censored Gaussian design, for any $\ds n \geq 
\frac{2}{3} (\kappa_1 + \kappa_2) n_{\epsilon}$, 
with probability at least $1 - 5 \epsilon$, 
\begin{multline*}
\B{E} \bigl[ \bigl( Y - \langle \wh{\theta}, X \rangle \bigr)^2 \bigr] 
- \B{E} \bigl[ \bigl( Y - \langle \theta_*. X \rangle \bigr)^2 \bigr] 
\\ \geq \frac{35}{1000} \times \frac{\bigl[ d - 3 \log(\epsilon^{-1}) \bigr] 
(\kappa_1 + \kappa_2)}{n} \B{E} \bigl[ \bigl( Y - \langle \theta_*, 
X \rangle \bigr)^2 \bigr]. 
\end{multline*}
\end{prop}
Note that the result given in the end 
of this proposition 
is not the best possible and could be 
improved by using PAC-Bayes bounds specific to the Gaussian case, 
resulting in a smaller $n_\epsilon$ with a better 
dependence in the dimension $d$. 
However, since we will give another lower bound
further on, we thought it wiser
not to delve into these technicalities. 
The proof of this proposition is given in the appendix. 

\subsection{Exact convergence rate for the empirical risk minimizer}

The results of the previous section show that the convergence 
speed of the empirical risk mimimizer cannot be $\bO_{n 
\rightarrow n} \Bigl( R(\theta_*) d / n \Bigr)$ in 
the worst case. In this section, we clarify the 
situation by showing under mild conditions that 
\[
\B{E} \Bigl[ \min \bigl\{ R(\wh{\theta}) - R(\theta_*), C \exp \bigl( 
n^{2-q} \bigr) \bigr\} \Bigr] \underset{n \rightarrow \infty}{\sim} C / n,   
\]
where we recall that 
$R(\theta) = \B{E} \bigl[ \bigl( Y - \langle \theta, X \rangle \bigr)^2 
\bigr]$ and where  
the exact constant $C$ is given by  
\[ 
C = \B{E} \Bigl[ \bigl( Y - \langle \theta_*, X \rangle \bigr)^2 
\bigl\lVert G^{-1/2} X \bigr\rVert^2 \Bigr]. 
\] 
The form of this constant shows that we get a $R(\theta_*) \, d / n$ 
convergence rate when the noise $Y - \langle \theta_*, X \rangle$
is independent from $X$ (and $G$ is of full rank), and that it can be larger 
or smaller otherwise. 
\begin{prop}
\label{prop:4.10} 
Consider a sample $(X_1, Y_1), \dots, (X_n, Y_n)$ made of 
$n$ independent copies of the couple of random variables 
$(X, Y) \in \B{R}^d \times \B{R}$. 
Let $G = \B{E} \bigl( X X^{\top} \bigr)$ be the Gram matrix of 
the design $X$ and $G^{-1}$ its pseudo-inverse. Assume that
\begin{align*}
\kappa & = \sup \Bigl\{ \B{E} \bigl( \langle \theta, X \rangle^4 \bigr) 
\, : \, \theta \in \B{R}^d, \B{E} \bigl( \langle \theta, X \rangle^2 
\bigr) \leq 1 \Bigr\} < \infty \\ 
\text{and that } \qquad C & = \B{E} \Bigl[ \bigl(Y - \langle \theta_*, X \rangle 
\bigr)^2 \, \bigl\lVert G^{-1/2} X \bigr\rVert^2 \Bigr] < \infty, 
\end{align*}
where
\[ 
\theta_* \in \arg\min_{\theta \in \B{R}^d} \B{E} \bigl[ 
\bigl( Y - \langle \theta, X \rangle \bigr)^2 \bigr]
\] 
is some optimal regression parameter. 
Define $\mu$ as in Proposition \vref{prop1.2.3} 
and $\gamma_-$ and $\wh{\delta}$ as in Proposition \vref{prop3.1}. 
Consider any empirical risk minimizer
\[ 
\wh{\theta} \in \arg \min_{\theta \in \B{R}^d} \sum_{i=1}^n 
\bigl( Y_i - \langle \theta, X_i \rangle \bigr)^2. 
\] 
For any $n \geq n_\epsilon 
= \bO \bigl( \kappa \bigl[ d 
+ \log(\epsilon^{-1}) \bigr] \bigr)$ 
given by equation \myeq{eq2.2}, 
there is an event $\Omega$ of probability at least 
$1 - \epsilon$, such that the 
excess risk conditional to $\Omega$ satisfies  
\begin{multline*}
\frac{n}{C} \, \B{E} \Bigl[ R(\wh{\theta} \, ) - R(\theta_*) 
\, \big| \, \Omega
\Bigr] \leq \frac{(1 + \wh{\delta})^2}{(1 - \gamma_-)^2 \, \B{P}(\Omega) \, 
} 
\\ \leq 
\Biggl( 1 + \bO \Biggl( \sqrt{
\frac{\kappa \bigl[ d + \log \bigl(\epsilon^{-1} \bigr) \bigr]}{n}} 
\; \Biggr) \Biggr) 
(1 - \epsilon)^{-1}.
\end{multline*}
Consequently, for any $\epsilon$ and $n$ satisfying equation \myeq{eq2.2}, 
for any positive constant $M \in \B{R}_+$,
\[ 
\B{E} \Bigl[ \min \bigl\{ R ( \wh{\theta} \, ) - R(\theta_*), M \bigr\} 
\Bigr] \leq \frac{(1 + \wh{\delta})^2 \, C }{(1 - \gamma_-)^2 \, n} 
+ \epsilon M. 
\] 
Taking $\ds \epsilon = \frac{C}{Mn^2}$, we see that equation 
\myeq{eq2.2} is satisfied for $n > \bO 
\Bigl( \kappa \Bigl[ d + \log \bigl[ \kappa M / C \bigr] \Bigr] \Bigr)$, 
and that  
\[
\frac{n}{C} \, \B{E} \Bigl[ \min \bigl\{ R(\wh{\theta}) - R(\theta_*), M \bigr\} 
\Bigr] \leq 
\Biggl( 1 + \bO \Biggl( \sqrt{
\frac{\kappa \bigl[ d + \log \bigl( Mn/C \bigr) \bigr]}{n}} 
\; \Biggr) \Biggr).
\]

Assume now moreover that 
$Y = \langle \theta_*, X \rangle + \eta$, where $\theta_* \in \B{R}^d$
and $\eta$ is a random variable independent from $X$ such that
$\B{E} ( \eta ) = 0$ and $0 < \B{E}( \eta^2 ) = \sigma^2 < \infty$.  
In this case, for $n \geq \bO 
\Bigl( \kappa \bigl[ d + \log \bigl( \kappa M / \sigma^2 \bigr) \bigr] \Bigr)$,  
\begin{multline*}
\B{E} \Bigl[ \min \bigl\{ R(\wh{\theta} \, ) - R(\theta_*), M \bigr\} \Bigr] 
\\ \leq \Biggl( 1 + \bO 
\Biggl( \sqrt{ \frac{\kappa \bigl[ d + \log \bigl( M n / \sigma^2 \bigr) 
\bigr]}{n}} 
\; \Biggr) \Biggr) \frac{\sigma^2 d}{n}.  
\end{multline*}
Moreover, when $n \geq \bO \Bigl( \kappa \bigl[ d + \log(\epsilon^{-1}) 
\bigr] \Bigr)$ 
satisfies equation \myeq{eq2.2}, with probability at least $1 - \epsilon$, 
\begin{multline*}
\B{E} \bigl[ R(\wh{\theta}\, )  - R(\theta_*) \, | \, X_1, \dots, X_n \bigr] 
\leq \frac{(1 + \wh{\delta}) \, d \, \sigma^2 }{(1 - \gamma_-) \, n} 
\\ = \Biggl( 1 + \bO 
\Biggl( \sqrt{\frac{\kappa \bigl[ d + 
\log(\epsilon^{-1}) \bigr]}{n}} \; \Biggr) \Biggr) \frac{d \, \sigma^2}{n}. 
\end{multline*}
\end{prop}
Remark that the last statement of this proposition is about the random design risk 
\[
R(\theta) = \B{E} \bigl[ \bigl( Y - \langle \theta, X \rangle \bigr)^2
\bigr]
\]  and not about the weaker fixed design risk 
\[ 
\B{E} \Biggl( \frac{1}{n} \sum_{i=1}^n \bigl( Y_i - \langle 
\theta, X_i \rangle \bigr)^2 \, \Bigl| \, X_1, \dots, X_n \Biggr). 
\] 
The proof of Proposition \ref{prop:4.10} is presented in the appendix. 

The next proposition states an upper bound with large probability
when the noise is Gaussian and independent from the design $X$. 
\begin{prop}
Make the same assumptions as in the end of the previous proposition. Assume 
moreover that the noise $\eta$ is Gaussian (but not necessarily the 
design $X$). 
For any $n > n_\epsilon 
= \bO \bigl( \kappa \bigl[ d 
+ \log(\epsilon^{-1}) \bigr] \bigr)$ 
given by equation \myeq{eq2.2}, 
with probability at least $1 - 2 \epsilon$, 
\begin{multline*}
R\bigl( \wh{\theta} \, \bigr) - R(\theta_*) \leq 
\frac{(1 + \wh{\delta}) \, \sigma^2}{ (1 - \gamma_-) \, n}  
\bigl[ 2 \, d \, \log(2) + 4 \log(\epsilon^{-1}) \bigr] 
\\ = \Biggl( 1 + \bO \Biggl( \sqrt{
\frac{\kappa \bigl[ d + \log \bigl(\epsilon^{-1}) \bigr]}{n}} 
\; \Biggr) \Biggr) \frac{\sigma^2 \bigl[ 2 \, d \log(2) + 4 \log(\epsilon^{-1}) \bigr]}{n},  
\end{multline*}
where it is interesting to remind that $2 \log(2) \leq 1.4$.
\end{prop}
\begin{proof}
Using the same notation as in the proof of Proposition \vref{prop:4.10}, 
\[
\B{P}_{\ds \ov{G}^{-1} W | X_1, \dots, X_n} = \C{N} \biggl(
0, \frac{\sigma^2}{n} \ov{G}^{-1} \ov{G} \biggr),
\]
where $\ov{G}^{-1} \ov{G}$ is the orthogonal projection on $\IM(\ov{G})$.  
Therefore, we can modify the end of the proof of Proposition \ref{prop:4.10}, 
using a Chernoff deviation bound for Gaussian vectors. 
\end{proof}

Now, let us close this section with the corresponding lower 
bound, to get the announced exact convergence rate in expectation 
of the excess risk of the empirical risk minimizer. 

\begin{prop}
\label{prop:4.12}
Consider a sample $(X_1, Y_1), \dots, (X_n, Y_n)$ made of $n$ 
independent copies of the couple of random variables $(X, Y) 
\in \B{R}^d \times \B{R}$. Choose two exponents $p, q 
\in [1,2]$. Assume that 
\begin{align*}
\kappa & = \sup \bigl\{ \B{E} \bigl( \langle \theta, X \rangle^4 \bigr) \, : 
\, \theta \in \B{R}^d, \B{E} \bigl( \langle \theta, X \rangle^2 \bigr) 
\leq 1 \bigr\} < \infty, \\ 
C & = \B{E} \Bigl[ \bigl( Y - \langle \theta_*, X \rangle \bigr)^2 
\bigl\lVert G^{-1/2} X \bigr\rVert^2 \Bigr] < \infty, \\
C & > 0, \\ 
\kappa' & = \frac{\B{E} \Bigl[ \bigl( Y - \langle \theta_*, X \rangle \bigr)^4 
\bigl\lVert G^{-1/2} X \bigr\rVert^4 \Bigr]}{ 
\B{E} \Bigl[ \bigl( Y - \langle \theta_*, X \rangle \bigr)^2 
\bigl\lVert G^{-1/2} X \bigr\rVert^2 \Bigr]^2} < \infty, \\ 
\B{E} & \Bigl( \bigl\lVert G^{-1/2} X \bigr\rVert^{2 q(p+1)} \Bigr) 
 < \infty. 
\end{align*}
Under these hypotheses, we will give a technical meaning, in two 
different ways, to the fact that $C/n$ is the exact 
convergence rate of the excess risk $R(\wh{\theta}) - R(\theta_*)$ 
of the empiricial risk minimizer 
\[ 
\wh{\theta} \in \arg \min_{\theta \in \B{R}^d } \sum_{i=1}^n 
\bigl( Y_i - \langle \theta, X_i \rangle \bigr)^2. 
\] 
Define $\mu$ as in Proposition \vref{prop1.2.3}, $\wh{\delta} 
= \mu / (1 - 2 \mu)$ and define 
\[
\wt{\gamma}_+ = \bO \Biggl( 
\biggl(\frac{\log(\epsilon^{-1}) + d}{ \kappa n} \biggr)^{p/2}
 \Biggl[ \B{E} \bigl( \bigl\lVert G^{-1/2} X \bigr\rVert^{2(p+1)} 
+ \frac{ \B{E} \bigl( \bigl\lVert G^{-1/2} X \bigr\rVert^{2q(p+1)} 
\bigr)^{1/q}}{\epsilon^{1/q} n^{1 - 1/q}} \Biggr] \Biggr)
\]
as in Proposition \vref{prop:2.6}. 

For any $n \geq n_{\epsilon} = 
\bO \bigl( \kappa \bigl[ d + \log(\epsilon^{-1}) \bigr] \bigr)$ 
given by equation \myeq{eq2.2}, 
there is an event $\Omega$ of probability at least $1 - 2 \epsilon$ 
such that 
\begin{multline}
\label{eq:22}
(1 - \wt{\gamma}_+)_+^2 (1 - \wh{\delta})_+^2 \biggl[ 1 - \sqrt{6} \, \biggl( 
1 + \frac{\kappa' - 3}{3 n} \biggr)^{1/2} \epsilon^{1/2}\biggr] \times \frac{C}{n}
\\ \leq 
\B{E} \Bigl( \bigl[ R(\wh{\theta}) - R(\theta_*) \bigr] \B{1}_{\Omega} \Bigr) 
\leq
\B{E} \Bigl( R(\wh{\theta} - R(\theta_*) \, \big| \, \Omega \Bigr)   
\\ \leq \frac{ (1 + \wh{\delta})^2}{(1 - \gamma_-)^2 (1 - 2 \epsilon)} 
\times \frac{C}{n}. 
\end{multline}
Remark that when $p > 1$, this gives
\begin{align*}
\Biggl\lvert \, 
\frac{n}{C} \, \B{E} \Bigl( R(\wh{\theta}) - R(\theta_*) \, \big| \, 
\Omega \Bigr) - 1 \, \Biggr\rvert & \leq \bO_{n \rightarrow \infty} \Biggl( \sqrt{ \frac{\kappa \bigl[ 
d + \log(\epsilon^{-1}) \bigr]}{n}} + \epsilon^{1/2} \Biggr), \\ 
\text{and } \Biggl\lvert \, \frac{n}{C} \, \B{E} 
\Bigl( \bigl[ R(\wh{\theta}) - R(\theta_*) \bigr]
\B{1}_{\Omega} \Bigr) - 1 \, \Biggr\rvert & \leq \bO_{n \rightarrow \infty} \Biggl( \sqrt{ \frac{\kappa \bigl[ 
d + \log(\epsilon^{-1}) \bigr]}{n}} + \epsilon^{1/2} \Biggr). 
\end{align*}
For any $M > 0$, any $\epsilon$ satisfying equation \myeq{eq2.2}, 
\begin{multline}
\label{eq:23}
\frac{n}{C} \, \B{E} \Bigl[ \min \bigl\{ R(\wh{\theta}) - R(\theta_*), M \bigr\} 
\Bigr] \\ \geq (1 - \wh{\delta})_+^2 (1 - \wt{\gamma}_+)_+^2 \Biggl[ 1 - 
\frac{3 C}{4 M n} \biggl( 1 + \frac{\kappa' - 3}{3n} \biggr)
\\ - \sqrt{6} \biggl( 1 + \frac{\kappa' - 3}{3n} \biggr)^{1/2} \epsilon^{1/2} 
\Biggr].   
\end{multline}
If we assume that $q > 1$ 
and bind $\epsilon$ and $n$ by the relation $\epsilon = n^{-(q-1)}$, 
equation \myeq{eq2.2} is satisfied when $n \geq \bO \bigl( \kappa \bigl[ 
d + \log(\kappa) \bigr] \bigr)$ and  
\[ 
\wt{\gamma}_+ = \bO \Biggl( \biggl( \frac{\log(n) + d}{\kappa n} \biggr)^{p/2} 
\B{E} \Bigl( \bigl\lVert G^{-1/2} X \rVert^{2q(p+1)} \Bigr)^{1/q} \Biggr), 
\] 
so that 
\begin{multline*} 
\frac{n}{C} \B{E} \Bigl[ \min \bigl\{ R(\wh{\theta}) - R(\theta_*), 
M \bigr\} \Bigr] \\ \geq 1 - \bO \Biggl( 
\sqrt{\frac{ \kappa \bigl[ d + \log(n) \bigr]}{n}} + \biggl( \frac{\log(n) + d}{
\kappa n} \biggr)^{p/2} \B{E} 
\Bigl( \bigl\lVert G^{-1/2} X \bigr\rVert^{2q(p+1)} \Bigr)^{1/q} 
\\ + \frac{(1 + \kappa'/n)  C}{Mn} + \bigl( 1 + \kappa'/n \bigr)^{1/2} 
n^{-(q-1)/2} \Biggr) \\ \underset{\text{when } q < 2}{=} 1 - \bO_{n \rightarrow \infty} \Bigl( n^{-(q-1)/2} 
\Bigr). 
\end{multline*}
Combining this result with the reverse bound of Proposition \vref{prop:4.10}
gives, in the case when $q \in ]1, 2[$, 
\[ 
\biggl\lvert \frac{n}{C} \, \B{E} \Bigl[ \min \bigl\{ R(\wh{\theta}) - 
R(\theta_*), M \bigr\} \Bigr] - 1 \biggr\rvert \leq \bO_{n \rightarrow \infty} 
\bigl( n^{- (q-1)/2} \bigr).
\] 
Looking at the non-asymptotic bounds, we see that we can bind $M$ to $n$
by the relation $M = C \exp \bigl( n^{2-q} \bigr)$ and still get 
when $q \in ]1,2[$, 
\begin{equation}
\label{eq:22.2}
\biggl\lvert \frac{n}{C} \B{E} \Bigl[ \min \bigl\{ R(\wh{\theta}) - 
R(\theta_*), \, C \exp \bigl( n^{2 - q} \bigr)  \bigr\} \Bigr] - 1 \biggr\rvert \leq \bO_{n \rightarrow \infty} 
\bigl( n^{- (q-1)/2} \bigr).
\end{equation}
This means that we can threshold $R(\wh{\theta}) - R(\theta_*)$ at a very 
high level, in the sens that the threshold is reached with a very 
small probability when the sample size is large, since from 
Markov's inequality, 
\begin{multline*}
\B{P} \Bigl[ R(\wh{\theta}) - R(\theta_*) \geq C \exp(n^{2 - q}) \Bigr] 
\leq 
\frac{\B{E} \Bigl[ \min \bigl\{ R(\wh{\theta}) - R(\theta_*), C 
\exp \bigl( n^{2-q} \bigr) \bigr\} \Bigr]}{C \exp \bigl( n^{2-q} \bigr)} 
\\ \leq \bO_{n \rightarrow \infty} \biggl( \frac{1}{n \exp \bigl( n^{2 - q} \bigr)} \biggr). 
\end{multline*}
\end{prop}
The proof of this proposition is given in appendix.\\[1ex]  
The exact rate $C/n$ can be used to construct another 
lower bound, this time for the expected excess risk, 
to complement the lower bound 
on the deviations of the excess risk given in Proposition \vref{prop:4.9}.  
For this, we want to 
describe a case where $C$ is much larger 
than $d \, R(\theta_*) / n$. 
Consider $\eta \in \{ -1, +1 \}$,  a Rademacher random variable 
independent from $X$. This means more precisely that 
\[
\B{P}(\eta = +1 \, | \, X ) 
= \B{P}(\eta = -1 \, | \, X ) = 1/2.
\]
Define for some $\theta_* \in \B{R}^d$
\[ 
Y = \langle \theta_*, X \rangle + \eta \lVert G^{-1/2} X \rVert.
\]  
Let $\kappa_1$ and $\kappa_2$ be defined as in Lemma \vref{lemma:4.3}.   
In this case, 
\begin{multline*}
C = \B{E} \Bigl( \bigl( Y - \langle \theta_*, X \rangle \bigr)^4 \Bigr) 
= \kappa_2 \B{E} \Bigl( \bigl( Y - \langle \theta_*, X \rangle \bigr)^2 
\Bigr)^2 \\ = \kappa_2 \B{E} \Bigl( \bigl( Y - \langle \theta_*, X \rangle 
\bigr)^2 \Bigr) \B{E} \Bigl( \bigl\lVert G^{-1/2} X \bigr\rVert^2 \Bigr) 
= \kappa_2 R(\theta_*) d. 
\end{multline*}
We can add some hypotheses on the structure of $X$ to make sure 
that $\kappa_2$ and $\kappa_1 + \kappa_2$ are of the same order of 
magnitude. Consider a centered Gaussian vector $W \in \B{R}^d$ 
such that $W \sim \C{N}(0, G)$ and an independent non-negative 
real valued random variable $\rho$ such that $\B{E}(\rho^2) = 1$. 
This implies that $G = \B{E}(X X^{\top})$ is the Gram matrix of 
$X$ as well as of $W$. Since 
\begin{multline*}
\B{E} \bigl( \langle \theta, X \rangle^4 \bigr) =  \B{E}\bigl( \rho^4 \bigr)
\B{E} \bigl( \langle \theta, W \rangle^4 \bigr) \\ = 3 \,  
\B{E} \bigl( \rho^4 \bigr) \B{E} \bigl( \langle \theta, W \rangle^2 \bigr)^2
= 3 \,  
\B{E} \bigl( \rho^4 \bigr) \B{E} \bigl( \langle \theta, X \rangle^2 \bigr)^2,
\end{multline*}
$\kappa_1 = 3 \B{E} \bigl( \rho^4 \bigr)$. On the other hand 
\begin{multline*}
\B{E} \Bigl( \bigl( Y - \langle \theta_*, X \rangle \bigr)^4 \Bigr)
= \B{E} \Bigl( \bigl\lVert G^{-1/2} X \bigr\rVert^4 \Bigr) = 
\B{E} \bigl( \rho^4 \bigr) \B{E} \Bigl( \bigl\lVert G^{-1/2} W 
\bigr\rVert^4 \Bigr) \\ 
= \B{E} \bigl( \rho^4 \bigr) \bigl[ 3 d + d(d-1) \bigr] 
= \biggl( 1 + \frac{2}{d} \biggr) \B{E} \bigl( \rho^4 \bigr) 
\B{E} \Bigl( \bigl\lVert G^{-1/2} X \bigr\rVert^2 \Bigr)^2
\\ = 
\biggl( 1 + \frac{2}{d} \biggr) \B{E} \bigl( \rho^4 \bigr) 
\B{E} \Bigl( \bigl( Y - \langle \theta_*, X \rangle \bigr)^2 \Bigr)^2,  
\end{multline*}
so that 
$\ds \kappa_2 = \biggl( 1 + \frac{2}{d} \biggr) \B{E} \bigl( \rho^4 \bigr)$ 
and 
\[ 
C = \frac{1 + 2/d}{4 + 2/d} \bigl( \kappa_1 + \kappa_2) R(\theta_*) d.  
\] 
If we assume that $\B{E} \bigl( \rho^8 \bigr) < \infty$, 
then the hypotheses of Proposition \ref{prop:4.12} are fulfilled 
for $p = 1$ and any $q \in [1,2[$, so that for any $q < 2$, 
\begin{multline*}
\B{E} \Bigl( R(\wh{\theta}) - R(\theta_*) \Bigr) \geq 
\B{E} \Bigl( \min \bigl\{ R(\wh{\theta}) - R(\theta_*), 
C \exp \bigl( n^{2-q} \bigr) \bigr\} \Bigr) \\ 
\geq \Bigl( 1 - 
\bO \bigl( n^{-(q-1)/2} \bigr) \Bigr) 
\frac{(1 + 2 / d) \bigl( 
\kappa_1 + \kappa_2 \bigr) d R(\theta_*)}{(4 + 2/d) n}. 
\end{multline*}
This provides a lower bound for the expected excess risk that 
complements Proposition \vref{prop:4.9}.  

Here, we just described a case where $C$ can be arbitrarily 
larger than $d R(\theta_*)$, since $\kappa_1 + \kappa_2 = \bigl( 4 + 2 / d \bigr) \B{E} \bigl( \rho^4 \bigr)$ can 
be arbitrarily large. It is also interesting to remark that 
$C$ can be arbitrarily smaller than $d \, R(\theta_*)$. 
Let us describe such a situation. 
Let $U$ be a uniform random vector on the unit sphere of $\B{R}^d$,
so that $\B{E} \bigl( U U^{\top} \bigr) = d^{-1} \mb{I}$. 
Let $\rho$ be an independent random variable taking the two 
positive real values $a$ and $b$ each with probability $1/2$.  
Let $\eta \in \{-1, +1\}$ be a Rademacher random variable 
independent from $U$ and $\rho$. 
\[ 
\text{Let } X = \rho U \text{ and } Y = \langle \theta_*, X \rangle + \eta \lVert 
X \rVert^{-1} = \langle \theta_*, X \rangle + \eta / \rho. 
\]  
In this case 
\begin{align*}
G & = \B{E} \bigl( X X^{\top} \bigr) = \B{E} \bigl( 
\rho^2 \bigr) \B{E} \bigl( U U^{\top} \bigr) = \frac{a^2 + b^2}{2d} \mb{I}, \\  
C & = \frac{2d}{a^2 + b^2}, \\
\text{and } R(\theta_*) & =  \frac{a^{-2} + b^{-2}}{2}, \\ 
\text{so that  } \frac{C}{d R(\theta_*)} & = \frac{4}{(a^2 + b^2)(a^{-2} + b^{-2})} = \frac{4}{2 + a^2/b^2 + b^2/a^2}
\end{align*}
can take any value in the range $]0,1]$ depending on the value of the 
ratio $a/b$. Meanwhile, it is easy to check that  
the hypotheses of Proposition \ref{prop:4.12} are fulfilled for any 
$p$ and $q \in [1,2]$ so that 
equation \myeq{eq:22.2} is satisfied for any $q \in ]1,2[$. 

\appendix

\section{Proof of Proposition \vref{prop1.2.3}}
\label{appA}

Let us assume without loss of generality that $G$ is of full rank.
Indeed it is easy to see that $X \in \IM(G)$ almost surely. 
This comes from the fact that 
for any $\theta \in \Ker(G)$, $\B{E}\bigl( \langle \theta, X \rangle^2 
\bigr) = 0$,  
and therefore that $\B{P} \bigl( \langle \theta, X \rangle = 0 \bigr) = 1$. 
Taking a basis of $\Ker(G)$, we obtain that  
$\B{P} \bigl( X \in \Ker (G)^{\perp} \bigr) = 1$. Remark 
now that $\Ker(G)^{\perp} = \IM(G)$, since $G$ is 
symmetric, so that $\B{P} \bigl( x \in \IM(G) \bigr) = 1$. 
Restricting the state space to $\IM(G)$ and considering the 
coordinates of $X$ in some orthonormal basis of $\IM(G)$ sets us 
back to the case where $\Ker(G) = \{ 0 \}$. 

Let us consider the new variables $\wt{\theta} = G^{1/2} \theta$, and 
$\wt{X} = G^{-1/2} X$. Working with $(\wt{\theta}, 
\wt{X})$ instead of $(\theta, X)$ is the same as assuming that $G = \Id$
(the identity matrix of rank $d$), 
and consequently that $N(\theta) = \Vert \theta \rVert^2$. 
This is what we will do in the following of this proof, keeping 
in mind that with this convention, although $\langle \theta, X_i \rangle$ 
is still observable, $X_i$ and $\theta$ themselves  are not. 

For any $\theta \in \B{R}^d$,  let us consider $\pi_{\theta} = 
\C{N}(\theta, \beta^{-1} \Id)$, the Gaussian distribution with mean 
$\theta$ and covariance matrix $\beta^{-1} \Id$. 

In order to apply some PAC-Bayesian inequality, we introduce
into the computations the perturbation of $\theta$ defined by $\pi_{\theta}$. 

We are going to prove a succession of lemmas leading to 
\begin{prop}
\label{prop1.2.2}
Let us introduce the numerical constant 
\[
c = \frac{15}{8 \log(2) (\sqrt{2} - 1)}   \exp \biggl( \frac{1 + 2 \sqrt{2}}{2} 
\bigg) \leq 44.3. 
\]
For any real parameter $\lambda > 0$, any $x, \theta \in \B{R}^d$, 
\begin{multline*}
\psi \Bigl\{ \lambda \bigl[ \langle \theta, x \rangle^2 - 1 \bigr] \Bigr\} 
\leq \int \log \Biggl\{ 1 + \lambda \biggl[ \langle \theta', x \rangle^2 
- 1 - \frac{ \lVert x \rVert^2}{\beta} \biggr] \\ 
+ \frac{\lambda^2}{2} \biggl[ \langle \theta', x \rangle^2 - 1 - \frac{\lVert 
x \rVert^2}{\beta} \biggr]^2 + 
\frac{c \lambda^2 \lVert x \rVert^2}{\beta} 
\biggl( \langle \theta', x \rangle^2 
+ \frac{ \lVert x \rVert^2}{2 \beta} \biggr) \Biggr\} \, \ud \pi_{\theta}(\theta').
\end{multline*} 
\end{prop}
Let us introduce the function 
\[
\chi(x) = 
\begin{cases} 
\psi(x), & x \leq x_1, \\ 
y_1 + p_1(x-x_1) - (x-x_1)^2/8, & x_1 \leq x \leq x_1 + 4 p_1, \\
y_1 + 2 p_1^2, & x \geq x_1 + 4 p_1,
\end{cases} 
\]
where $x_1 \in [0, 1], y_1, p_1$ are defined by the conditions 
$\psi''(x_1) = - 1/4$, $y_1 = \psi(x_1)$, and $p_1 = \psi'(x_1)$.

Since $\psi''(x)$ continues 
to decrease after $x_1$, whereas $\chi''(x)$ remains constant,
until $\chi'(z) = 0$, afterwhich the function $\chi$ is constant, 
we see that $\psi(z) \leq \chi(z)$ for all $z \in \B{R}$.  
On the other hand, 
\[
\chi(z) \leq \log \bigl( 1 + z + z^2 / 2 \bigr).
\]
Indeed, we already saw that this is the case for $\psi$, and 
as the function $f(z) = \log \bigl( 1 + z + z^2 / 2 \bigr)$ 
is such that $f(x_1) \geq \chi(x_1) = \psi(x_1)$, $f'(x_1) 
\geq \chi'(x_1) = \psi'(x_1)$, and $\inf f'' = - 1/4$, 
$f$ is above $\chi$ on the right-hand side of $x_1$ also.  

Starting from the expressions 
\[ 
\psi'(z) = \frac{1-z}{1 - z + z^2/2} 
\quad  
\text{ and } \quad \psi''(z) = \frac{- z + z^2 / 2}{ 1 - z + z^2/2},
\] 
it is 
easy to compute $x_1$, $y_1$ and $p_1$. We obtain
\begin{align*}
x_1 & = 1 - \textstyle \sqrt{4 \sqrt{2} - 5}, \\ 
y_1 & = - \log \bigl[ 2 ( \sqrt{2} - 1 ) \bigr], \\ 
p_1 & = \frac{ \sqrt{4 \sqrt{2} - 5}}{2 ( \sqrt{2} - 1)},\\ 
\sup_{z \in \B{R}} \chi(z) & = y_1 + 2 p_1^2 = \frac{1 + 2 \sqrt{2}}{2} 
- \log \bigl[ 2 ( \sqrt{2} - 1) \bigr].
\end{align*}

We start with a bound comparing perturbations of the parameter 
$\theta$ inside and outside of $\chi$. 

\begin{lemma} For any $\rho \in \C{M}_+^1(\Theta)$ and any $h 
\in \B{L}^1(\rho)$,
$$
\chi \biggl( \int h \, \ud \rho \biggr) \leq \int \chi (h) \, \ud \rho 
+ \frac{1}{8} \Var \bigl( h \, \ud \rho \bigr),
$$
where by definition 
$$
\Var \bigl( h \, \ud \rho \bigr) = 
\int \biggl( h - \int h \, \ud \rho \biggr)^2 \, \ud \rho(\theta) 
\in \B{R} \cup + \infty.
$$
\end{lemma}
\begin{proof}
Let us consider the function 
$$
g(y ) = \chi ( y ) + \frac{1}{8} \biggl(y - \int h \ \ud \rho \biggr)^2, \qquad 
y \in \B{R}.
$$
Since $ \inf \chi'' = 1/4 $, the function $g$ is convex. 
Jensen's inequality shows that 
$$
g \biggl( \int h \, \ud \rho \biggr) \leq \int g ( h ) \, \ud \rho,
$$ 
and we conclude, remarking that 
$$
g\biggl( \int h \, \ud \rho \biggr) = \chi \biggl( \int h \, \ud 
\rho \biggr).
$$
\end{proof}
\begin{lemma}
For any $\rho \in \C{M}_+^1(\Theta)$ and any $h \in \B{L}^1(\rho)$, 
$$
\psi \biggl( \int h \, \ud \rho \biggr) \leq \int \chi(h) \, \ud \rho 
+ \min \Bigl\{ \log(4), \frac{1}{8} \Var \bigl( h \, \ud \rho \bigr) \Bigr\}.
$$
\end{lemma}
\begin{proof}
\[ 
\psi \biggl( \int h \, \ud \rho \biggr) \leq \sup \psi + \int \chi(h) \, 
\ud \rho - \inf \chi = \int \chi(h) \, \ud \rho + \log(4).
\] 
We obtain this lemma by combining this inequality with the previous 
lemma (since $\psi(z) \leq \chi(z)$ for any $z \in \B{R}$).
\end{proof}

Applying this to our problem gives, for any $\theta \in \B{R}^d$, 
\begin{multline*}
\psi \Bigl\{ \lambda \bigl[ \langle \theta, x \rangle^2 - 1 \bigr] \Bigr\} 
= \psi \Biggl\{ \lambda \biggl[ \int \langle \theta', x \rangle^2  
\, \ud \pi_{\theta}(\theta') - \frac{\lVert x \rVert^2}{\beta}  - 1 \biggr] 
\Biggr\} \\ 
\leq \int \chi \Biggl\{ \lambda \biggl[ 
\langle \theta', x \rangle^2 - \frac{\lVert x \rVert^2}{\beta} - 1 \biggr] 
\Biggr\} \, \ud \pi_{\theta}(\theta') 
\\ + \min \Bigl\{ \log(4), \frac{\lambda^2}{8} \Var \bigl[ 
\langle x, \theta' \rangle^2 \, \ud \pi_{\theta}(\theta') \bigr]  
\Bigr\} \\ 
= \int \chi \Biggl\{ \lambda \biggl[ \langle \theta', x \rangle^2 - 
\frac{\lVert x \rVert^2}{\beta} - 1 \biggr] \Biggr\} \, \ud 
\pi_{\theta}(\theta') \\ 
+ \min \biggl\{ \log(4), \frac{\lambda^2 \lVert x \rVert^2 \langle
\theta, x \rangle^2}{2 \beta} + \frac{ \lambda^2 \lVert x \rVert^4}{
4 \beta^2} \biggr\},
\end{multline*}
where we have used the fact that when $W \sim \C{N}(0, \sigma^2)$, 
\begin{multline*}
\Var  \bigl[ (m + W)^2 \bigr] = \Var \bigl( W^2 + 2 m W \bigr) 
\\ = \B{E} \Bigl[ \bigl( W^2 + 2 m W \bigr)^2 \Bigr]- 
\B{E} \bigl( W^2 + 2 m W \bigr)^2 
\\ = \B{E} \bigl( W^4 \bigr) + 4 m^2 \sigma^2 - \sigma^4 
= 2 \sigma^4 + 4 m^2 \sigma^2,
\end{multline*}
to compute $\Var \bigl[ \langle x, \theta' \rangle^2 \, \ud \pi_{\theta} 
( \theta' ) \bigr]$. \\
Remark now that for any positive real numbers $a$, $b$, 
and $c$,  
\[ 
\min \bigl\{ a, b m^2 + c \bigr\} \leq \min \bigl\{ a, 
b (m+W)^2 + c \bigr\} + \min \bigl\{ a, b (m - W)^2 + c \bigr\},
\] 
so that, the distribution of $W$ and $-W$ being the same, 
\[ 
\min \bigl\{ a, b m^2 + c \bigr\} \leq 2 \B{E} \Bigl[ 
\min \bigl\{ a, b (m+W)^2 + c \bigr\} \Bigr].
\] 
Applying this to the gaussian distribution $\pi_{\theta}$, 
we obtain that 
\begin{multline*}
\psi \bigl\{ \lambda \bigl[ \langle \theta, x \rangle^2 - 1 \bigr] 
\bigr\} \leq \int \chi \biggl\{ 
\lambda \biggl[ \langle \theta' , x \rangle^2 - 
\frac{\lVert x \rVert^2}{\beta} - 1 \biggr] \biggr\} \, 
\ud \pi_{\theta}(\theta') \\ 
+ \int \min \biggl\{ 4 \log(2), \frac{\lambda^2 \lVert x \rVert^2 
\langle \theta', x \rangle^2}{\beta} + \frac{\lambda^2 
\lVert x \rVert^4}{2 \beta^2} \biggr\} \, \ud \pi_{\theta}(\theta'). 
\end{multline*} 
We are now going to use the fact that for any $a, b, y \in \B{R}_+$, such that $y \leq b$,
\begin{multline*}
\log(a) + \min \{b, y \} = \log  \Bigl[ a \exp \bigl( \min \{ 
b, y \} \bigr) \Bigr]  
\\ \leq  \log \biggl\{ a + \min \{ b, y \} 
\frac{a \bigl( \exp(b) - 1 \bigr)}{b}
\biggr\}   \leq \log \biggl\{ a + y \frac{a \bigl( \exp(b) -1 \bigr)}{
b} \biggr\}. 
\end{multline*}
Applying this inequality to $a = \exp \bigl( \chi(z) \bigr)$ 
and reminding that $ \exp \bigl[ \chi(z) \bigr] \leq 1 + z + z^2$, 
we obtain that 
\[ 
\chi(z) + \min \{ b, y \} \leq \log \biggl\{ 
1 + z + z^2 / 2 + y \frac{ 
\exp \bigl( \sup \chi \bigr)  \bigl( \exp(b) -1 \bigr)}{
b} \biggr\}. 
\] 
Coming back to our problem, where we can take $b = 4 \log(2)$, we get
\begin{multline*}
\psi \Bigl\{ \lambda \bigl[ \langle \theta, x \rangle^2 - 1 \bigr] 
\Bigr\} \\ \leq \int \log \Biggl\{ 
1 + \lambda \biggl[ \langle \theta', x \rangle^2 - 
1 - \frac{\lVert x \rVert^2}{\beta}  \biggr]  
+ \frac{\lambda^2}{2} \biggl[ \langle \theta', x \rangle^2 
- 1 - \frac{\lVert x \rVert^2}{\beta} \biggr]^2 \\ + 
\frac{c \lambda^2 \lVert x \rVert^2}{\beta} \biggl[ \langle 
\theta', x \rangle^2 + \frac{\lVert x \rVert^2}{2 \beta} \biggr]
\Biggr\} \, \ud \pi_{\theta}(\theta'),
\end{multline*}
where 
\[ 
c = \frac{15}{8 \log(2)(\sqrt{2} - 1)} \exp \biggl( \frac{1 + 2 \sqrt{2}}{2} 
\biggr), 
\] 
as announced in Proposition \vref{prop1.2.2}. 

Now that we have compared $\psi \Bigl\{ \lambda \bigl[ 
\langle \theta, x \rangle^2 - 1 \bigr] \Bigr\}$ with 
an expectation with respect to a Gaussian perturbation 
of the parameter $\theta$, we are prepared to 
use the following PAC-Bayes inequality.
\begin{lemma}
\label{lemma1.4}
Let us consider two measurable spaces $\C{X}$ and $\Theta$ and 
an i.i.d. sample $(X_i)_{i=1}^n \in \C{X}^n$ with probability distribution
$\B{P}^{\otimes n}$. Let $\oB{P} = \frac{1}{n} \sum_{i =1}^n 
\delta_{X_i}$ be its empirical measure. 
Let us consider a prior 
probability measure $\nu \in \C{M}_+^1(\Theta)$. 
For any measurable function $f : \C{X} \times \Theta \rightarrow [a, + \infty[$, 
where $-1 < a < + \infty$, with probability at least $1 - \epsilon$ 
according to the sample distribution $\B{P}^{\otimes n}$, 
for any posterior probability measure $\rho \in \C{M}_+^1(\Theta)$ 
such that $\C{K}(\rho, \nu) < + \infty$,
\begin{multline}
\label{eq1}
\int \log \bigl[  1 + f(x,\theta) \bigr] \, \ud 
\rho(\theta) 
\, \ud \oB{P}(x)  \leq \int f(x, \theta) \, \ud \B{P}(x) \, \ud 
\rho(\theta) 
\\ + \frac{\C{K}(\rho, \nu) - \log(\epsilon)}{n}. 
\end{multline}
\end{lemma}
Let us remark that in this lemma and in the following of this paper
we will encounter events that are the union or the intersection of 
an uncountable family of measurable sets, as well as suprema of 
uncountable families of functions. To give a meaning to this, consider 
that when we write 
$\B{P}(A) \geq 1 - \epsilon$, we mean that there exists a measurable 
set $B \subset A$ such that $\B{P}(B) \geq 1 - \epsilon$. In the 
same way when we write 
$$
\int h \, \ud \B{P} \leq \eta,
$$
we mean that there is a measurale function $g$ such that $h(x) \leq g(x)$ 
for all $x \in \C{X}$ and 
$$
\int g \, \ud \B{P} \leq \eta.
$$

The proof of Lemma \vref{lemma1.4} requires a succession of 
preliminary results.
\begin{lemma}
For any measurable space $\Theta$, any upper bounded measurable function 
$h : \Theta \rightarrow \B{R}$, 
any probability measures $\rho, \pi \in \C{M}_+^1(\Theta)$, 
$$
\int h \, \ud \rho - \C{K}(\rho, \pi) \leq  \log \biggl( 
\int \exp(h) \, \ud \pi \biggr). 
$$
\end{lemma}
\begin{proof}
Let us introduce the probability measure $\pi_{\exp(h)} \ll \pi$ 
with density $\ds \frac{\ud \pi_{\exp(h)}}{\ud \pi} = \frac{\exp(h)}{\int 
\exp(h) \, \ud \pi}$. We can check that 
$$
\int h \, \ud \rho - \C{K}(\rho, \pi) - \log \biggl( \int \exp(h) \, \ud 
\pi \biggr) = - \C{K}(\rho, \pi_{\exp(h)} ) \leq 0,
$$
(where it is possible that $\C{K}(\rho, \pi) = + \infty$, 
$\C{K}(\rho, \pi_{\exp(h)}) = + \infty$, or $\int h \, \ud \rho = - \infty$).
\end{proof}
\begin{lemma}
Let us consider some measurable state space $\C{X}$, 
some measurable parameter space $\Theta$,
a prior probability measure $\nu$. 
For any measurable function $g : \C{X} \times \Theta \rightarrow 
\B{R}$, such that 
$$
\int \exp \bigl[ g(x, \theta) \bigr] \, \ud \B{P}(x) \leq 1, \qquad \theta 
\in \Theta,
$$ 
$$
\int \exp \biggl( \sup_{\substack{\rho \in \C{M}_+^1(\Theta),\\ \C{K}(\rho, \nu) < 
+ \infty}} n \int g(x, \theta) \, \ud \oB{P}(x) \, \ud \rho 
(\theta) - \C{K}(\rho, \nu) \biggr) \, \ud \B{P}^{\otimes n} 
\leq 1
$$
Therefore, with probability at least $1 - \epsilon$, for any 
$\rho \in \C{M}_+^1(\Theta)$ such that $\C{K}(\rho, \nu) < 
+ \infty$,
$$
\int g(x, \theta) \, \ud \oB{P}(x) 
\, \ud \rho(\theta) <  
\frac{\C{K}(\rho, 
\nu) - \log(\epsilon)}{n}. 
$$
\end{lemma}
Let us give some needed explanations about the meaning of this lemma.
The assumptions do not imply that $\theta \mapsto g(x, \theta)  
\in \B{L}^1(\rho)$. 
As we will see in the proof, the lemma 
is true without further assumptions if we use the convention 
$$
\int g(x, \theta) \, \ud \rho(\theta) = - \infty 
\text{ when } \int \min \bigl\{ g(x, \theta), 0 \bigr\} \, \ud \rho 
(\theta) = - \infty,
$$
and is defined otherwise as usual (so that the integral is equal to $+ \infty$
when the integral of the positive part of the integrand is infinite 
and the integral of its negative part is finite).

\begin{proof}
Let us put $g_k(x, \theta) = \min \{ g(x, \theta), k \}$, $k \in \B{N}$. 
Applying the monotone convergence theorem, the previous lemma, and the 
monotone convergence theorem again, we see that
\begin{multline*}
\exp \Biggl( \sup_{\substack{\rho \in \C{M}_+^1(\Theta),\\ \C{K}(\rho, \nu) 
< + \infty}}  n \int 
g(x, \theta) \, \ud \oB{P}(x) \, \ud \rho(\theta) - \C{K}(\rho
, \nu ) \Biggr) \\ 
= \sup_{k \in \B{N}} \exp \Biggl( \sup_{\substack{\rho \in \C{M}_+^1(\Theta), 
\\ \C{K}(\rho, \nu) < + \infty}} 
n \int g_k(x, \theta) \, \ud \oB{P}(x) \, \ud \rho(\theta) 
- \C{K}(\rho, \nu) \Biggr) 
\\  \leq 
\sup_{k \in \B{N}} 
\Biggl[ \int \exp \biggl( n \int g_k(x, \theta) \, \ud \oB{P}(x) \biggr) 
\, \ud \nu(\theta) \Biggr] 
\\ = 
\int \exp \biggl( n \int g(x, \theta) \, \ud \oB{P}(x) \biggr) \, \ud 
\nu(\theta). 
\end{multline*}
According to Fubini's theorem, the right-hand side of the last equality 
is a measurable function of the sample.
Thus, with our conventions about the use of integrals of non-measurable 
functions, we can write
\begin{multline*}
\int \exp \Biggl( \sup_{\substack{\rho \in \C{M}_+^1(\Theta),\\ \C{K}(\rho, 
\nu) < + \infty}} 
n \int g(x, \theta) \, \ud \oB{P}(x) \, \ud \rho(\theta) 
- \C{K}(\rho, \nu) \Biggr) \, \ud \B{P}^{\otimes n} \\ 
\leq \int \exp \biggl( n \int g(x, \theta) \, \ud \oB{P}(x) \biggr) 
\, \ud \nu(\theta) \, \ud \B{P}^{\otimes n}
\\ =  \int \exp \biggl( n \int g(x, \theta) \, \ud \oB{P}(x) \biggr) 
\, \ud \B{P}^{\otimes n} \, \ud \nu(\theta) 
\\ = \int \biggl( \int \exp \bigl[ g(x, \theta) \bigr] 
\, \ud \B{P}(x) \biggr)^n \, \ud \nu (\theta) \leq 1,  
\end{multline*}
where we have used Fubini's theorem for non-negative functions 
again and the independence of the sample. 

This proves the first inequality of the lemma. 

To prove the second one, let us consider the two events  
\begin{multline*}
A  = \Biggl\{ \int g(x, \theta) \, \ud \oB{P}(x) \, \ud \rho(\theta) 
< \frac{ \C{K}(\rho, \nu) - \log(\epsilon)}{n}, \\
\shoveright{ \rho \in \C{M}_+^1(
\Theta), \C{K}(\rho, \nu) < + \infty \Biggr\},} \\ 
\shoveleft{B  = \Biggl\{ 
\log \biggl[ \int \exp \biggl( n \int g(x, \theta) \, \ud 
\oB{P}(x) \biggr) \, \ud \nu(\theta) \biggr] < 
- \log(\epsilon) \Biggr\}. } \hfill 
\end{multline*}
According to the first part of the proof, $B \subset A$.
Moreover $B$ is measurable. 
Let us introduce the random variable 
$$
W = \log \biggl[ \int \exp \biggl( n \int g(x, \theta) \, \ud \oB{P} 
\biggr) \, \ud \nu(\theta) \biggr] + \log(\epsilon).
$$
It is measurable and $B = \{ W < 0 \}$, so that 
$\B{P}^{\otimes n} (B) = 1 - \B{P}^{\otimes n} \bigl( W \geq 0 \bigr).$  
We end the proof using the exponential Markov inequality
$$
\B{P}^{\otimes n} \bigl( W \geq 0 \bigr) 
\leq \int \exp (W) \, \ud \B{P}^{\otimes n} \leq \epsilon,
$$
the last inequality being a consequence of the first part of the proof.
\end{proof}

{\sc Proof of Lemma \vref{lemma1.4}}

Let us put $\ds m(\theta) = \int f(x, \theta) \, \ud \B{P}(x) 
\in [ a , + \infty] $. 
Let us consider 
$$
g(x, \theta ) = \B{1} \bigl[ m(\theta) < + \infty \bigr] \Bigl[ 
\log \bigl( 1 + f(x, \theta) \bigr) - m(\theta) \Bigr]
$$
We can check that 
$$
\int \exp \bigl[ g(x, \theta) \bigr] \, \ud \B{P}(x) = 
\begin{cases}
\ds \frac{1 + m(\theta)}{\exp\bigl[ m(\theta) \bigr]}, & \text{when } m(\theta) < + \infty\\ 
1, & \text{otherwise.} 
\end{cases} 
$$
Since $1 + m \leq \exp(m)$, from the convexity of the exponential function, 
we see that 
$$
\int \exp \bigl[ g(x, \theta) \bigr] \, \ud \B{P}(x) \leq 1.
$$ 
Therefore, we can apply the previous lemma to this choice of $g$. 
It shows that with probability at least $1 - \epsilon$, 
for any $\rho \in \C{M}_+^1(\Theta)$, such that $\C{K}(\rho, \nu) 
< + \infty$,
\begin{equation}
\label{eq2}
\int \B{1}\bigl[ m(\theta') < + \infty \bigr] 
\Bigl\{ \log \bigl[ 1 + f(x, \theta') \bigr] - m(\theta') \Bigl\} 
\, \ud \oB{P}(x) \, \ud \rho_{\theta}(\theta') < 
B(\rho, \epsilon),
\end{equation}
where $\ds B(\rho, \epsilon) = \frac{ \C{K}(\rho, \nu) - 
\log(\epsilon)}{n}$. 
We can then remark that in the case when 
$\ds \int m(\theta) \, \ud \rho(\theta) < + \infty, $
$\ds \int \B{1} \bigl[ m(\theta) = + \infty \bigr]  
\, \ud \rho(\theta) = 0$ and the left-hand side of equation \eqref{eq2} is equal to  
$$
\int \log \bigl[ 1 + f(x, \theta) \bigr] \, \ud \oB{P}(x) 
\, \ud \rho(\theta)  - \int m(\theta) 
\, \ud \rho (\theta),
$$
so that in this case equation \eqref{eq2} is equation 
\myeq{eq1} with a strict inequality, and therefore implies 
equation \eqref{eq1}. On the other hand, when $\ds \int m(\theta) \, \ud \rho
(\theta) = + \infty$, inequality \myeq{eq1} is also true because its 
right-hand side is equal to $+ \infty$.  
$\square$

Let us now apply Lemma \thmref{lemma1.4} to our problem, 
choosing $\C{X} = \Theta = \B{R}^d$, $\rho = \pi_{\theta}$ 
and $\nu = \pi_0$. 
It proves that, with probability at least $1 - \epsilon$, for any 
$\theta \in \B{R}^d$,  
\begin{multline*}
\int \psi \Bigl\{ \lambda \bigl[ \langle \theta, x \rangle^2 - 1 \bigr] \Bigr\} 
\, \ud \oB{P}(x) \leq \int \biggl\{ \lambda \biggl( \langle \theta', 
x \rangle^2 - 1 - \frac{\lVert x \rVert^2}{\beta} \biggr) \\ 
+ \frac{\lambda^2}{2} \biggl( \langle \theta', x \rangle^2 - 1 - \frac{\lVert
x \rVert^2}{\beta} \biggr)^2 + \frac{c \lambda^2 \lVert x \rVert^2}{\beta} 
\biggl( \langle \theta', x \rangle^2 + 
\frac{\lVert x \rVert^2}{2 \beta} \biggr) \biggr\} \, \ud \pi_{\theta} (\theta') \, \ud \B{P}(x) \\ + \frac{\C{K}(\pi_{\theta}, \pi_0) - \log(\epsilon)}{n}.
\end{multline*}
We can now compute 
$\ds \C{K}(\pi_{\theta}, \pi_0) 
= \frac{\beta \lVert \theta \rVert^2}{2} = \int \frac{ \beta \langle \theta, 
x \rangle^2}{2} \, \ud \B{P}(x)$, taking here into account the 
change of variables that turns the Gram matrix into the identity.
Integrating then explicitly with 
respect to the Gaussian measure $\pi_{\theta}$, we get with probability 
at least $1 - \epsilon$, for all $\theta \in \B{R}^d$, that
\begin{multline*}
\int \psi \Bigl\{ \lambda \bigl[ \langle \theta, x \rangle^2 - 1 \bigr] 
\Bigr\} \, \ud \oB{P}(x) 
\leq \int \biggl\{ \lambda \bigl( \langle \theta, x \rangle^2 - 1 \bigr) 
\\ + \frac{\lambda^2}{2} \biggl[ \bigl( \langle \theta, x \rangle^2 - 1 
\bigr)^2 + \frac{4 \lVert x \rVert^2}{\beta} \langle \theta, x \rangle^2 
+ \frac{2 \lVert x \rVert^4}{\beta^2} \biggr] \\ 
+ \frac{c \lambda^2 \lVert x \rVert^2}{\beta} \biggl( 
\langle \theta , x \rangle^2 + \frac{3 \lVert x \rVert^2}{2 \beta} 
\biggr) + \frac{\beta \langle \theta, x \rangle^2}{2 n} \biggr\}
\, \ud \B{P}(x) - \frac{\log ( \epsilon )}{n}.
\end{multline*}
Let us introduce
\begin{align*}
s_4^4 & = \int \lVert x \rVert^4 \, \ud \B{P}(x), \\
\kappa & = \sup \biggl\{ \int \langle \theta, x \rangle^4 \, \ud \B{P}(x)\, ; 
\; \theta \in \B{R}^d, \,\int \langle \theta, x \rangle^2 \, 
\ud \B{P}(x) \leq 1 \biggr\}.
\end{align*}

\begin{lemma}
\label{lem1.7}
The coefficients $s_4$ and $\kappa$ are such that 
$$
s_4^4 \leq \kappa \biggl( \int \lVert x \rVert^2 \, \ud \B{P}(x) 
\biggr)^2 = \kappa d^2.
$$
\end{lemma}
\begin{proof}
Taking coordinates, we can write
\begin{multline*}
s_4^4 = \int \biggl( \sum_{i=1}^d x_i^2 \biggr)^2 \, \ud \B{P}(x)
= \sum_{1 \leq i, j \leq d} \int x_i^2 x_j^2 \, \ud \B{P}(x) 
\\ \leq  \sum_{1 \leq i, j \leq d} \biggl( \int x_i^4 \, \ud \B{P}(x) 
\biggr)^{1/2} \biggl( \int x_j^4 \, \ud \B{P}(x) \biggr)^{1/2}
\\ \leq \kappa \sum_{1 \leq i, j \leq d} \biggl( \int x_i^2 \, \ud \B{P}(x) \biggr)
\biggl( \int x_j^2 \, \ud \B{P}(x) \biggr) \\ = 
\kappa \biggl( \int \lVert x \rVert^2 \, \ud \B{P}(x) \biggr)^2. 
\end{multline*}
\end{proof}

Using the Cauchy-Schwarz inequality, we obtain with probability at least 
$1 - \epsilon$, for any $\theta \in \B{R}$, 
\begin{multline*}
r_{\lambda}(\theta) \leq \biggl[ 1 + (\kappa - 1) \lambda  
+ \frac{(2 + c) \lambda s_4^2 \sqrt{\kappa}}{\beta} + \frac{\beta}{2 n 
\lambda} \biggr] \bigl[ N(\theta) - 1 \bigr] \\ 
+ \frac{\kappa \lambda}{2} \bigl[ N(\theta) - 1 \bigr]^2 
+ \frac{(\kappa -1 ) \lambda}{2} 
+ \frac{(2 + c) \sqrt{\kappa} \lambda s_4^2}{\beta} 
\\ + \frac{(2 + 3 c) \lambda s_4^4}{2 \beta^2} 
+ \frac{\beta}{2 n \lambda} - \frac{\log ( \epsilon )}{n 
\lambda}.  
\end{multline*}
Let us choose some special values for $\beta$, the strength of 
the perturbation of the parameter, and $\lambda$, the scale 
parameter of the influence function $\psi$. Let us also 
introduce some more concise notations. Let us put accordingly 
\begin{align*}
\beta & = \lambda s_4 \kappa^{1/4} \sqrt{2 (2 + c) n},\\
\lambda & = 
\sqrt{ \frac{2}{(\kappa-1) n} \biggl[ 
\log\bigl( \epsilon^{-1} \bigr) + \frac{(2 + 3 c) s_4^2}{4 
\sqrt{\kappa}(2 + c)} \biggr]},\\
\eta & = (\kappa - 1) \lambda = 
\sqrt{ \frac{2(\kappa-1)}{n} \biggl[ 
\log\bigl( \epsilon^{-1} \bigr) + \frac{(2 + 3 c) s_4^2}{4 
\sqrt{\kappa}(2 + c)} \biggr]},\\
\gamma & =  \sqrt{\frac{2 (2 + c)s_4^2 \sqrt{\kappa}}{n}},\\
\mu  & = \eta + \gamma,\\
\xi & = \frac{\kappa \eta}{2(\kappa-1)}.
\end{align*}
Let us remark that, according to Lemma \thmref{lem1.7}, $s_4^2 
\leq \sqrt{\kappa} d$, so that 
$$
\mu \leq \sqrt{ \frac{2 (\kappa - 1)}{n} \biggl[ 
\log(\epsilon^{-1}) + \frac{(2 + 3c) d}{4 (2 +  c)} \biggr]} 
+ \sqrt{\frac{2 (2 + c) \kappa d}{n}}
$$
Numerically, 
$$
\mu \leq \sqrt{ \frac{2(\kappa -1)}{n} 
\bigl[ \log(\epsilon^{-1}) + 0.73 \, d \bigr]} + 6.81 \sqrt{ 
\frac{2 \kappa \, d }{n}}
$$
\begin{prop}
\label{prop1.1}
With probability at least $1 - \epsilon$, for any $\theta \in \B{R}^d$, 
$$
r_{\lambda}(\theta) \leq \xi \bigl[ N(\theta) - 1 \bigr]^2 
+ ( 1 + \mu) \bigl[ N(\theta) - 1 \bigr] 
+ \mu. 
$$
\end{prop}
Let us now study the reverse inequality.
Starting from 
\begin{multline*}
\psi \Bigl\{ \lambda \bigl[ 1 - \langle \theta, x \rangle^2 \bigr] \Bigr\} 
= \psi \Biggl\{ \lambda \biggl[ 1 + \frac{\lVert x \rVert^2}{\beta} 
- \int \langle \theta', x \rangle^2 \, \ud \pi_{\theta}(\theta') \biggr] 
\Biggr\} \\ \leq \int \chi \Biggl\{ \lambda \biggl[ 1 
+ \frac{\lVert x \rVert^2}{\beta} - \langle \theta', x \rangle^2 \biggr] 
\Biggr\} \, \ud \pi_{\theta}(\theta') \\ + 
\min \Bigl\{ \log(4), \frac{\lambda^2}{8} \Var \bigl[ \langle \theta', x \rangle^2 
\, \ud \pi_{\theta}(\theta') \bigr] \Bigr\},
\end{multline*}
and proceeding in the same way as previously, we obtain 
\begin{prop}
\label{prop1.2}
With probability at least $1 - \epsilon$, for any $\theta \in 
\B{R}^d$, 
$$
r_{\lambda}(\theta) \geq - \xi \bigl[ N(\theta) - 1 \bigr]^2 
+ ( 1 - \mu) \bigl[ N(\theta) - 1 \bigr] - \mu.
$$
\end{prop}

\begin{prop}
\label{prop1.3}
Assume that 
\begin{equation}
\label{eq1.1}
2 \mu + \xi < 1,
\end{equation}
which can be written as
\begin{multline*}
n > \left\{ \rule{0ex}{5ex} \right. \sqrt{ 8 ( 2 + c)} \kappa^{1/4} s_4 
\\ + \biggl( \frac{5}{2} + \frac{1}{2 (\kappa -1)} \biggr) 
\sqrt{ 2 (\kappa -1) \biggl[ \log(\epsilon^{-1}) + 
\frac{(2 + 3 c) s_4^2}{4 \sqrt{\kappa}(2 + c)} \biggr]} \; \left. \rule{0ex}{5ex} \right\}^{\!2}.
\end{multline*}
A sufficient numerical condition implying the one above is that 
\[
n > \Biggl[ 20 \sqrt{\kappa d} + \biggl( \frac{5}{2} + 
\frac{1}{2 (\kappa -1)} \biggr) 
\sqrt{2 (\kappa -1) \bigl[ 
\log(\epsilon^{-1}) + 0.73 \, d \bigr]} \Biggr]^2.
\]
With probability at least $1 - 2 \epsilon$, 
for any $\theta \in \B{R}^d$, 
$$
\biggl( 1 - \frac{\mu}{1 - 2 \mu} \biggr) \wh{N}(\theta) 
\leq N(\theta) \leq \biggl( 1 + \frac{\mu}{1 - 2 \mu} \biggr) \wh{N}(\theta), 
$$
\[
\text{where } \quad c = \frac{15}{8 \log(2) \bigl( \sqrt{2} - 1 \bigr)} \exp \biggl( 
\frac{1 + 2 \sqrt{2}}{2} \biggr),  
\]
\begin{multline*}
\text{and } \mu = \sqrt{ \frac{2 (\kappa -1)}{n} \biggl[ \log \bigl( 
\epsilon^{-1} \bigr)  + \frac{(2 + 3 c) s_4^2}{4 \sqrt{\kappa} 
( 2 + c )} \biggr]} + \sqrt{\frac{2 (2 + c) \sqrt{\kappa} s_4^2}{n}}
\\ \leq \sqrt{ \frac{2(\kappa-1)}{n} \biggl[ \log \bigl( \epsilon^{-1} 
\bigr) + \frac{(2 + 3 c) d}{ 4 (2 + c)} \biggr]} + 
2 \sqrt{ \frac{2 (2 + c) \kappa d}{n}} \\ 
\leq \sqrt{ \frac{2 (\kappa - 1) }{ n } \bigl[ \log 
\bigl( \epsilon^{-1} \bigr) + 0.73 d \bigr]} + 6.81 \sqrt{
\frac{2 \kappa d}{n}}.
\end{multline*}
\end{prop}
\begin{proof}
Assume that both inequalities of Proposition \ref{prop1.1} and \ref{prop1.2} 
hold for any $\theta \in \B{R}^d$, which happens at least with probability 
$1 - 2 \epsilon$. In this case, when $N(\theta) = 0$, for any $\alpha \in \B{R}_+$,  
$$
r_{\lambda}(\alpha \theta) \leq - 1 + \xi < 0, 
$$
so that $\wh{\alpha} = + \infty$, and $\wh{N}(\theta) = 0 = N(\theta)$. 
Assume now that $N(\theta) 
\neq 0$. In this case, there is $\ds \alpha = \sqrt{\frac{2}{N(\theta)}}$
such that $N(\alpha \theta) = 2$ and therefore such that 
$$
r_{\lambda} (\alpha \theta) \geq 1 - 2 \mu - \xi > 0,
$$ 
proving that $\wh{\alpha}(\theta) < + \infty$, and consequently that 
$r \bigl[\wh{\alpha}(\theta) \theta \bigr] = 0$, as explained before. 
For any $\alpha \in  
I = [0, \wh{\alpha}(\theta)]$, $r_{\lambda}(\alpha \theta) \leq 0$, because $\alpha \mapsto 
r_{\lambda}(\alpha \theta)$ is non-decreasing. Thus $N(\alpha \theta) - 1$ 
is solution of the following quadratic inequality in the unknown variable $z$ : 
\begin{equation}
\label{eq4}
0 \geq - \xi z^2 +  (1 - \mu) z 
- \mu. 
\end{equation}
Since the right-hand side of this inequality is positive when $z = 1$, 
from assumption \myeq{eq1.1}, its discriminant
$$
\Delta = (1 - \mu)^2 - 4 \xi \mu 
$$
is strictly positive. Let us consider the interval 
$$
J = \frac{1}{2 \xi} \Bigl( 1 - \mu + 
\sqrt{\Delta} \; \bigl]-1, +1 \bigr[ \Bigr).
$$
We obtain that 
$$
\bigl[ N(\alpha \theta) - 1 \bigr] \not\in J, \qquad \alpha \in [0, \wh{\alpha}(\theta)].
$$
Since $\alpha \mapsto N(\alpha \theta) - 1$ is continuous and $-1 \leq \inf J$, 
necessarily 
$$ 
N \bigl[ \wh{\alpha}(\theta) \theta \bigr] - 1 \leq \inf J \leq 
\frac{\mu}{1 - 2 \mu}.  
$$
Indeed, when $\ds z = \frac{\mu}{1 - 2 \mu}$, 
the right-hand side of inequality \eqref{eq4} is equal to 
$$
\frac{1 - 2 \mu - \xi}{
( 1 - 2 \mu )^{2}} > 0 
$$
according to assumption \myeq{eq1.1}, so that this value of $z$ belongs
 to $J$. 

On the other hand, $1 - N \bigl[ \wh{\alpha}(\theta) \theta \bigr]  \leq 1$ is also solution of 
$$
0 \leq \xi z^2 - (1 + \mu) z + 
\mu. 
$$
When $z=1$, the right-hand side of this inequality is strictly negative, 
due to assumption \myeq{eq1.1}, 
showing that the corresponding equality has two roots and that $1 - 
N\bigl[ \wh{\alpha}(\theta) \theta \bigr] $ is lower than its lowest root which is in turn lower than $\ds 
\frac{\mu}{1 - 2\mu}$. Thus, as $N\bigl[ \wh{\alpha}(\theta) \theta \bigr] 
= \wh{\alpha}(\theta)^2 N(\theta) = N(\theta) / \wh{N}(\theta)$, we have proved that 
$$
\biggl\lvert \frac{N(\theta)}{\wh{N}(\theta)} - 1 \biggr\rvert 
\leq \frac{\mu}{1 - 2\mu}, 
$$ 
from which Proposition 
\vref{prop1.3}, and therefore Proposition \vref{prop1.2.3}, follow.
\end{proof}

\section{Proof of Proposition \vref{prop3.1}}

Let $\Omega$ be the event of probability at least $1 - 2 \epsilon$ 
that appears in Proposition \vref{prop1.2.3}. 
Remark that on $\Omega$, 
\begin{multline}
\label{eq3}
\int \Bigl( \langle \theta, x \rangle^2 \wh{N}(\theta)^{-1} 
- 1 \Bigr)_+^3 \, \ud \oB{P}(x) \leq 
\int \frac{\langle \theta, x \rangle^6}{\wh{N}(\theta)^3} 
\, \ud \oB{P}(x) \\\leq 
\frac{\ds \max_{i=1, \dots, n} \langle \theta, X_i \rangle^4}{\wh{N}(\theta)^2} 
\int \frac{\langle \theta, x \rangle^2}{\wh{N}(\theta)} \, \ud \oB{P}(x)
\\ \leq  \frac{\ds \max_{i=1, \dots, n} \lVert G^{1/2} \theta \rVert^4 
\lVert G^{-1/2}X_i \rVert^4}{\wh{N}(\theta)^2} \times \frac{\ov{N}(\theta)}{
\wh{N}(\theta)} \\ = \frac{N(\theta)^2}{\wh{N}(\theta)^2} \times 
R^4 \times \frac{\ov{N}(\theta)}{\wh{N}(\theta)}
\leq \bigl( 1 + \wh{\delta} \bigr)^2 R^4 
\frac{\ov{N}(\theta)}{\wh{N}(\theta)}.
\end{multline}
Remark also that the conclusions of Proposition \vref{prop2.1.2} 
hold on $\Omega$, so that 
\[
- \ov{\delta}_-(\theta) \leq \frac{\ov{N}(\theta)}{\wh{N}(\theta)} - 1 
\leq \ov{\delta}_+(\theta).
\]
From equation \eqref{eq3}, 
\[
\ov{\delta}_+(\theta) \leq \gamma_+ \frac{ \ov{N}(\theta)}{\wh{N}(
\theta)}, 
\]
so that 
\[
\frac{\ov{N}(\theta)}{\wh{N}(\theta)} \leq \frac{1}{(1 - \gamma_+)_+} 
\]
and 
\[
\ov{\delta}_+(\theta) \leq \frac{\gamma_+}{(1 - \gamma_+)_+}.
\]
This gives 
\[
- \gamma_- \leq - \ov{\delta}_-(\theta) \leq \frac{\ov{N}(\theta)}{
\wh{N}(\theta)} - 1 \leq \ov{\delta}_+(\theta) \leq \frac{\gamma_+}{
(1 - \gamma_+)_+}. 
\]
Using Proposition \ref{prop1.2.3} again, we get 
that on the event $\Omega$, 
\[
\frac{(1 - \wh{\delta}) \ov{N}(\theta) }{1 + \ov{\delta}_+(\theta)} \leq ( 1 - \wh{\delta}) 
\wh{N}(\theta) \leq N(\theta) \leq (1 + \wh{\delta}) \wh{N}(\theta) 
\leq \frac{ (1 + \wh{\delta})\ov{N}(\theta)}{1 - \ov{\delta}_-(\theta)}. 
\]
As a consequence
\begin{multline*}
- \frac{\wh{\delta} + \gamma_-}{1 + \wh{\delta}}  \leq - \frac{\wh{\delta} + \ov{\delta}_-(\theta)}{1 + 
\wh{\delta}} \leq \frac{\ov{N}(\theta)}{N(\theta)} - 1 
\leq \frac{\wh{\delta} + \ov{\delta}_+(\theta)}{1 - 
\wh{\delta}} \\ \leq \frac{\wh{\delta} + \gamma_+ / (1 - \gamma_+)_+}{1 - 
\wh{\delta}} = \frac{1}{(1 - \wh{\delta})(1 - \gamma_+)_+} - 1  \leq \frac{\wh{\delta} + \gamma_+}{(1 - \wh{\delta})(1 - \gamma_+)_+}.
\end{multline*}

\section{Proof of Lemma \vref{lem:2.5}} 

Consider $u \in \B{R}_+$ that we will choose
appropriately and write 
the decomposition 
\[ 
\frac{1}{n} \sum_{i=1}^n W_i = \frac{1}{n} \sum_{i=1}^n 
\min \{ W_i, u \} + \frac{1}{n} \sum_{i=1}^n \bigl( W_i - u \bigr)_+.
\]  
Using Bienaym\'e Chebyshev's inequality on one hand and Markov's 
inequality on the other hand, we obtain that with probability 
at least $1 - 2 \epsilon$, 
\[ 
\frac{1}{n} \sum_{i=1}^n W_i \leq \B{E} ( W ) + \sqrt{ \frac{
\B{E} \bigl( \min \{ W, u \}^2 \bigr)}{n \epsilon}} + 
\frac{\B{E} \bigl[ \bigl( W - u \bigr)_+ \bigr] }{\epsilon}. 
\] 
Let us now remark that 
\[ 
\min \{ W, u \}^2 \leq u^{2-q} W^q \text{ and } \bigl( W - u \bigr)_+ 
\leq \frac{(q-1)^{q-1}}{q^q} u^{-(q-1)} W^q.  
\] 
The first inequality is obvious. The second one can be checked remarking that 
the two functions of $W$ meet at $W = u/(1 - q^{-1})$, where they have 
the same derivative and that the function on the right-hand side being 
convex is necessarily above its tangent $W \mapsto W - u$, and 
therefore also above $\bigl(W-u\bigr)_+$, since it is positive. 
More precisely, if we compute the solution of the contact equations 
\[ 
W_c - u = a W_c^q \text{ and } 1 = q a W_c^{q-1},
\] 
we find 
\[ 
W_c = \frac{q u}{q-1} \text{ and } a = \frac{1}{q} \biggl( \frac{q-1}{qu} 
\biggr)^{q-1} = \frac{(q-1)^{q-1}}{q^q} u^{-(q-1)}, 
\] 
as claimed. Thus with probability at least $1 - 2 \epsilon$, 
\[ 
\frac{1}{n} \sum_{i=1}^n W_i \leq \B{E}(W) + \sqrt{\frac{u^{2-q} \B{E}(W^q)}{
n \epsilon}} + \frac{(q-1)^{q-1} \B{E} \bigl( W^q \bigr)}{q^q u^{q-1} \epsilon}.
\]  
The optimal value of the threshold $u$ can be found computing the derivative 
of the bound. It is 
\[ 
u = \frac{ (q-1)^2 \B{E} \bigl( W^q \bigr)^{1/q} n^{1/q} }{q^2 (1 - q/2)^{2/q} 
\epsilon^{1/q} }.
\] 
Replacing $u$ by its value gives as stated in the lemma
\[ 
\frac{1}{n} \sum_{i=1}^n W_i \leq \B{E}(W) + \frac{q^{q-1} \B{E}(W^q)^{1/q}}{
2 (q-1)^{q-1} (1 - q/2)^{(2-q)/q} \epsilon^{1/q} n^{1 - 1/q}}. 
\]

\section{Proof of Proposition \vref{prop:3.3}}

Let us put 
\[
\wt{N}(\theta, \xi ) = 
\inf \Biggl\{ \rho \in \B{R}_+^*, \sum_{i=1}^n \psi \biggl[ 
\lambda \biggl( \frac{ \bigl( \langle \theta, X \rangle - \xi \bigr)^2}{
\rho} - 1 \biggr) \biggr] \leq 0  \Biggr\}.    
\] 
From Lemma \vref{lemma:3.2}, for any $(\theta, \xi) \in \B{R}^{d+1}$, 
\[ 
\B{E} \bigl[ \bigl( \langle \theta, X \rangle - \xi \bigr)^4 \bigr] 
\leq ( \kappa^{1/2} + 1)^2 \; \B{E} \bigl[ \bigl( \langle \theta, 
X \rangle - \xi \bigr)^2 \bigr]^2. 
\] 
We can therefore apply Proposition \vref{prop1.2.3} and obtain 
for any sample size $n$ satisfying equation \myeq{eq6.2} that 
with probability at least $1 - 2 \epsilon$, for any $(\theta, \xi) 
\in \B{R}^{d+1}$, 
\[ 
\biggl\lvert \frac{\B{E} \bigl[ \bigl( \langle \theta, x \rangle - \xi
\bigr)^2 \bigr]}{\wt{N}(\theta, \xi )} - 1 \biggr\rvert \leq 
\frac{\mu}{1 - 2 \mu}.
\]  
We can then apply Proposition \vref{prop3.1.2} to conclude.

\section{Proof of lemma \vref{lemma:3.4}} 

Consider $a = \B{E} \bigl( \langle \theta, X - \B{E}(X) 
\rangle^2 \bigr)^{1/2}$. 
If $a = 0$, then almost surely $\langle \theta, X_i - \B{E}(X) 
\rangle = 0$, so that in this case
\begin{align*}
\langle \theta, X_i \rangle - \xi & = \langle \theta, \B{E}(X) \rangle 
- \xi , \\ \text{and } \qquad  
\B{E} \bigl[ \bigl( \langle \theta, X \rangle - 
\xi \bigr)^2 \bigr] & = \bigl( \langle \theta, \B{E}(X)  \rangle - 
\xi \bigr)^2 \leq 1.
\end{align*}
Consequently $\langle \theta, X_i \rangle - \xi \leq 1$. 

Assume now that $a > 0$. In this case, from the Cauchy-Schwarz inequality 
in $\B{R}^2$,   
\begin{multline*}
\bigl( \langle \theta, X_i \rangle - \xi \bigr)^2 = 
\bigl( a a^{-1} \langle \theta, X_i - \B{E}(X) \rangle + \langle 
\theta, \B{E}(X) \rangle - \xi \bigr)^2 
\\ \leq \bigl[ a^{-2} \langle \theta, X_i - \B{E}(X) \rangle^2 
+ 1 \bigr] \bigl[ a^2 + \bigl( \langle \theta, \B{E}(X) \rangle - \xi \bigr)^2
\bigr]. 
\end{multline*}
Remark that almost surely $X_i - \B{E}(X) \in \IM (\Sigma)$, 
so that $\bigl( X_i - \B{E}(X) \bigr) = \Sigma^{1/2} \Sigma^{-1/2} 
\bigl( X_i - \B{E}(X) \bigr)$, and therefore, 
\begin{multline*}
\langle \theta, X_i - \B{E}(X) \rangle^2 = 
\langle \Sigma^{1/2} \theta, \Sigma^{-1/2} \bigl( X_i - \B{E}(X) \rangle^2
\\ \leq \lVert \Sigma^{1/2} \theta \rVert^2 
\lVert \Sigma^{-1/2} (X_i - \B{E}(X)) \rVert^2 = a^2 
\lVert \Sigma^{-1/2} \bigl( X_i - \B{E}(X) \bigr) \rVert^2.
\end{multline*}
On the other hand, 
\begin{multline*}
a^2 + \bigl( \langle \theta, \B{E}(X) - \xi \bigr)^2 
= \B{E} \bigl( \langle \theta, X - \B{E}(X) \rangle^2 \bigr)
+ \bigl( \langle \theta, \B{E}(X) - \xi \bigr)^2 
\\ = \B{E} \bigl[ \bigl( \langle \theta, X \rangle - \xi \bigr)^2 \bigr]
\leq 1. 
\end{multline*}
Putting the three last equations together, 
we obtain that 
\[ 
\bigl( \langle \theta, X_i \rangle - \xi \bigr)^2 
\leq \Bigl( \bigl\lVert \Sigma^{-1/2} \bigl( X_i 
- \B{E}(X) \bigr) \bigr\rVert^2 + 1 \Bigr), 
\] 
which achieves the proof. 

\section{Proof of Proposition \vref{prop:4.9}}
Let us put in this proof
\[ 
R(\theta) = \B{E} \bigl[ \bigl( \wt{Y} - \langle \theta, \wt{X} \rangle
\bigr)^2 \bigr] = p^{-1} \B{E} \bigl[ \bigl( Y - \langle \theta, X 
\rangle \bigr)^2 \bigr]. 
\] 
Let $\eta_1, \dots, \eta_n$ be $n$ independent copies of $\eta$, 
independent from $\wt{X}_1, \dots, \wt{X}_n$ and such that 
$\wt{Y}_i = \langle \theta_*, \wt{X}_i \rangle + \eta_i$. 
Consider 
\[ 
\wt{G} = \frac{1}{n} \sum_{i=1}^n \wt{X}_i \wt{X}_i^{\top}, \qquad 
W = \frac{1}{\sqrt{n} \sigma} \sum_{i=1}^n \eta_i \wt{G}^{-1/2} \wt{X}_i, 
\quad \text{ and } \quad \wt{d} = \rank(\wt{G}).  
\]  
Remark that 
\[ 
\B{P}_{W \, | \,\wt{X}_1, \dots, \wt{X}_n} = \C{N}
\bigl( 0, \wt{G}^{-1} \wt{G} \; \bigr), 
\] 
where $\wt{G}^{-1}$ is the pseudo inverse of $\wt{G}$, 
so that $\wt{G}^{-1} \wt{G}$ is the orthogonal projection 
on $\IM(\wt{G})$. 
Applying Chernoff's bound, we get 
\[ 
\B{P} \biggl( \lVert W \rVert^2 \geq \frac{2}{\alpha} \biggl[ \frac{\wt{d}}{2} 
\log ( 1 + \alpha ) - \log(\epsilon^{-1}) \biggr] \, \Big| \, 
\wt{X}_1 , \dots, \wt{X}_n \biggr) \geq 1 - \epsilon,  
\] 
and therefore, integrating with respect to $\wt{X}_1, \dots, \wt{X}_n$, 
\begin{equation}
\label{eq:19}
\B{P} \biggl( \lVert W \rVert^2 \geq \frac{2}{\alpha} \biggl[ \frac{\wt{d}}{2} 
\log ( 1 + \alpha ) - \log(\epsilon^{-1}) \biggr] 
\biggr) \geq 1 - \epsilon.  
\end{equation}
Remark now that
\[ 
\wt{\theta} = \wt{G}^{-1} \Biggl( \frac{1}{n} \sum_{i=1}^n \wt{Y_i} \wt{X}_i \Biggr) + \xi, 
\] 
where $\xi \in \Ker(\wt{G})$. 
Therefore 
\[ 
\wt{\theta} = \wt{G}^{-1} \Biggl( \frac{1}{n} \sum_{i=1}^n \langle \theta_*, 
X_i \rangle X_i + \eta_i X_i \Biggr) + \xi 
= \wt{G}^{-1} \wt{G} \theta_* + \frac{1}{n} 
\sum_{i=1}^n \eta_i \wt{G}^{-1} \wt{X}_i + \xi,  
\] 
We see that
\begin{multline*}
\frac{1}{n} \sum_{i=1}^n \langle \wt{\theta} - \theta_*, 
\wt{X}_i \rangle^2  = \bigl( \wt{\theta} - \theta_* \bigr)^{\top} \wt{G} 
\bigl( \wt{\theta} - \theta_* \bigr) = \bigl\lVert \wt{G}^{1/2} 
\bigl( \wt{\theta} - \theta_* \bigr) \bigr\rVert^2  
\\ 
= \biggl\lVert \frac{1}{n} \sum_{i=1}^n 
\eta_i \wt{G}^{-1/2} \wt{X}_i \biggr\rVert^2 = \frac{\sigma^2}{n}
\lVert W \rVert^2.
\end{multline*}
When $n \geq n_{\epsilon}$, with probability at least $1 - 2 \epsilon$, 
equation \myeq{eq:18} 
is satisfied, so that $\wt{d} = \rank(\wt{G}) = \rank(G) = d$
and 
\[ 
R(\wt{\theta}) - R(\theta_*) = \B{E} \bigl( \langle \wt{\theta} - \theta_*, 
X \rangle^2 \bigr) \geq \frac{1}{(1 + \delta) \, n} \sum_{i=1}^n \langle 
\wt{\theta} - \theta_*, X_i \rangle^2 = \frac{\sigma^2}{(1 + \delta) \, n} 
\lVert W \rVert^2. 
\] 
Combining this inequality with equation \eqref{eq:19}, 
we obtain that with probability at least $1 - 3 \epsilon$, 
\[ 
R(\wt{\theta}) - R(\theta_*) \geq \frac{\sigma^2}{(1 + \delta) \, \alpha n} 
\bigl[ d \log(1 + \alpha) - 2 \log(\epsilon^{-1}) \bigr].    
\] 
Choosing for simplicity $\alpha = 1$ gives the first statement 
of the proposition. 

Consider $\sigma_1, \dots, \sigma_n$, $n$ independent copies of $\sigma$, 
independent of 
\[
(\wt{X}_1, \wt{Y}_1), \dots, (\wt{X}_n, \wt{Y}_n), 
\] 
and such that $(X_i, Y_i) = \sigma_i (\wt{X}_i, \wt{Y}_i)$. 
Define 
\[
\ov{p} = \frac{1}{n} \sum_{i=1}^n \sigma_i.
\]
Conditioning with respect to $\sigma_1, \dots, \sigma_n$, we 
deduce from the first part of the proposition that 
\[ 
\B{P} \biggl( R(\wh{\theta}) - R(\theta_*) \geq \frac{\sigma^2 
\bigl[ d \log(2) - 2 \log(\epsilon^{-1}) \bigr]}{(1 + \delta) 
\ov{p} n } \, \Big| \, n_{\epsilon} \leq n \ov{p} \leq 7 n p / 4 \biggr) \geq 1 - 3 \epsilon. 
\] 
We proved in \cite[Proposition 20.1 page 289]{Cat2015} that 
with probability at least $1 - 2 \epsilon$, 
\[
\bigl\lvert \ov{p} - p \bigr\rvert \leq \sqrt{ 2 \log(\epsilon^{-1}) p / n} 
+ 2 \log(\epsilon^{-1})/n.
\]
Therefore, when 
\[
n \geq \frac{8 \log(\epsilon^{-1})}{p} = 
\frac{8 (\kappa_1 + \kappa_2) \log(\epsilon^{-1})}{\wt{\kappa}_1 
+ \wt{\kappa}_2}, 
\]
with probability at least $1 - 2 \epsilon$, 
\[
p/4 \leq \ov{p} \leq 7 p / 4.
\]
When moreover 
\[ 
n \geq \frac{4 n_{\epsilon}}{p} = 
\frac{4 (\kappa_1 + \kappa_2) n_{\epsilon}}{\wt{\kappa}_1 
+ \wt{\kappa}_2}. 
\]
as required in the proposition, with probability at least 
$1 - 2 \epsilon$, 
\[
n_\epsilon \leq n \ov{p} \leq 7 n p / 4.
\]
We can then write that
\begin{multline*}
\B{P} \biggl( 
R(\wh{\theta}) - R(\theta_*) \geq 
\frac{4 \sigma^2 \bigl[ d \log(2) - 2 \log(\epsilon^{-1}) \bigr]}{7 (1 + 
\delta) p n} \biggr) \\ \geq \B{P} \biggl(   
R(\wh{\theta}) - R(\theta_*) \geq 
\frac{4 \sigma^2 \bigl[ d \log(2) - 2 \log(\epsilon^{-1}) \bigr]}{7 (1 + 
\delta) p n} \\ \text{ and } n_{\epsilon} \leq n \ov{p} \leq 
7 n p / 4 \biggr) 
\\ = \B{P} \biggl( 
R(\wh{\theta}) - R(\theta_*) \geq  
\frac{4 \sigma^2 \bigl[ d \log(2) - 2 \log(\epsilon^{-1}) \bigr]}{7 (1 + 
\delta) p n} \, \Big| \, n_{\epsilon} \leq n \ov{p} \leq 7 n p / 4 \biggr) 
\\ \times \B{P} \Bigl( n_{\epsilon} \leq n \ov{p} \leq 7 n p / 4 \Bigr) 
\\ \geq   
\B{P} \biggl( 
R(\wh{\theta}) - R(\theta_*) \geq 
\frac{\sigma^2 \bigl[ d \log(2) - 2 \log(\epsilon^{-1}) \bigr]}{(1 + 
\delta) \ov{p} n} \, \Big| \, n_{\epsilon} \leq n \ov{p} \leq 7 n p / 4 \biggr) 
(1 - 2 \epsilon)  
\\ \geq (1 - 3 \epsilon)(1 - 2 \epsilon) \geq 1 - 5 \epsilon. 
\end{multline*}
We conclude the proof of equation \myeq{eq:19.2} by remarking that 
\[ 
\B{E} \bigl[ \bigl( Y - \langle \theta, X \rangle 
\bigr)^2 \bigr] = p R(\theta) \text{ and } 
\B{E} \bigl[ \bigl( Y - \langle \theta_*, X \rangle \bigr)^2 \bigr] 
= p R(\theta_*) = p \sigma^2, 
\] 
and that 
\[ 
p^{-1} = \frac{\kappa_1 + \kappa_2}{\wt{\kappa}_1 + \wt{\kappa}_2}. 
\] 

Applying Proposition \vref{prop:2.3} to a Gaussian design, 
we obtain, for any parameter $\alpha \in ]0,1[$, that
\[
\wh{\gamma}_+ \leq \frac{(1 + \wh{\delta})^2}{3} 
a \Bigl[ \frac{b}{\sqrt{n}}  + \frac{2 \log(n)}{\alpha \sqrt{n}}  \Bigr]^2, 
\] 
where 
\begin{align*}
a & = 0.73 \, d + \log(\epsilon^{-1}), \\ 
b & = \frac{d}{\alpha} \log \biggl( \frac{1}{1 - \alpha} \biggr) + 
\frac{2}{\alpha} \log(\epsilon^{-1}).  
\end{align*}
We can then use the fact that the log is concave and that for any constant 
$c > 0$ 
\[ 
\log(n) = 3 \log \bigl( n^{1/3} \bigr) \leq 3 \biggl( \log(c) 
+ \frac{n^{1/3} - c}{c} \biggr)
= 3 \biggl( \log ( c/e) + \frac{n^{1/3}}{c} \biggr).
\] 
Therefore, $\wh{\gamma}_+ \leq \chi$ when 
\[ 
\Biggl( \frac{\ds b + \frac{6}{\alpha} \log \bigl( c / e \bigr)}{n^{1/2}} 
+ \frac{6}{\alpha c n^{1/6}} \Biggr)^2 \leq \frac{3 \chi}{
(1 + \wh{\delta})^2 a}.
\] 
This condition is satisfied when 
\begin{align*}
\frac{\ds b + \frac{6}{\alpha} \log \bigl( c / e \bigr)}{ n^{1/2} } 
& \leq \frac{1}{2} \sqrt{ \frac{ 3 \chi}{(1 + \wh{\delta})^2 a}}, \\ 
\text{and } \frac{6}{\alpha c n^{1/6}} & \leq
\frac{1}{2} \sqrt{ \frac{ 3 \chi}{(1 + \wh{\delta})^2 a}}.
\end{align*}
This can be rewritten as 
\begin{align*}
n^{1/2} & \geq 2 \frac{(1 + \wh{\delta})}{(3 \chi)^{1/2}} 
\Bigl( b + 6 \log \bigl( c / e \bigr) / \alpha \Bigr) a^{1/2}, \\ 
\text{and } n^{1/2} & \geq \biggl( \frac{12}{\alpha c} \biggr)^3
\frac{(1 + \wh{\delta})^3}{\bigl( 3 \chi \bigr)^{3/2}} a^{3/2}.
\end{align*}
Due to the fact that $b \geq a$, the first condition implies the 
second one when 
\[ 
c^3 = \frac{1}{2} \biggl( \frac{12}{\alpha} \biggr)^3 \frac{
\bigl(1 + \wh{\delta} \bigr)^2}{ 3 \chi}, 
\] 
so that with this choice of constant $c$ the condition becomes
\[
n_{\epsilon} \geq 4 \frac{(1 + \wh{\delta})^2}{3 \chi} a 
\big[ b + 6 \log \bigl( c / e \bigr) / \alpha \bigr]^2. 
\]
Taking $\chi = 1/4$, $\alpha = 1/2$, $\wh{\delta} = 1/4$
and working out the constants numerically gives the condition 
on $n_{\epsilon}$ stated in the proposition. One can then check 
that the condition on $n_\epsilon$ 
necessary to obtain $\mu \leq 1/6$ and therefore
$\wh{\delta} \leq 1/4$ in Proposition \vref{prop:2.3} 
is weaker, so that we are entitled to take $\wh{\delta} \leq 1/4$, 
which gives $\wh{\gamma}_+ \leq \chi = 1/4$ 
and 
\[ 
\delta = \frac{ \wh{\delta} + \wh{\gamma}_+}{(1 - \wh{\delta}) 
( 1  - \wh{\gamma}_+)} \leq \frac{8}{9} < 1
\] 
as required. The last equation of the proposition is then 
obtained by substituting this value in equation \myeq{eq:19.2}. 

\section{Proof of Proposition \vref{prop:4.10}}
Consider the empirical Gram matrix
\[ 
\ov{G} = \frac{1}{n} \sum_{i=1}^n X_i X_i^{\top}. 
\] 
According to Proposition \vref{prop3.1}, 
under the same assumptions as in Proposition \vref{prop1.2.3}, 
there is an event $\Omega$ 
of probability at least $1 - \epsilon$ on which for any $\theta \in \B{R}^d$, 
\[ 
\frac{ \theta^{\top} \ov{G} \theta}{\theta^{\top} G \theta} 
\geq 1 - \frac{\wh{\delta} + \gamma_-}{1 + \wh{\delta}} = 
\frac{1 - \gamma_-}{1 + \wh{\delta}}.   
\] 
(Given that we use only one side of the inequalities involved 
in Proposition \ref{prop3.1} and that each side holds with 
probability at least $1 - \epsilon$, as is clear from the proof
of this proposition.) 
Remark that on $\Omega$, $\IM(G) \subset \IM( \ov{G} )$.
As $\IM(\ov{G}) \subset \IM(G)$ almost surely, we may remove 
a set of measure zero from $\Omega$ and assume that on 
$\Omega$, $\IM(G) = \IM(\ov{G})$. 
Since it does not change $R(\theta_*)$, we may assume without 
loss of generality that $\theta_* \in \IM(G)$, by projecting 
it on $\IM(G)$ if necessary. 
Remark then that for some $\xi \in \Ker(G) = \Ker(\ov{G})$, 
\begin{multline*}
\wh{\theta} = \ov{G}^{-1} \biggl( \frac{1}{n} \sum_{i=1}^n Y_i X_i \biggr) 
+ \xi = \ov{G}^{-1} \biggl( \frac{1}{n} \sum_{i=1}^n \bigl( Y_i 
- \langle \theta_*, X_i \rangle \bigr) X_i \biggr) 
\\ + \ov{G}^{-1}  
\biggl( \frac{1}{n} \sum_{i=1}^n \langle \theta_*, X_i \rangle X_i 
\biggr) + \xi = \ov{G}^{-1} W + \theta_* + \xi,   
\end{multline*}
where we have introduced the notation
\begin{equation}
\label{eq:21}
W = \frac{1}{n} \sum_{i=1}^n \bigl( Y_i - \langle \theta_*, X_i \rangle 
\bigr) X_i. 
\end{equation} 
Let us put $\ds \rho = \frac{1 + \wh{\delta}}{1 - \gamma_-}$ 
and let us remark that on the event $\Omega$, for any $\theta \in \IM(G)$, 
\begin{multline*}
\bigl\lVert \theta \bigr\rVert^2 = \bigl\lVert G^{1/2} G^{-1/2} \theta 
\bigr\rVert^2 
\leq \rho \, \bigl\lVert \ov{G}^{1/2} G^{-1/2} \theta \bigr\rVert^2 
\\ = \rho \, \bigl\lVert \bigl( G^{-1/2} \ov{G} G^{-1/2} \bigr)^{1/2} 
\theta \bigr\rVert^2  \leq \rho^2 \bigl\lVert  
G^{-1/2} \, \ov{G} \, G^{-1/2} \, \theta \bigr\rVert^2. 
\end{multline*}
Assuming that we are on the event $\Omega$, we can then write
\[ 
R(\wh{\theta}) - R(\theta_*) = \bigl\lVert G^{1/2} \bigl( \wh{\theta} 
- \theta_* \bigr) \bigr\rVert^2 = \bigl\lVert G^{1/2} \ov{G}^{-1} 
W \bigr\rVert^2 \leq \rho^2 \bigl\lVert G^{-1/2} W \bigr\rVert^2. 
\] 
Thus,
\begin{multline*}
\B{E} \Bigl[ \bigl( R(\wh{\theta}) - R(\theta_*) \bigr) \B{1}_{\Omega} 
\Bigr] \leq \rho^2 \B{E} \Bigl[ \bigl\lVert G^{-1/2} W \bigr\rVert^2 
\B{1}_{\Omega} \Bigr] \leq \rho^2 \B{E} \Bigl[ 
\bigl\lVert G^{-1/2} W \bigr\rVert^2 \Bigr] \\ = 
\frac{(1 + \wh{\delta})^2}{(1 - \gamma_-)^2 \, n} \B{E} \Bigl[ 
\bigl( Y - \langle \theta_*, X \rangle \bigr)^2 \bigl\lVert G^{-1/2} 
X \bigr\rVert^2 \Bigr].   
\end{multline*}
This proves the second statement of the proposition. The first statement 
is deduced from the second one, writing that
\begin{multline*}
\B{E} \Bigl[ \min \bigl\{ R(\wh{\theta} - R(\theta_*), M \bigr\} \Bigr] 
\leq \B{E} \Bigl[ M \B{1}_{\Omega^c} + 
\bigl[ R(\wh{\theta}) - R(\theta_*) \bigr] \B{1}_{\Omega} \Bigr] 
\\ \leq M \epsilon + \B{E} \Bigl[ \bigl( R(\wh{\theta}) 
- R(\theta_*) \bigr) \B{1}_{\Omega} \Bigr].  
\end{multline*}
Then if we take $\epsilon = C / (M n^2)$, the condition
$n \geq \bO \Bigl( \kappa \bigl[ d + \log(\epsilon^{-1}) \bigr] \Bigr)$
becomes $n \geq a \kappa \bigl[ d + \log \bigl( M n^2 / C \bigr) \bigr]$, 
for some constant $a > 0$. Consider the constant $b = 4 a \kappa$.
Remarking that $\log(n) \leq \log(b / e) + n / b$, 
we get that 
\begin{multline*}
a \kappa \bigl[ d + \log \bigl( M n^2 / C \bigr) \bigr] 
\leq 
a \kappa \bigl[ d + \log \bigl( M / C ) + 2 \log \bigl( b / e) + 2 n / b \bigr] 
\\ = a \kappa \bigl[ d + \log \bigl( M / C \bigr) + 2 \log \bigl( 4 a \kappa / e \bigr) \bigr] + n / 2,
\end{multline*}
so that the condition above is implied by the condition 
\begin{equation}
\label{eq:31}
n \geq 2 
a \kappa \bigl[ d + \log \bigl( M / C \bigr) + 2 \log 
\bigl( 4 a \kappa / e \bigr) \bigr] = \bO \Bigl( \kappa 
\bigl[ d + \log \bigl( \kappa M / C \bigr) \bigr] \Bigr).
\end{equation}
Let us now assume that 
$Y = \langle \theta_*, X \rangle + \eta$, where $\eta$ is independent 
from $X$ and such that $\B{E}(\eta) = 0$ and $\B{E} \bigl( \sigma^2 \bigr) 
= \sigma^2$. The event $\Omega$ described above is measurable 
with respect to the sigma-algebra generated by $(X_1, \dots, X_n)$. 
Reasoning as previously we see that on $\Omega$, 
\begin{multline*}
\B{E} \bigl[ R(\wh{\theta}) - R(\theta_*) \; \big| \; X_1, \dots, X_n \bigr] 
\leq \rho \, \B{E} \Bigl[ \bigl\lVert \ov{G}^{-1/2} W \bigr\rVert^2  \; \big| 
\; X_1, \dots, X_n \Bigr] \\ 
= \rho \frac{\rank(\ov{G})  \sigma^2}{n} \leq \frac{\rho d \sigma^2}{n}.  
\end{multline*}
The statement about $\B{E} \Bigl[ \min \bigl\{ R(\wh{\theta}) 
- R(\theta_*), M \bigr\} \Bigr]$ is then deduced as above.
\section{Proof of Proposition \vref{prop:4.12}}
From Proposition \vref{prop3.1}, we see that, 
when $n \geq n_{\epsilon}$ given by equation \vref{eq2.2}
on some event $\Omega$ of probability at least $1 - 2 \epsilon$, 
for any $\theta \in \B{R}^d$, 
\[ 
(1 - \wh{\delta})(1 - \wt{\gamma}_+) \lVert \ov{G}^{1/2} \theta \rVert^2 
\leq  \lVert G^{1/2} \theta \rVert^2 \leq \frac{1 + \wh{\delta}}{1 - \gamma_-} 
\lVert \ov{G}^{1/2} \theta \rVert^2.
\]  
Let us put for short $\ds \rho_+ = \frac{1 + \wh{\delta}}{1 - \gamma_-}$
and $\ds \rho_- = (1 - \wh{\delta})(1 - \wt{\gamma}_+)$.\\[2ex]
As already shown in a previous proof, the upper bound implies that 
$\IM(\ov{G}) = \IM(G)$. We can for this reason write 
that, on $\Omega$, 
\[ 
R(\wh{\theta}) - R(\theta_*) = \bigl\lVert G^{1/2} \bigl( \wh{\theta} 
- \theta_* \bigr) \bigr\rVert^2 = \bigl\lVert G^{1/2} 
\ov{G}^{-1} W \bigr\rVert^2,  
\] 
where the random variable $W$ is defined by equation \myeq{eq:21}. 
Moreover, for 
any $\theta \in \IM(G)$, 
\[ 
\lVert \theta \rVert^2 = \lVert G^{1/2} G^{-1/2} \theta \rVert^2
\geq \rho_- \, \lVert \ov{G}^{1/2} G^{-1/2} \theta \rVert^2 = 
\rho_- \, \bigl\lVert \bigl( G^{-1/2} \ov{G} G^{-1/2} \bigr)^{1/2} \theta \bigr\rVert^2. 
\] 
Iterate this inequality to get, for any $\theta \in \IM(G)$, 
\[ 
\lVert \theta \rVert^2 \geq \rho_-^2 \bigl\lVert G^{-1/2} \ov{G} G^{1/2} 
\theta \rVert^2.
\] 
Apply this to $\theta = G^{1/2} \ov{G}^{-1} W$, to get 
on the event $\Omega$ 
\[ 
R(\wh{\theta}) - R(\theta_*) \geq \rho_-^2 \bigl\lVert G^{-1/2} W \bigr\rVert^2.
\] 
We have also a reverse inequality on $\Omega$, that is proved as in Propsition 
\vref{prop:4.10} (where the event $\Omega$ was larger), and writes as 
\[ 
R(\wh{\theta}) - R(\theta_*) \leq \rho_+^2 \bigl\lVert G^{-1/2} W 
\bigr\rVert^2. 
\]  
\begin{multline*}
\text{Consequently }
\B{E} \Bigl( \bigl[ R(\wh{\theta}) - R(\theta_*) \bigr] \B{1}_{\Omega} \Bigr) 
\leq \rho_+^2 \B{E} \Bigl( \lVert G^{-1/2} W \bigr\rVert^2 \B{1}_{\Omega} 
\Bigr)
\\ \shoveright{\leq \rho_+^2 \B{E} \Bigl( \lVert G^{-1/2} 
W \bigr\rVert^2 \Bigr)} \\ 
\shoveleft{\text{and } \B{E} \Bigl( \bigl[ R(\wh{\theta}) - R(\theta_*) 
\bigr] \B{1}_{\Omega} 
\Bigr) 
\geq \rho_-^2 \B{E} \Bigl( \bigl\lVert G^{-1/2} 
W \bigr\rVert^2 \B{1}_{\Omega} \Bigr)} \\ 
= \rho_-^2 \Bigl[ \B{E} \Bigl( \bigl\lVert G^{-1/2} W \bigr\rVert^2 \Bigr) - 
\B{E} \Bigl( \bigl\lVert G^{-1/2} W \bigr\rVert^2 \B{1}_{\Omega^c} 
\Bigr) \Bigr]  \\ \geq \rho_-^2 \Bigl[ \B{E} \Bigl( \bigl\lVert G^{-1/2} 
W \bigr\rVert^2 \Bigr) - \B{E} \Bigl( \lVert G^{-1/2} W \bigr\rVert^4 
\Bigr)^{1/2} \B{P} \bigl( \Omega^c \bigr)^{1/2} \Bigr],
\end{multline*}
where we have used the Cauchy-Schwarz inequality. 
Let us put
\[ 
G^{-1/2} W = \frac{1}{n} \sum_{i=1}^n Z_i, 
\] 
where 
\[ 
Z_i = \bigl( Y_i - \langle \theta_*, X_i \rangle \bigr) G^{-1/2} X_i. 
\] 
Remark that $\B{E}(Z_i) = 0$, $\B{E} \bigl( \lVert Z_i \rVert^2 
\bigr) = C$, $\B{E} \bigl( \lVert Z_i \rVert^4 \bigr) = \kappa' C^2$ 
and that the random vectors $Z_i$ 
are i.i.d. 
Compute
\begin{multline*}
\B{E} \Bigl( \bigl\lVert G^{-1/2} W \bigr\rVert^2 \Bigr) = 
\B{E} \Bigl( \bigl\langle G^{-1/2} W, G^{-1/2} W \bigr\rangle \Bigr) =
\frac{1}{n^2} \sum_{i=1}^n \B{E} \Bigl( \lVert Z_i \rVert^2 \Bigr) 
= \frac{C}{n}, \\ 
\shoveleft{\B{E} \Bigl( \bigl\lVert G^{-1/2} W \bigr\rVert^4 \Bigr) = 
\B{E} \Bigl( \bigl\langle G^{-1/2} W, G^{-1/2} W \bigr\rangle^2 \Bigr)} \\ = 
\frac{1}{n^4} \B{E} \Biggr[ \biggl( \sum_{i=1}^n \lVert Z_i \rVert^2 
+ 2 \sum_{1 \leq i < j \leq n}  \langle Z_i, Z_j \rangle \biggr)^2 \Biggr]
\\ = \frac{1}{n^4} \B{E} \Biggl( \sum_{i=1}^n \lVert Z_i \rVert^4 
+ 2 \sum_{1 \leq i < j \leq n} \lVert Z_i \rVert^2 \lVert Z_j \rVert^2
+ 4 \sum_{1 \leq i < j \leq n} \langle Z_i, Z_j \rangle^2 \Biggr)
\\ \leq \frac{1}{n^4} \B{E} \Biggl( \sum_{i=1}^n \lVert Z_i \rVert^4 
+ 2 \sum_{1 \leq i < j \leq n} \lVert Z_i \rVert^2 \lVert Z_j \rVert^2
+ 4 \sum_{1 \leq i < j \leq n} \lVert Z_i \rVert^2  \lVert Z_j \rVert^2 \Biggr)
\\ = \frac{n \kappa' + 3 n (n-1) }{n^4} \, C^2 
= 3 \biggl( 1 + \frac{\kappa' - 3}{3n} 
\biggr) \frac{C^2}{n^2}. 
\end{multline*}
This proves that 
\begin{multline*}
\rho_-^2 \biggl[ 1 - \sqrt{6} \biggl( 1 + \frac{\kappa' - 3}{3 n} \biggr)^{1/2} 
\epsilon^{1/2} \biggr] \frac{C}{n} \leq \B{E} \Bigl( \bigl[ R(\wh{\theta}) - 
R(\theta_*) \bigr] \B{1}_{\Omega} \Bigr) \\ \leq \B{E} 
\Bigl( R(\wh{\theta}) - R(\theta_*) \, | \, \Omega \Bigr) 
\leq \frac{\rho_+^2 \, C }{n \, \B{P}(\Omega)}, 
\end{multline*}
as stated in equation \myeq{eq:22} of the proposition. 
Now 
\begin{multline*}
\B{E} \Bigl( \min \bigl\{ R(\wh{\theta}) - R(\theta_*), M \bigr\} \Bigr) 
\geq 
\B{E} \Bigl( \min \bigl\{ R(\wh{\theta}) - R(\theta_*), M \bigr\} \B{1}_{\Omega}\Bigr) 
\\ \geq \B{E} \Bigl( \min \bigl\{ \rho_-^2 \bigl\lVert G^{-1/2} W 
\bigr\rVert^2, M \bigr\} \B{1}_{\Omega} \Bigr)
\\ \geq \rho_-^2 \B{E} \Bigl( \min \bigl\{ \bigl\lVert 
G^{-1/2} W \bigr\rVert^2, M \bigr\} \B{1}_{\Omega} \Bigr)
\\ 
\geq \rho_-^2 \biggl[ \B{E} \Bigl( \min \bigl\{ \bigl\lVert G^{-1/2} W 
\bigr\rVert^2, M \bigr\} \Bigr) - \B{E} \Bigl( \bigl\lVert G^{-1/2} W 
\bigr\rVert^2 \B{1}_{\Omega^c} \Bigr) \biggr]
\\ \geq \rho_-^2 \Biggl[ 
\B{E} \Bigl( \bigl\lVert G^{-1/2} W \bigr\rVert^2 \Bigr) - 
\B{E} \Bigl[ \Bigl( \bigl\lVert G^{-1/2} W \bigr\rVert^2 - M \Bigr)_+ 
\Bigr] 
\\ - \B{E} \Bigl( \bigl\lVert G^{-1/2} W \bigr\rVert^4 \Bigr)^{1/2}  
\B{P} \bigl( \Omega^c \bigr)^{1/2} \Biggr]. 
\end{multline*}
Remark that for any $z \in \B{R}$, $(z - M)_+ \leq z^2 / (4 M)$. 
Thus 
\[ 
\B{E} \Bigl[ \Bigl( \bigl\lVert G^{-1/2} W \bigr\rVert^2 - M \Bigr)_+ \Bigr] 
\leq \frac{1}{4M} \B{E} \Bigl( \bigl\lVert G^{-1/2} W \bigr\rVert^4 
\Bigr) = \frac{3}{4M} \biggl( 1 + \frac{\kappa' - 3}{3n} \biggr) 
\frac{C^2}{n^2} 
\]  
and 
\begin{multline*}
\B{E} \Bigl( \min \bigl\{ R(\wh{\theta}) - R(\theta_*), M \bigr\} 
\Bigr) \\ \geq \rho_-^2 \Biggl[ \frac{C}{n} - \frac{3}{4 M} \biggl( 1 
+ \frac{\kappa' - 3}{3n} \biggr) \frac{C^2}{n^2} - \frac{\sqrt{6} C}{n}  
\biggl( 1 + \frac{\kappa' - 3}{3n} \biggr)^{1/2} \epsilon^{1/2} \Biggr],
\end{multline*}
that proves equation \eqref{eq:23} of the proposition. 
The end of the proposition is straightforward, the evaluation of 
$n_\epsilon$ in big $\bO$ notation when $\epsilon = n^{-(q-1)}$ 
being done using the same 
principle as in the proof of equation \myeq{eq:31}. 

\section{Obtaining a quadratic form}
\label{app:B} 

In this section, we will see how to deduce from the estimator 
of Proposition \vref{prop1.2.3}, that is not a quadratic form, 
a quadratic estimator, or equivalently an estimator 
of the Gram matrix $G = \B{E} \bigl( X X^{\top} \bigr)$
by a symmetric non-negative matrix $\wh{G}$. 
 
Let us assume that we derived as in Proposition \vref{prop1.2.3} 
an estimator $\wh{N}$ such that for some $\epsilon$ and $\delta \in ]0,1/2[$, 
on some event $\Omega'$ of probability at least $1 - \epsilon$, 
for any $\theta \in \B{R}^d$, 
\begin{equation}
\label{eq:B.13}
\biggl\lvert \frac{N(\theta)}{\wh{N}(\theta)} - 1 \biggr\rvert \leq \delta.
\end{equation}
Assume moreover that, as it is the case in Proposition \ref{prop1.2.3}, 
for any $\theta \in \B{R}^d$ such that $ \theta \perp \Span \bigl\{ X_1, 
\dots, X_n \bigr\}$, $\wh{N}(\theta) = 0$. 

Let us remark that equation \eqref{eq:B.13} implies that on $\Omega'$ 
\[ 
\Ker(G) = \bigl\{ \theta \in \B{R}^d \, : \,  N(\theta) = 0 \bigr\} 
= \bigl\{ \theta \in \B{R}^d \, : \,  \wh{N}(\theta) = 0 \bigr\}. 
\] 
Remark then that our second assumption implies that 
on $\Omega'$
\[ 
\Span \bigl\{ X_1, \dots, X_n \}^{\perp} \subset \Ker(G). 
\] 
Remark moreover that almost surely 
\[ 
\Ker(G)  \subset \Span \bigl\{ X_1, \dots, X_n \bigr\}^{\perp},  
\] 
since for any $\theta$ in a basis of $\Ker(G)$ (that is a finite set), $ \B{E} \bigl( 
\langle \theta, X \rangle^2 \bigr) = 0$, so that almost surely 
$\langle X_i , \theta \rangle^2 = 0$, $1 \leq i \leq n$. 

This proves that under our assumptions, on some event $\Omega''$ 
of the same probability as $\Omega'$, 
\[ 
\IM(G) = \Span \bigl\{ X_1, \dots, X_n \bigr\}, 
\] 
so that we have an easily computable estimator of $\IM(G)$ that 
is exact on $\Omega''$, an event of probability at least $1 - \epsilon$.  

Consider some positive parameter $\rho$, let  
$\B{S}_d = \bigl\{ \theta \in \B{R}^d \, : \, \lVert \theta \lVert = 1 \bigr\}$ 
be the unit sphere of $\B{R}^d$, and 
$\Theta_{\rho}$ some arbitrary $\rho$-net of 
$\Span \{ X_1, \dots, X_n \} \cap \B{S}_d$. 
In other words, assume that $\Theta_{\rho}$ is a finite subset 
of $\Span \{ X_1, \dots, X_n \} \cap \B{S}_d$ such that 
\[ 
\sup \, \Bigl\{ \;
\inf_{\xi \in \Theta_\rho} \lVert \theta - \xi 
\rVert \, : \, \theta \in \Span \bigl\{ X_1, \dots, X_n \bigr\} 
\cap \B{S}_d
\, \Bigr\} \leq \rho. 
\]  
Assume now that $\wh{G} \in \B{R}^{d \times d}$ is a random symmetric 
matrix solution of 
\begin{multline*}
\wh{G} = \arg \min_H \Bigl\{ \Tr(H^2) \, : \, H \in \B{R}^{d \times d}, 
H^{\top} = H, \\ \IM(H) \subset \Span \bigl\{ X_1, \dots, X_n \bigr\}, \\ \wh{N}(\theta) (1 - \delta) \leq \theta^{\top} H \, \theta 
\leq \wh{N}(\theta) (1 + \delta), \quad \theta \in \Theta_{\rho} \Bigr\}. 
\end{multline*}

This minimization problem has a solution on $\Omega''$, since in this 
case $G$ itself satisfies the constraints. We will see below in more 
detail that $\wh{G}$ is the solution of a convex minimization problem
very similar to the one appearing in the estimation of the parameters 
of a support vector machine using the popular box constraint 
learning algorithm.  
 
\begin{prop}
The symmetric matrix $\wh{G}$ is such that 
on the event $\Omega''$ of probability at least
$1 - \epsilon$, $\IM(\wh{G}) \subset \IM(G)$ 
and for any $\theta \in \B{R}^d$, 
\[ 
\Bigl\lvert \theta^{\top} \wh{G} \theta - N(\theta) \Bigr\rvert 
\leq \frac{2 \delta}{1 - \delta} N(\theta) + \frac{ 4 \rho \sqrt{\Tr(G^2)}}{
(1 - \delta)} \lVert \theta \rVert^2.
\] 
The positive part $\wh{G}_+$ of $\wh{G}$ is such that 
on the event $\Omega''$ of probability at least $1 - \epsilon$, 
for any $\theta \in \B{R}^d$, 
\[ 
\Bigl\lvert \theta^{\top} \wh{G}_+ \theta - N(\theta) \Bigr\rvert 
\leq \frac{2 \delta }{1 - \delta} N(\theta) + \frac{6 \rho \sqrt{\Tr(G^2)} }{(1 - \delta)} 
\lVert \theta \rVert^2. 
\] 
\end{prop}
\begin{proof}
During all this proof, we will assume that the event $\Omega''$ 
defined above is satisfied, so that the results will hold with 
probability at least $1 - \epsilon$. 

Let us also assume first that $\theta \in \IM(G) \cap \B{S}_d$. Recall that on 
$\Omega''$, $\IM(G) = \Span \{ X_1, \dots, X_n \}$, so that 
by construction of $\wh{G}$, $\IM(\wh{G}) \subset \IM(G)$. 

Since $\Theta_{\rho}$ is a $\rho$-net of $\IM(G) \cap \B{S}_d$,  
there is $\xi$ in $\Theta_{\rho}$ such that $\lVert \theta - \xi \rVert 
\leq \rho$. Consequently,
\[ 
\Bigl\lvert \theta^{\top} \wh{G} \theta - \xi^{\top} \wh{G} \xi \Bigr\rvert 
= 
 \Bigl\lvert (\theta + \xi)^{\top} \wh{G} (\theta - \xi) \Bigr\rvert 
\leq 2 \rho \, \lVert \wh{G} \rVert_{\infty}, 
\]  
where 
\[
\lVert \wh{G} \rVert_{\infty} = \sup \, \Bigl\{ \, 
\theta \in \B{S}_d \, : \, 
\lVert \wh{G} \theta \rVert \, \Bigr\}
\]
is the operator norm---and 
spectral radius, since $\wh{G}$ is symmetric---of $\wh{G}$. 
Moreover, 
\begin{multline*}
\Bigl\lvert \xi^{\top} \wh{G} \xi - N(\theta) \Bigr\rvert \leq 
\Bigl\lvert \xi^{\top} \wh{G} \xi - \wh{N}(\xi) \Bigr\rvert 
+ \Bigl\lvert \wh{N}(\xi) - N(\xi) \Bigr\rvert + \Bigl\lvert 
\xi^{\top} G \xi - \theta^{\top} G \theta \Bigr\rvert \\ \leq 2 \delta 
\wh{N}(\xi) + 2 \rho \lVert G \rVert_{\infty}   
\leq \frac{2 \delta}{1 - \delta} \, \xi^{\top} G \xi + 2 \rho \lVert G \rVert_{\infty} \\  
\leq \frac{2 \delta}{1 - \delta} \Bigl( \theta^{\top} G \theta  
+ 2 \rho \lVert G \rVert_{\infty} \Bigr) + 2 \rho \lVert G \rVert_{\infty} \\
= \frac{2 \delta}{1 - \delta} N(\theta) + \frac{2 (1 + \delta)}{1 - \delta} 
\rho \lVert G \rVert_{\infty}.
\end{multline*}
Remark now that $\lVert \wh{G} \rVert_{\infty} \leq \sqrt{\Tr(\wh{G}^2)} 
\leq \sqrt{\vphantom{\Tr(\wh{G}^2)}\Tr(G^2)}$, so that we can deduce from 
the two above inequalities that 
\[ 
\Bigl\lvert \theta^{\top} \wh{G} \theta - N(\theta) 
\Bigr\rvert \leq \frac{2 \delta}{1 - \delta} N(\theta) + \frac{4 \rho}{1 - \delta} \sqrt{\Tr(G^2)}, \qquad \theta \in \IM(G) \cap \B{S}_d. 
\] 
Recall that $\IM(\wh{G}_-) \subset \Ker(\wh{G}_+) \cap \IM(G)$, so that for any 
$\theta \in \IM(\wh{G}_-) \cap \B{S}_d$, there is $\xi \in \Theta_{\rho}$ 
such that $\lVert \theta - \xi \rVert \leq \rho$. Consequently 
\begin{multline*}
\theta^{\top} \wh{G}_- \theta = - \theta^{\top} \wh{G} \theta 
\leq - \xi^{\top} \wh{G} \xi + 2 \rho \lVert
\wh{G} \rVert_{\infty} \\ \leq - (1 - \delta) \wh{N}(\xi) 
+ 2 \rho \sqrt{\Tr(\wh{G}^2)} \leq 2 \rho \sqrt{\Tr(G^2)}.
\end{multline*}
This proves that $\lVert \wh{G}_- \rVert_{\infty} \leq 
2 \rho \sqrt{\Tr(G^2)}$. As a consequence, 
\begin{multline*}
\Bigl\lvert \theta^{\top} \wh{G}_+ \theta - N(\theta) \Bigr\rvert 
\leq \Bigl\lvert \theta^{\top} \wh{G}_+ \theta - \theta^{\top} 
\wh{G} \theta \Bigr\rvert + \Bigl\lvert \theta^{\top} \wh{G} \theta 
 - N(\theta) \Bigr\rvert \\ = \theta^{\top} \wh{G}_- \theta 
+ \Bigl\lvert \theta^{\top} \wh{G} \theta - N(\theta) \Bigr\rvert 
\leq \frac{2 \delta}{1 - \delta} N(\theta) + \frac{6 \rho}{1 - \delta} 
\sqrt{\Tr(G^2)}, \\ \theta \in \IM(G) \cap \B{S}_d.
\end{multline*}
By homogeneity, we get on the event $\Omega''$ that   
\begin{align*}
& \Bigl\lvert \theta^{\top} \wh{G} \theta - N(\theta) \Bigr\rvert \leq 
\frac{2 \delta}{1 - \delta} N(\theta) + \frac{4 \rho}{1 - \delta} \sqrt{
\Tr(G^2)} \lVert \theta \rVert^2, && \theta \in \IM(G), \\ 
& \Bigl\lvert \theta^{\top} \wh{G}_+ \theta - N(\theta) \Bigr\rvert \leq 
\frac{2 \delta}{1 - \delta} N(\theta) + \frac{6 \rho}{1 - \delta} \sqrt{
\Tr(G^2)} \lVert \theta \rVert^2, && \theta \in \IM(G).
\end{align*}
Let us now deal with the general case of an arbitrary $\theta \in \B{R}^d$. 
We can decompose it into $\theta = \theta_1 + \theta_2$, where $\theta_1 
\in \IM(G)$ and $\theta_2 \in \Ker(G)$. Since $\theta_1$ and $\theta_2$ 
are orthogonal, and since $\IM(\wh{G}_+) \subset \IM(\wh{G}) \subset 
\IM(G)$ on $\Omega''$, 
$\theta_2 \in \Ker(\wh{G}) \subset \Ker(\wh{G}_+)$, so that 
\[ 
\theta^{\top} \wh{G} \theta = \theta_1 \wh{G} 
\theta_1, \quad \theta^{\top} \wh{G}_+ \theta = \theta_1^{\top} 
\wh{G}_+ \theta_1, 
\text{ and } \theta^{\top} G \theta = \theta_1^{\top} G \theta_1.
\] 
Therefore, on the event $\Omega''$ of probability at least $1 - \epsilon$,  
\begin{multline*}
\Bigl\lvert \theta^{\top} \wh{G}_+ \theta - N(\theta) \Bigr\rvert 
= 
\Bigl\lvert \theta_1^{\top} \wh{G}_+ \theta_1 - N(\theta_1) \Bigr\rvert 
\\ \leq \frac{2 \delta}{1 - \delta} N(\theta_1) + \frac{6 \rho}{1 - \delta} 
\sqrt{\Tr(G^2)} \lVert \theta_1 \rVert^2
\\ = \frac{2 \delta}{1 - \delta} N(\theta) + \frac{6 \rho}{1 - \delta} 
\sqrt{\Tr(G^2)} \bigl( \lVert \theta \rVert^2 - \lVert \theta_2 \rVert^2 \bigr)
\\ \leq \frac{2 \delta}{1 - \delta} N(\theta) + \frac{6 \rho}{1 - \delta} 
\sqrt{\Tr(G^2)} \lVert \theta \rVert^2, \qquad \theta \in \B{R}^d, 
\end{multline*}
and, due to a similar chain of inequalities, 
\[ 
\Bigl\lvert \theta^{\top} \wh{G} \theta - N(\theta) \Bigr\rvert \leq 
\frac{2 \delta}{1 - \delta} N(\theta) + \frac{4 \rho}{(1 - \delta)} 
\sqrt{\Tr(G^2)} \lVert \theta \rVert^2, \qquad \theta \in \B{R}^d. 
\] 
\end{proof}
\begin{cor} 
\label{cor:B.2}
Introduce $\lambda_{\min} = \inf
\Bigl\{ N(\theta) \, : \,   
\theta \in \IM(G) \cap \B{S}_d \Bigr\}  > 0$, the smallest non zero 
eigenvalue of $G$. On the event $\Omega''$ of probability at least $1 - \epsilon$, for any $\theta \in \B{R}^d$, 
\[
\biggl\lvert \frac{ \theta^{\top} \wh{G} \theta}{ \theta^{\top} 
G \theta} - 1 \biggr\rvert \leq \frac{1}{1 - \delta} \biggl( 
2 \delta + \frac{4 \rho \sqrt{\Tr(G^2)}}{\lambda_{\min}} \biggr). 
\]  
As a consequence, when 
\[ 
\frac{1}{1 - \delta} \biggl( 
2 \delta + \frac{4 \rho \sqrt{\Tr(G^2)}}{\lambda_{\min}} \biggr) 
< 1, 
\] 
that is when
\[ 
\rho < \frac{(1 - 3 \delta) \lambda_{\min}}{4 \sqrt{\Tr(G^2)}},
\] 
$\wh{G}_+ = \wh{G}$ on $\Omega''$. 
Moreover, if we choose $\rho$ small enough, and more precisely 
such that  
\[ 
\rho \leq \frac{\delta^2 \wh{\lambda}_{\min}}{2 \sum_{i=1}^d \wh{N} (e_i)}, 
\] 
where $(e_1, \dots, e_d)$ is some arbitrary orthonormal basis of $\B{R}^d$, 
and 
\[ 
\wh{\lambda}_{\min} = \inf \Bigl\{ \wh{N}(\theta)\, : \, \theta \in \Span \{X_1, \dots, X_n 
\} \cap \B{S}_d\Bigr\},
\] 
then on $\Omega''$, and therefore with probability at least $1 - \epsilon$, 
\[ 
\Biggl\lvert \frac{\theta^{\top} \wh{G} \theta}{\theta^{\top} G \theta} 
- 1 \Biggr\rvert \leq \frac{ 2 \delta}{1 - 2 \delta}. 
\] 
\end{cor} 
\begin{proof}
To prove the first inequality of the corollary, consider any $\theta \in 
\B{R}^d$ and write it as $\theta = \theta_1 + \theta_2$, where 
$\theta_1 \in \IM(G)$ and $\theta_2 \in \Ker(G)$. 
If $\theta_1 = 0$, since we saw that on $\Omega''$, 
$\IM(\wh{G}) \subset \IM(G)$, then $\theta^{\top} 
\wh{G} \theta = 0$, so that with the convention used 
throughout this paper
\[ 
\frac{\theta^{\top} \wh{G} \theta}{\theta^{\top} G 
\theta} = \frac{0}{0} = 1.
\] 
Otherwise, $\theta^{\top} G \theta = \theta_1^{\top} G 
\theta_1 > 0$, and,
as seen in the previous proof, on the event $\Omega''$, 
\[ 
\frac{\theta^{\top} \wh{G} \theta}{\theta^{\top} G \theta}  
= 
\frac{\theta_1^{\top} \wh{G} \theta_1}{\theta_1^{\top} G \theta_1}.  
\] 
Moreover
\[ 
\lVert \theta_1 \rVert^2 \leq \lambda_{\min}^{-1}  N(\theta_1), 
\] 
so that on $\Omega''$, according to the previous proposition,  
\[ 
\Bigl\lvert \frac{\theta^{\top} \wh{G} \theta}{\theta^{\top} 
G \theta} - 1 \Bigr\rvert \leq \frac{2 \delta}{1 - \delta} + 
\frac{4 \rho \sqrt{\Tr(G^2)}}{1 - \delta} \times \frac{\lVert \theta_1 \rVert^2}{
N(\theta_1)} \leq \frac{1}{1-\delta} \biggl( 
2 \delta + \frac{ 4 \rho \sqrt{\Tr(G^2)}}{\lambda_{\min}} \biggr).
\]
To prove the end of the corollary, remark that on $\Omega''$ 
\[ 
\wh{\lambda}_{\min} (1 - \delta)
\leq \lambda_{\min} 
\leq \wh{\lambda}_{\min} (1 + \delta) 
\] 
since $\Span \{ X_1, \dots, X_n \} = \IM(G)$ on $\Omega''$
and since 
\[
\lambda_{\min} = \inf \Bigl\{ N(\theta) \, : \, \theta \in \IM(G) \cap 
\B{S}_d \Bigr\}. 
\]
Remark also that on the event $\Omega''$. for any 
orthonormal basis $(e_1, \dots, e_d)$ of $\B{R}^d$, 
\begin{multline*}
\sqrt{\Tr (G^2)} \leq \sqrt{\lVert G \rVert_{\infty} \Tr(G)} 
\leq \Tr(G) = \sum_{i=1}^d N(e_i) \\ \leq (1 + \delta)  
\sum_{i=1}^d \wh{N}(e_i) 
\leq \frac{(1 + \delta)}{(1 - \delta)} \Tr(G),
\end{multline*}
where $\lVert G \rVert_{\infty} = \sup_{\theta \in \B{S}_d} 
\lVert G \theta \rVert$ is the operator norm (and spectral radius) 
of $G$. 

Therefore on $\Omega''$, for any $\theta \in \B{R}^d$, 
\[ 
\biggl\lvert \frac{\theta^{\top} \wh{G} \theta}{\theta^{\top} G 
\theta} - 1 \biggr\rvert \leq \frac{1}{1 - \delta} \biggl( 
2 \delta + \frac{4 \rho \sum_{i=1}^d \wh{N}(e_i)(1 + \delta)}{\wh{\lambda}_{\min}
(1 - \delta)} \biggr).
\] 
When $\ds \rho \leq \frac{\delta^2 \wh{\lambda}_{\min}}{2 
\sum_{i=1}^d \wh{N}(e_i)}$, we obtain
\[
\biggl\lvert \frac{\theta^{\top} \wh{G} \theta}{\theta^{\top} G 
\theta} - 1 \biggr\rvert \leq \frac{1}{1 - \delta} \biggl( 
2 \delta + \frac{2 \delta^2 (1 + \delta)}{(1 - \delta)} \biggr)
= \frac{2 \delta ( 1 + \delta^2)}{1 + \delta^2 - 2 \delta} \leq 
\frac{2 \delta}{1 - 2 \delta}. 
\]
\end{proof}

\begin{prop}
\label{prop2.1}
The estimator $\wh{G}$ studied in the previous proposition 
can be expressed as
\[ 
\wh{G} = \sum_{\theta \in \Theta_{\rho}} 
\bigl[ \wh{\xi}_+(\theta) - \wh{\xi}_-(\theta) \bigr] \theta \theta^{\top}
\] 
where 
\begin{multline*}
\bigl[ \wh{\xi}_+(\theta), \wh{\xi}_-(\theta) \bigr]_{\theta \in \Theta_{\rho}}
\in \arg \max_{\ds (\xi_+, \xi_-) \in (\B{R}_+^2)^{\Theta_{\rho}}}  \\ - \frac{1}{2} \sum_{(\theta, \theta') \in \Theta_{\rho}^2} \bigl[ \xi_+(\theta) - \xi_-(\theta) \bigr] \bigl[ \xi_+(\theta') 
- \xi_-(\theta') \bigr] \langle \theta, \theta' 
\rangle^2 \\ + \sum_{\theta \in \Theta_{\rho} } 
\xi_+(\theta) \wh{N}(\theta) ( 1 - \delta) - \xi_-(\theta) \wh{N}(\theta) 
(1 + \delta).  
\end{multline*}
\end{prop}
{\sc Proof of Proposition \ref{prop2.1}.} 
Let us put 
\[
B_-(\theta) = \wh{N}(\theta)(1 - \delta) \quad \text{ and } \quad 
B_+(\theta) = \wh{N}(\theta)(1 + \delta). 
\] 
The estimated matrix $\wh{G}$ is solution of the minmax optimisation problem 
$$
V = \inf_{H, H^{\top} = H} 
\sup_{(\xi_+, \xi_-) \in \bigl( \B{R}_+^2 \bigr)^{\Theta}} V(H, \xi_+, \xi_-), 
$$
where 
$$ 
V(H, \xi_+, \xi_-) = \frac{1}{2} \Tr(H^2) + \sum_{\theta \in \Theta} \xi_+(\theta) 
\bigl[ B_-(\theta) - \theta^{\top} H \theta \bigr] + \xi_-(\theta) 
\bigl[ \theta^{\top} H \theta - B_+(\theta) \bigr].  
$$
On the event $\Omega''$ of probability at least $1 - \epsilon$ 
the constraints are satisfied when $H = G$, and therefore, since they are 
linear constraints, Slater's conditions are satisfied \cite[page 226]{Boyd}. This means that 
there is no duality gap, or in other words that 
$$
\inf_H \sup_{\xi_+, \xi_-} V(H, \xi_+, \xi_-) = \sup_{\xi_+, \xi_-} 
\inf_H V(H, \xi_+, \xi_-). 
$$ 
It is then elementary to compute explicitly the solution of 
$$
\inf_{H, H^{\top} = H} V(H, \xi_+, \xi_-),
$$ 
that is 
\[ 
\wh{H}(\xi_+, \xi_-) = \sum_{ \theta \in \Theta_{\rho}} \bigl[ 
\xi_+(\theta) - \xi_- (\theta) \bigr] \theta \theta^{\top}.
\] 
Remark that 
\begin{multline*}
V \bigl( \wh{H}(\xi_+, \xi_-), \xi_+,  \xi_-) 
\\ = - \frac{1}{2} \sum_{(\theta, \theta') \in \Theta_{\rho}^2} 
\bigl[ \xi_+(\theta) - \xi_-(\theta) \bigr] 
\bigl[ \xi_+(\theta') - \xi_-(\theta') \bigr] 
\langle \theta, \theta' \rangle^2 \\ 
+ \sum_{\theta \in \Theta_{\rho}} \xi_+(\theta) B_-(\theta) 
 - \xi_-(\theta) B_+(\theta).  
\end{multline*}
Since
$$
\sup_{\xi_+, \xi_-} \inf_H V(H, \xi_+, \xi_-) = 
\sup_{\xi_+, \xi_-} V \bigl( \wh{H}(\xi_+, \xi_-) , \xi_+, \xi_-
\bigr) = 
\inf_{H} \sup_{\xi_+, \xi_-} V(H, \xi_+, \xi_-),
$$
we deduce that $\wh{G} = \wh{H}(\wh{\xi}_+, \wh{\xi}_-)$, 
as stated in the proposition. $\square$

Let us remark that, since $B_-(\theta) < B_+(\theta)$, the constraints 
cannot be reached at the same time for $\xi_+(\theta)$ 
and for $\xi_-(\theta)$, so that either $\xi_+(\theta) = 0$ 
or $\xi_-(\theta) = 0$, implying that $\xi_+$ and $\xi_-$ are 
the positive and the negative parts of $\xi = \xi_+ - \xi_-$.
We could thus also write $\wh{G}$ as 
$\ds \wh{G} = \sum_{\theta \in \Theta} \wh{\xi} (\theta) \theta \theta^{\top}$,
where 
\begin{multline*} 
\wh{\xi} \in \arg \max_{\xi \in \B{R}^{\Theta}} \, \Biggl( 
- \frac{1}{2} \sum_{\theta, \theta'  \in \Theta} \xi(\theta) \langle 
\theta, \theta' \rangle^2 \xi(\theta') \\ + \sum_{\theta \in \Theta} 
\max \{ \xi(\theta), 0 \} B_-(\theta) - \min \{ \xi(\theta), 0 \} 
B_+(\theta) \Biggr). 
\end{multline*} 

\section{Computation of the robust Gram matrix estimator}

\subsection{Computation of the robust estimator in a fixed 
di\-rection} 
In this section, we give some details on the computation 
of the estimator $\wh{N}(\theta)$ used in Proposition 
\vref{prop1.2.3}. In this discussion, $\theta$ is fixed. 

Let us remark first that 
\begin{align*}
r_{\lambda}(0) & = - \lambda^{-1} \psi ( \lambda ) 
< 0, \text{ and }\\  
\ell_{\lambda} (\theta) \overset{\text{def}}{=} \lim_{\alpha \rightarrow + \infty} r_{\lambda}(\alpha \theta) & = 
\frac{ \Bigl\lvert \bigl\{ i; \langle \theta, 
X_i \rangle \neq 0 \bigr\} \Bigr\rvert \log(2) - \Bigl\lvert 
\bigl\{ i ; \langle \theta, X_i \rangle = 0 \bigr\} \Bigr\rvert \psi(
\lambda)  }{n \lambda},
\end{align*}
so that these two quantities can easily be computed.

In the case when $ \ell_{\lambda}(\lambda) 
\leq 0$, $\wh{\alpha} = + \infty$ and $\wh{N}(\theta) = 0$. 

Otherwise, and this is the most usual case, $\ell_{\lambda}(\theta) > 0$, 
and we can compute $\wh{\alpha}(\theta)$ and therefore 
$\wh{N}(\theta)$ quickly using a modified Newton's method 
with global convergence properties. 

Indeed, two algorithms may come to mind to compute $\wh{\alpha}$. 

The first is a divide and conquer algorithm, 
based on the fact that $\alpha \mapsto r_{\lambda} \bigl( \alpha \theta \bigr)$ 
is non-decreasing. Starting from $a_0, b_0 \in \B{R}_+$
such that $r_{\lambda} \bigl( \sqrt{a_0} \theta \bigr) \leq 0$ and 
$r_{\lambda} \bigl( \sqrt{b_0}
\theta \bigr) \geq 0$, 
we can put 
\begin{align*}
a_k & = \max \biggl\{ u \in \Bigl\{ a_{k-1}, \frac{a_{k-1}+b_{k-1}}{2}, b_{k-1} 
\Bigr\}, r \bigl( \sqrt{u} \theta \bigr) \leq 0 \biggr\}, \\
b_k & = \min \biggl\{ u \in \Bigl\{ a_{k-1}, \frac{a_{k-1}+b_{k-1}}{2}, b_{k-1} 
\Bigr\}, r \bigl( \sqrt{u} \theta \bigr) > 0 \biggr\}.
\end{align*}

The second algorithm is the well known Newton's method, which is
described in this case as 
$$
\alpha_k^2 = \alpha_{k-1}^2 - \frac{r_{\lambda} \bigl( \alpha_{k-1} \theta
\bigr) }{
\ds \int \langle \theta, x \rangle^2 \psi' \bigl[ \lambda \bigl( 
\alpha_{k-1}^2 \langle \theta, x \rangle^2 - 1 \bigr) \bigr] \, 
\, \ud \oB{P}(x)}.
$$

Once we have reached a small enough neighborhood of the solution, 
Newton's method 
is faster, whereas the divide and conquer algorithm achieves 
a convergence speed of $2^{-k}$, concerning the accuracy of the 
computation of $\wh{N}(\theta)^{-1}$, from any starting point,
not necessarily close to the solution. 

Fortunately, it is quite easy to combine the two methods into a single 
algorithm that keeps the best of both worlds.

Let us introduce the function 
$$
f(u) = u - \frac{r_{\lambda} \bigl( \sqrt{u} \theta \bigr)}{
\ds \int \langle \theta, x \rangle^2 \psi' \bigl[ \lambda \bigl( 
u \langle \theta, x \rangle^2 - 1 \bigr) \bigr] \, \ud \oB{P}(x)}
$$
Let us put to define the starting point of the algorithm
$$
A_0 = \Biggl\{ 0, \biggl( \int \langle \theta, x \rangle^2 \, \ud \oB{P}(x) 
\biggr)^{-1}, \max \biggl\{ \frac{1 + \lambda^{-1}}{\langle \theta, X_i \rangle^2}, \langle \theta, X_i \rangle \neq 0, 1 \leq i \leq n \biggr\} \Biggr\},
$$
and
\begin{align*}
a_0 & = \max \bigl\{ u \in A_0; 
r_{\lambda} \bigl( \sqrt{u} \theta \bigr) \leq 0 
\bigr\}, \\ 
b_0 & = \min \bigl\{ u \in A_0; 
r_{\lambda} \bigl( \sqrt{u} \theta \bigr) > 0 
\bigr\}. 
\end{align*}
Let us define then
$$
A_k = \Bigl\{ a_{k-1}, b_{k-1}, \frac{a_{k-1} + b_{k-1}}{2}, 
f(a_{k-1}), f(b_{k-1}) \Bigr\},
$$
\begin{align*}
a_k & = \max \bigl\{ u \in A_k; 
r_{\lambda} \bigl( \sqrt{u} \theta \bigr) \leq 0 
\bigr\}, \\ 
b_k & = \min \bigl\{ u \in A_k; 
r_{\lambda} \bigl( \sqrt{u} \theta \bigr) > 0 
\bigr\}, 
\end{align*}
and $\alpha_k = \sqrt{a_k}$ if $r_{\lambda} \bigl( \sqrt{(a_k + b_k)/2} \, \theta 
\bigr) > 0$ and $ \alpha_k = \sqrt{b_k}$ otherwise.  

\begin{prop}
$$
\lvert \alpha_k^2 - \wh{\alpha}^2 \rvert = 
\min \bigl\{ 
\lvert a_k - \wh{\alpha}^2 \rvert, \lvert b_k - \wh{\alpha}^2 \rvert \bigr\}
\leq 2^{-(k+1)} \lvert b_0 - a_0 \rvert,
$$
and 
$$
\lvert \alpha_k^2 - \wh{\alpha}^2 \rvert \leq 
\bigl\lvert f \bigl(\alpha_{k-1}^2 \bigr) - \wh{\alpha}^2 \bigr\rvert.  
$$
\end{prop}
The second inequality shows that the local convergence speed of the 
combined algorithm is at least as good as Newton's method. 

\subsection{Computation of a robust estimate of the Gram ma\-trix} 

We describe here a simplified algorithm, that does not 
share the mathematical properties of the convex optimization 
scheme described in Appendix \ref{app:B}, but turns out to 
be efficient in practice to improve on the empirical Gram 
matrix when dealing with the estimation of the Gram matrix, 
or the empirical risk minimization when dealing with least squares regression. 

Since our approach of robust least squares regression in dimension $d$ 
is based on the robust estimation of a Gram matrix in dimension $d+1$, 
we start with the robust estimation of the Gram 
matrix from a sample $X_1, \dots, X_n \in \B{R}^d$ made of independent 
copies of some random variable $X$. 

For any vector of weights $p = (p_i, i = 1, \dots, n) \in \B{R}^n$, 
and any positive parameter $\lambda$, let $S(p, \lambda)$ be the 
solution of 
\[ 
\sum_{i=1}^n \psi \Bigl[ \lambda \Bigl( S(p,\lambda)^{-1} p_i^2 
 - 1 \Bigr) \Bigr] = 0,
\]  
that can be computed as explained in the previous section, 
using some suitable Newton algorithm. 
Given some confidence parameter $\epsilon$, 
define $S(p) = S\bigl(p, \lambda(p)\bigr)$, where 
\begin{align*}
\lambda(p) & = m \sqrt{ \frac{1}{v} \Bigl[ \frac{2}{n} \log(\epsilon^{-1}) 
\Bigl(1 - \frac{2}{n} \log(\epsilon^{-1}) \Bigr) \Bigr]}, \\ 
\text{with } m & = \frac{1}{n} \sum_{i=1}^n p_i^2  
\quad \text{ and } \quad v = \frac{1}{n} \sum_{i=1}^n 
\bigl( p_i^2 - m \bigr)^2. 
\end{align*}
This value of the scale parameter is based on the optimal value for 
the estimation of a single expectation as described in \cite{Cat10} 
and maybe expected in practice to be more efficient than 
the conservative value allowing for the mathematical 
proof of generalization bounds. 

The algorithm to compute a robust estimate $\wh{G}$ of the 
Gram matrix $G = \B{E}(X X^{\top})$ works as follows.

Start with the empirical Gram matrix estimate
\[ 
\wh{G}_0 = \ov{G} = \frac{1}{n} \sum_{i=1}^n X_i \, X_i^{\top}.
\] 
Assuming that at iteration $k$ we have computed the 
estimate $\wh{G}(k) \in \B{R}^{d \times d}$, 
decompose it into 
\[ 
\wh{G}(k) = U(k)^{\top}  D(k) U(k)
\] 
where $U(k) \, U(k)^{\top} = I$ is an orthogonal matrix and $D(k)$ 
is a diagonal matrix. 
Define the $d \times d$ matrix
\begin{multline*}
M(k)_{i,j} = \frac{1}{4} \biggl[ 
S \Bigl( (U(k) X_{\ell})_i + (U(k) X_{\ell})_j, \ell=1, \dots, n \Bigr)
\\ - 
S \Bigl( (U(k) X_{\ell})_i - (U(k) X_{\ell})_j, \ell=1, \dots, n \Bigr)
\biggr], 
\end{multline*}
and update the estimate by the formula 
\[ 
\wh{G}(k+1) = U(k)^{\top}  M(k) \, U(k), 
\] 
until some stopping rule is reached (we can for instance use 
a fixed number of iterations, as in the simulation below, 
 or stop when the Frobenius norm 
$\lVert \wh{G}(k+1) - \wh{G}(k) \rVert_{\mathrm{F}}$ falls under some 
threshold). 
The idea is to update the estimator of the Gram matrix  
using the polarization formula 
\myeq{eq:1} in a basis of eigenvectors of the current 
estimate. This uses more directions than using the 
polarization formula only in the canonical basis 
of $\B{R}^d$, while trying to get accurate eigenvectors, 
and is faster than using a net of directions as 
in the mathematically more justified algorithm 
described in Appendix \ref{app:B}.  

To solve the least squares problem
\[ 
\inf_{\theta \in \B{R}^d} \B{E} \bigl[ \bigl( Y - \langle \theta, X \rangle \bigr)^2 \bigr] 
\] 
from a sample $(X_1, Y_1), \dots, (X_n, Y_n)$ of independent 
copies of $(X, Y) \in \B{R}^{d + 1}$, first compute as above 
a robust estimate 
\[ 
\wh{G} = \begin{pmatrix} \wh{G}_{1,1} & \wh{G}_{1,2} \\ 
\wh{G}_{2,1} & \wh{G}_{2,2} \end{pmatrix}  
\] 
of the Gram matrix 
\[ 
\B{E} \Biggl[ 
\begin{pmatrix} X \\ - Y \end{pmatrix} 
\bigl( X^{\top}, - Y \bigr) 
\Biggr], 
\] 
then define the robust estimate $\wh{\theta}$ of $\theta_*$ 
as 
\[ 
\wh{\theta} = - \wh{G}_{1,1}^{-1} \wh{G}_{1,2},  
\] 
where $\wh{G}_{1,1}^{-1}$ is the pseudo inverse of the symmetric 
matrix $\wh{G}_{1,1}$, obtained by inverting only its non-zero 
eigenvalues. 

\section{Some simulation} 

We present a small simulation to illustrate the benefit of 
using our robust least squares estimator in the case 
of a long tail noise. The idea of this simulation 
is to show that even in a very simple situation
using a robust estimator may bring a significant 
improvement. 

Consider some noise $\eta \sim 0.9 \times \C{N}(0, 1) + 0.1 
\times \C{N}(0,30^2)$ 
that is the mixture of two Gaussian random variables with 
different variances. Consider a Gaussian random variable $\wt{X} \sim 
\C{N}(0, 10^2)$ independent of the noise $\eta$, and define 
\[ 
Y = \wt{X} + \eta + 1.
\]  
Putting $X = (\wt{X}, 1)^{\top}$, we can write this problem as 
\[ 
Y = \langle \theta_*, X \rangle + \eta, 
\] 
where $\theta_* = (1,1)^{\top}$. 

Figure \vref{fig:1} shows a typical sample, where $n = 100$, 
and where the two components of the mixture have been plotted with 
different colors. 
\begin{figure}
\label{fig:1}
\begin{center}
\includegraphics[width=0.8\textwidth]{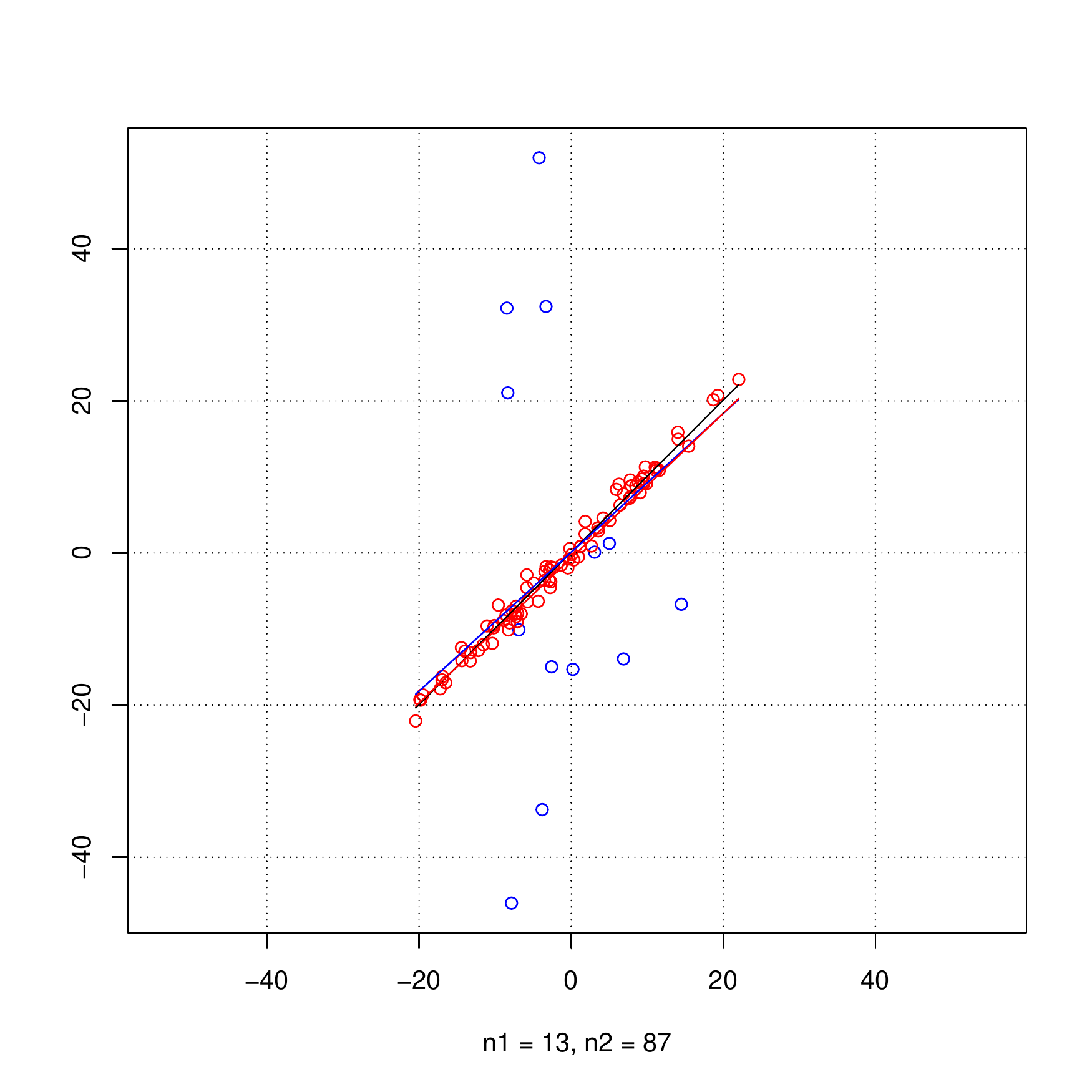}
\end{center}
\caption{A sample}
\end{figure}
Figures \vref{fig:2} plots the excess risk of the empirical risk minimizer and 
of the robust estimate for 500 trials of the experiment. 
One can see that we get a substantial improvement of the mean 
excess risk using the robust estimator, since the estimated 
expected excess risks are of 1.7 against less than 1.1 for the 
robust estimator.  

In conclusion, it is not necessary to envision very large sample
sizes or very exotic noise structures to feel the improvement 
brought by the more stable robust estimator.

\begin{figure}
\label{fig:2}
\begin{center} 
\includegraphics[width=0.8\textwidth]{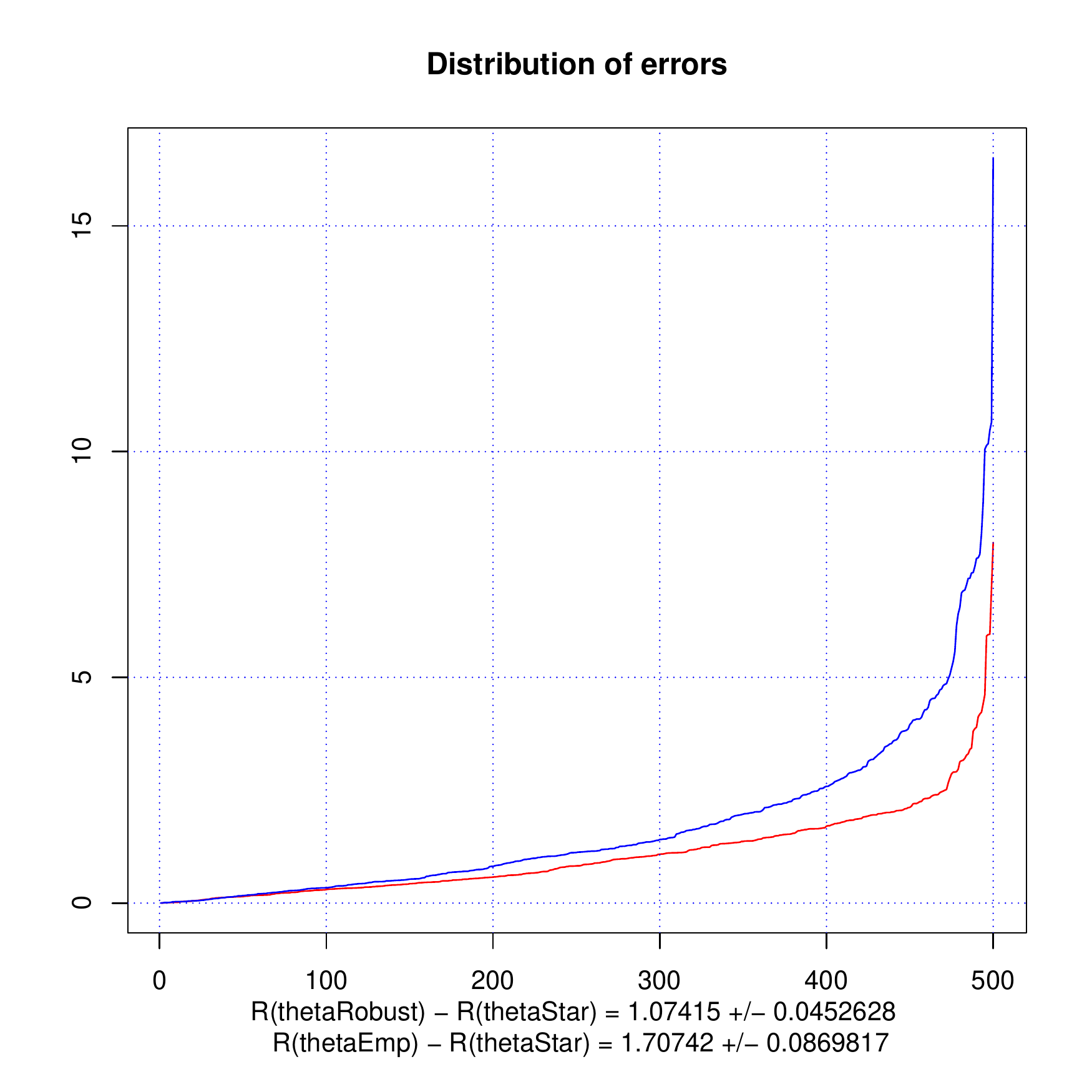}
\end{center}
\caption{The empirical quantile function of the excess risk}  
\end{figure}

\bibliographystyle{plain}
\nocite{Alq05,Alq08,Aud03a,AuCat10a,AuCat10c,AuCat10b,Cat10,Cat05,Cat01,McA99,McA01,See02,McA03} 
\bibliography{ref}

\end{document}